\documentclass[leqno,a4paper]{article}
\usepackage{amsmath,amsthm,amssymb}
\usepackage[utf8]{inputenc}
\usepackage[T1]{fontenc}

\usepackage{mathtools,enumerate}
\usepackage{bbm}
\usepackage{todonotes}
\usepackage{mathrsfs}
\usepackage{mathabx}
\usepackage{hyperref}
\usepackage{nicefrac}
\usepackage{tikz}
\usepackage{cleveref}
\usetikzlibrary{matrix}

\usepackage{url}
\makeatletter
\g@addto@macro{\UrlBreaks}{\UrlOrds}
\makeatother

\usepackage[sort,numbers]{natbib}
\providecommand{\noopsort}[1]{} 

\newtheorem{Th}{Theorem}[section]
\newtheorem{Prop}[Th]{Proposition}
\newtheorem{Lemma}[Th]{Lemma}

\theoremstyle{definition}
\newtheorem{Remark}[Th]{Remark}
\newtheorem{Def}{Definition}[section]
\newtheorem{Cor}[Th]{Corollary}
\newtheorem{Example}{Example}[section]
\newtheorem{Question}{Question}

\newcommand{\beq}{\begin{equation}}
\newcommand{\eeq}{\end{equation}}

\def\scalar(#1,#2){(#1\mid#2)}

\newcommand{\un}{\underline}
\newcommand{\raz}{\mathbbm{1}}
\newcommand{\1}{\mathbbm{1}}

\newcommand{\cs}{{\cal S}}

\newcommand{\cb}{{\cal B}}

\newcommand{\cd}{{\cal D}}

\newcommand{\cf}{{\cal F}}

\newcommand{\cm}{{\cal M}}
\newcommand{\co}{{\cal O}}

\newcommand{\xbm}{(X,{\cal B},\mu)}

\newcommand{\ot}{\otimes}
\newcommand{\ov}{\overline}

\newcommand{\Q}{\mathbb{Q}}
\newcommand{\R}{{\mathbb{R}}}
\newcommand{\T}{{\mathbb{T}}}
\newcommand{\C}{{\mathbb{C}}}
\newcommand{\Z}{{\mathbb{Z}}}
\newcommand{\N}{{\mathbb{N}}}

\newcommand{\vep}{\varepsilon}
\newcommand{\va}{\varphi}

\newcommand{\mob}{\boldsymbol{\mu}}
\newcommand{\lio}{\boldsymbol{\lambda}}

\newcommand{\tend}[3][]{\xrightarrow[#2\to#3]{#1}}
\newcommand{\tfs}{T_{\va,\cs}}

\newcommand{\bfu}{\boldsymbol{u}}
\newcommand{\bfv}{\boldsymbol{v}}
\newcommand{\bfe}{\boldsymbol{e}}
\newcommand{\bfd}{\boldsymbol{d}}
\newcommand{\bfsigma}{\boldsymbol{\sigma}}
\newcommand{\bfw}{\boldsymbol{w}}
\newcommand{\bdelta}{\boldsymbol{\delta}}
\newcommand{\sB}{\mathscr{B}}
\newcommand{\PP}{\mathbb{P}}
\newcommand{\xbmt}{(X,{\cal B},\mu,T)}

\newcommand{\EE}{{\mathbb{E}}}
\newcommand{\D}{{\mathbb{D}}}
\newcommand{\ab}{|}

\title{Sarnak's Conjecture -- what's new}
\author{S. Ferenczi\and J. Ku\l aga-Przymus\and M. Lema\'nczyk}
\date{}

\begin{document}
\bibliographystyle{amsplaininitials_nomrnumber}

\maketitle

\begin{abstract}An overview of last seven years results concerning Sarnak's conjecture on M\"obius disjointness is presented, focusing on ergodic theory aspects of the conjecture.
\end{abstract}

\tableofcontents

\section*{Introduction}

\paragraph{M\"obius disjointness}
Assume that $T$ is a continuous map\footnote{Most often, however not always, $T$ will be a homeomorphism.} of a compact metric space $X$. Following Peter Sarnak \cite{Sa,Sa:Af}, we will say  that $T$, or, more precisely, the topological dynamical system $(X,T)$  {\em is M\"obius disjoint} (or {\em M\"obius orthogonal})\footnote{$\mob$ stands for the  arithmetic M\"obius function, see next sections for explanations of notions that appear in Introduction.} if:
\begin{equation}\label{mdis1}
\lim_{N\to\infty}\frac1N\sum_{n\leq N} f(T^nx)\mob(n)=0\text{ for each $f\in C(X)$ and $x\in X$}.
\end{equation}
In 2010, Sarnak \cite{Sa,Sa:Af} formulated the following conjecture:\footnote{To be compared with
M\"obius Randomness Law by Iwaniec and Kowalski \cite{Iw-Ko}, page 338,
 that any ``reasonable'' sequence of complex numbers is orthogonal to $\mob$.}
\begin{equation}\label{sc1}
\parbox{0.85\textwidth}{Each zero entropy continuous map $T$ of a compact metric space $X$ is M\"obius disjoint.}
\end{equation}
Note that if $f$ is constant then convergence \eqref{mdis1} takes place in an arbitrary topological system $(X,T)$; indeed, $\frac1N\sum_{n\leq N}\mob(n)\to0$ is equivalent to the Prime Number Theorem (PNT), e.g.\ \cite{Hi,Te}. We can also interpret this statement as the equivalence of the PNT and the M\"obius disjointness of the one-point dynamical system. The Prime Number Theorem in arithmetic progressions (Dirichlet's theorem) can also be viewed similarly: it is equivalent to the M\"obius disjointness of the system $(X,T)$, where $Tx=x+1$ on $X=\Z/k\Z$ for each $k\geq1$. Note also that the classical Davenport's \cite{Da} estimate:
for each $A>0$, we have
\beq\label{vin}
\max_{t \in \mathbb{T}}\left|\displaystyle\sum_{n \leq N}e^{2\pi int}{\boldsymbol{\mu}}(n)\right|\leq C_A\frac{N}{\log^{A}N}\text{ for some }C_A>0\text{ and all }N\geq 2,\eeq
yields the M\"obius disjointness of irrational rotations.\footnote{\label{f:lindense}In order to establish M\"obius disjointness, we need to show convergence~\eqref{mdis1} (for all $x\in X$) only for a set of functions linearly dense in $C(X)$, so, for the rotations on the (additive) circle $\T=[0,1)$, we only need to consider characters. Note also that if the topological system $(X,T)$ is uniquely ergodic then we need to check~\eqref{mdis1} (for all $x\in X$) only for a subset of $C(X)$ which is linearly dense in $L^1$.

In what follows, for inequalities (as~\eqref{vin}), we will also use notation $\ll$ or ${\rm O}(\cdot)$, or $\ll_A$ or ${\rm O}_A(\cdot)$ if we need to emphasize a role of $A>0$.}

The present article is concentrated on an overview of research done during the last seven years\footnote{For a presentation of a part of it, see~\cite{dR:Ga}.} on Sarnak's conjecture~\eqref{sc1} from the ergodic theory point of view. It is also rather aimed at the readers with a good orientation in dynamics, especially in ergodic theory. It means that we assume that the reader is familiar with at least basics of ergodic theory, but often more than that is required, monographs \cite{Co-Fo-Si,Ei-Wa,Fu1,Gl,Wal} are among best sources to be consulted. In contrast to that, we included in the article a selection of some basics of analytic number theory. Those which appear here, in principle, are not contained in~\cite{Ri-this} and, as we hope, allow one for a better understanding of dynamical aspects of some number-theoretic results. We should however warn the reader that some number-theoretic results will be presented in their simplified (typically, non-quantitative) forms, sufficient for some ergodic interpretations but not putting across the whole complexity and depth of the results. In particular, this remark applies to recent break-through results of Matom\"aki and  Radziwi\l{}\l{}~\cite{Ma-Ra} and some related concerning a behavior of multiplicative functions on short intervals.\footnote{For a detailed account of these results, we refer the reader to \cite{Sau}.}

\paragraph{Ergodic theory viewpoint on Sarnak's conjecture} Sarnak's conjecture~\eqref{sc1} is formulated as a problem in topological dynamics. However, for each topological system $(X,T)$ the set $M(X,T)$ of (Borel, probability) $T$-invariant measures is non-empty and we can study dynamical properties of $(X,T)$ by looking at all measure-theoretic dynamical systems $(X,{\cal B},\mu,T)$ for $\mu\in M(X,T)$. Via the Variational Principle, Sarnak's conjecture can  be now formulated as M\"obius disjointness of the topological systems $(X,T)$ whose measure-theoretic systems $(X,{\cal B},\mu,T)$ for all $\mu\in M(X,T)$ have zero Kolmogorov-Sinai entropy. But one of main motivations for~\eqref{sc1} in \cite{Sa} was that this condition is weaker than a certain (open since 1965) pure number-theoretic result, known as the Chowla conjecture (see Section~\ref{s:sec31}).
Since the Chowla conjecture has its pure ergodic theory interpretation (Section~\ref{s:sec31}), the approach through invariant measures allows one to see the implication
$$
\mbox{ Chowla conjecture $\Rightarrow$ Sarnak's conjecture}\footnote{As proved by Tao \cite{Ta4}, the logarithmic averages version of the Chowla conjecture is equivalent to the logarithmic version of Sarnak's conjecture. We will see later in Section~\ref{czesc3} that once the logarithmic Chowla conjecture holds for the Liouville function $\lio$, we have that all configurations of $\pm1$s appear in $\lio$ (infinitely often).}
$$
as a consequence of some disjointness (in the sense of Furstenberg) results in ergodic theory. While the Chowla conjecture remains open, some recent break-through results in number theory find their natural interpretation as particular instances of the validity of Sarnak's conjecture. Samples of such results are (see Sections~\ref{s:interpr} and~\ref{s:frho}):
\begin{enumerate}
\item
The result of Matom\"aki, Radziwi\l\l \ and Tao \cite{Ma-Ra-Ta}: $$\sum_{h\leq H}\left|\sum_{m\leq M}\mob(m)\mob(m+h)\right|={\rm o}(HM)$$ (when $H,M\to \infty$, $H\ll M$) implies that each system $(X,T)$ for which all invariant measures yield measure-theoretic systems with discrete spectrum is M\"obius disjoint.\footnote{The same argument applied to the Liouville function $\lio$ implies that the subshift $X_{\lio}$ generated by $\lio$ is uncountable, see Section~\ref{czesc3}.}
\item
The result of Tao \cite{Ta3}:
$$\sum_{n\leq N}\frac{\mob(n)\mob(n+h)}n={\rm o}(\log N)$$ (when $N\to\infty$) for each $h\neq0$  implies that each system $(X,T)$ for which all invariant measures yield measure-theoretic systems with singular spectrum are logarithmically M\"obius disjoint.
\end{enumerate}
This is done by:
\begin{itemize}
\item
interpreting the number theoretic results as ergodic properties of the dynamical systems given by the invariant measures of the subshift $X_{\mob}$ for which $\mob$ is quasi-generic,
\item
using classical disjointness results in ergodic theory.
\end{itemize}
It is surprising and important that the ergodic theoretical methods of the last decades that led to new non-conventional ergodic theorems and showed a particular role of nil-systems, also appear in the context of Sarnak’s conjecture, and again the role of nil-systems seems to be decisive. Together with some new disjointness results in ergodic theory, it pushes forward significantly our understanding of M\"obius disjointness, at least on the level of logarithmic version of Sarnak's conjecture. The most spectacular achievement here is the recent result of Frantzikinakis and Host \cite{Fr-Ho1} (see Section~\ref{s:frho}) who proved that each zero entropy topological system $(X,T)$ with only countably many ergodic measures is logarithmically M\"obius disjoint.

The proofs reflect the ``local'' nature of all the aforementioned results. In other words, regardless the total entropy of the system, to obtain \eqref{mdis1} for a FIXED $x\in X$  (and all $f\in C(X)$), we only need to look at ergodic properties of the dynamical systems given by measures ``produced'' by $x$ itself (the limit points of the empiric measures given by $x$).
So, if all such measures yield zero entropy systems, the Chowla conjecture implies \eqref{mdis1} (for the fixed  $x$ and all $f\in C(X)$). When  all such measures yield systems with discrete spectrum / singular  spectrum / countably many ergodic components then
the relevant M\"obius disjointness holds (at $x$). Points with one of the listed properties may appear in $(X,T)$ having positive entropy. In fact, a positive entropy system can be M\"obius disjoint \cite{Do-Se}. To distinguish between zero and positive entropy systems it is natural to expect that in the zero entropy case the behavior of sums in \eqref{mdis1} is homogenous in $x$ (for a fixed $f\in C(X)$). Indeed, the uniform convergence (in $x\in X$, under the Chowla conjecture) of sums \eqref{mdis1} has been proved in \cite{Ab-Ku-Le-Ru1} (see Section~\ref{s:AOP}); in fact \eqref{sc1} is equivalent to Sarnak's conjecture in its uniform form and also in a uniform short interval form. Moreover, for the Liouville function, no positive entropy system satisfies \eqref{mdis1} in its uniform short interval form. The problem of uniform convergence turns out to be closely related to the general problem whether M\"obius disjointness is stable under our ergodic theory approach.  More precisely, suppose that the topological dynamical systems $(X,T)$ and $(X',T')$ are such that the dynamical systems obtained from invariant measures are the same for each of them (up to measure-theoretic isomorphism). Does the M\"obius disjointness of $(X,T)$ imply the M\"obius disjointness of $(X',T')$? Although the answer in general seems unknown, in case of uniquely ergodic models of the same measure-theoretic system a satisfactory (positive) answer can be given \cite{Ab-Ku-Le-Ru1}.


\paragraph{Content of the article}
We include the following topics:
\begin{itemize}
\item Sarnak's conjecture a.e., Sarnak's conjecture versus Prime Number Theorem in dynamics -- see Introduction and Section~\ref{czesc1}.
\item Brief introduction to multiplicative functions, Prime Number Theorem, K\'atai-Bourgain-Sarnak-Ziegler criterion -- see Section~\ref{czesc2}.
\item Results of Matom\"aki, Radziwi\l\l \ and Matom\"aki, Radziwi\l\l, Tao on multiplicative functions and some of their ergodic interpretations -- see Section~\ref{czesc3}.
\item Chowla conjecture, logarithmic Chowla and logarithmic Sarnak conjectures (Tao's results and Frantzikinakis and Host's results) -- see Section~\ref{czesc3}.
\item Frantzikinakis' theorem on some consequences of ergodicity of measures for which $\mob$ is quasi-generic -- see Section~\ref{czesc3}.
\item Ergodic criterion for Sarnak's conjecture -- the AOP and MOMO properties (uniform convergence in~\eqref{mdis1}), Sarnak's conjecture in topological models -- see Section~\ref{s:AOP}.
\item Glimpses of results on Sarnak's conjecture: systems of algebraic origin (horocycle flows, nilflows); systems of measure-theoretic origin (finite rank systems, distal systems), interval exchange transformations, systems of number-theoretic origin (automatic sequences and related) -- see Section~\ref{czesc5}.
\item Related research: $\mathscr{B}$-free systems, applications to ergodic Ramsey theory -- see Section~\ref{czesc6}.
\end{itemize}

\paragraph{Sarnak's conjecture a.e.}
Before we really get into the subject of Sarnak's conjecture, let us emphasize that this is the requirement ``for each $f\in C(X)$ and $x\in X$'' in~\eqref{mdis1} that makes Sarnak's conjecture deep and difficult to establish. As it has been already noticed in \cite{Sa}, the a.e.\ version of~\eqref{sc1} is always true regardless of the entropy assumption:

\begin{Prop}[\cite{Sa}]\label{p:aklt1}
Let $T$ be an automorphism of a standard Borel probability space $(X,{\mathcal{B}},\mu )$ and let $f \in L^1(X,\mathcal{B},\mu)$. Then, for a.e.\ $x \in X$, we have
\[
\frac1{N}\sum_{n\leq N}f(T^nx) {\boldsymbol{\mu}} (n) \tend{N}{\infty}0.
\]
\end{Prop}
For a complete proof, see \cite{Ab-Ku-Le-Ru}. The main ingredient is the Spectral Theorem which replaces
$
\Big\|\frac1{N}\sum_{n\leq N} f(T^nx){\boldsymbol{\mu}} (n)\Big\|_2$ by $\Big\|\frac1{N}\sum_{n\leq N} z^n {\boldsymbol{\mu}} (n)\Big\|_{L^2(\sigma_f)}\footnote{$\sigma_f$ stands for the spectral measure of $f$.}$,
together
with Davenport's estimate~\eqref{vin} (for $A=2$) which yields
$$
 \Big\|\frac1{N}\sum_{n\leq N} f(T^nx){\boldsymbol{\mu}} (n)\Big\|_2 \ll \frac{1}{{\log^{2}N}},\:N\geq2.
$$
The latter shows that, for $\rho>1$, the function $\sum_{k\geq1}\left|\frac1{\rho^k}\sum_{n\leq \rho^k} f(T^n\cdot)\mob(n)\right|$ is in $L^2(X,\mu)$ which, letting $\rho\to1$ allows one to conclude for $f\in L^\infty(X,\mu)$. The general case $f\in L^1(X,\mu)$ follows from the pointwise ergodic theorem.

As shown in \cite{EI1}, a use of Davenport's type estimate proved in \cite{Gr-Ta} for the nil-case, yields a polynomial version of Proposition~\ref{p:aklt1}. See also \cite{Cu-We} for the pointwise ergodic theorem for other arithmetic weights.

\section{From a PNT in dynamics to Sarnak's conjecture}\label{czesc1}
The content of this section can be viewed as a kind of motivation for Sarnak's conjecture (and is written on the base of Tao's post~\cite{Ta2} and Sarnak's lecture given at CIRM~\cite{Sa-CIRM}).

We denote by $\N:=\{1,2,\ldots\}$ the set of positive integers.
Given $N\in\N$, we let
$\pi(N):=\{p\leq N: p\in\PP\}$.
The classical Prime Number Theorem states that
\beq\label{pnt1}
\lim_{N\to\infty}\frac{\pi(N)}{N/\log N}=1.\eeq
We will always refer to this theorem as \underline{the} (classical) PNT.

Assume that $T$ is a continuous map of a compact metric space $X$.  Assume moreover that $(X,T)$ is uniquely ergodic, that is, the set $M(X,T)$ of $T$-invariant probability Borel measures is reduced to one measure, say $\mu$. By unique ergodicity, the ergodic averages go to zero (even uniformly) for zero mean continuous functions:
$$
\frac1N\sum_{n\leq N}f(T^nx)\tend{N}{\infty} 0$$
for each $f\in C(X)$, $\int_Xf\,d\mu=0$, and $x\in X$.
Hence, the statement that {\em a PNT holds in} $(X,T)$ ``should'' mean
\beq\label{meaningPNT}
\lim_{N\to\infty}\frac1{\pi(N)}\sum_{\PP\ni p\leq N}f(T^px)=0\eeq
for all zero mean $f\in C(X)$ and $x\in X$ (in what follows, instead of $\sum_{\PP\ni p}$, we write simply $\sum_{p}$ if no confusion arises).\footnote{We recall that Bourgain in \cite{Bo10,Bo12,Bo13},
proved that for each $\alpha\geq(1+\sqrt3)/2$, each automorphism $T$ of a probability standard Borel space $\xbm$ and each $f\in L^\alpha\xbm$ the sums in~\eqref{meaningPNT} converge for a.e.~$x\in X$. The result has been extended by Wierdl in \cite{Wi1} for all $\alpha>1$.} Let us see how to arrive at~\eqref{meaningPNT} differently.

Recall that the von Mangoldt function $\boldsymbol{\Lambda}$ is defined by $\boldsymbol{\Lambda}(n)=\log p$ if $n=p^k$ for  a prime number $p$ (and $k\geq1$) and $\boldsymbol{\Lambda}(n)=0$ otherwise. Contrary to most of arithmetic functions considered in this article, $\boldsymbol{\Lambda}$ is not multiplicative. It is not bounded either and its support has zero density. The (classical) PNT is equivalent to
$$
\frac1N\sum_{n\leq N}\boldsymbol{\Lambda}(n)\tend{N}{\infty} 1.
$$
A given sequence $(a_n)\subset\C$ can be said {\em to satisfy \underline{a}\ PNT}  whenever we can give an asymptotic estimate on
$$
\sum_{n\leq N}a_n\boldsymbol{\Lambda}(n)
$$
when $N$ tends to infinity; thus the classical PNT is a PNT for the sequence $a_n=1$. In particular, a sequence $(a_n)$  also satisfies a PNT  if
\begin{equation}\label{APNT}
\sum_{n\leq N}a_n\boldsymbol{\Lambda}(n)=\sum_{n\leq N}a_n+{\rm o}(N),
\end{equation}
and, if additionally $(a_n)$ has zero mean, i.e.\ if $\frac1N\sum_{n\leq N}a_n\tend{N}{\infty} 0$, then $(a_n)$ satisfies a PNT if
\begin{equation}\label{apnt}
\frac1N\sum_{n\leq N}a_n\boldsymbol{\Lambda}(n)\tend{N}{\infty} 0.
\end{equation}
An interesting special case is $a_n=(-1)^n$, which has zero mean. Here, we do have estimates of the sums of $\boldsymbol{\Lambda}(n)$ over the odd numbers smaller than $N$, but they  are of the order of $N$, thus~\eqref{apnt} is not satisfied. Beyond this point, we will not study such particular cases and we shall always write that {\em the sequence $(a_n)$ satisfies \underline{a}\ PNT} whenever~\eqref{APNT} holds.

Zero mean sequences are easily ``produced'' in uniquely ergodic systems. We will say that a {\em uniquely ergodic topological dynamical system $(X,T)$ satisfies \underline{a} PNT} if
\beq\label{pnt1a}
\frac1N\sum_{n\leq N}f(T^nx)\boldsymbol{\Lambda}(n)\tend{N}{\infty} 0\eeq
for all zero mean $f\in C(X)$ and $x\in X$.
We have
$$
\frac1N\sum_{n\leq N}f(T^nx)\boldsymbol{\Lambda}(n)=\frac1N\sum_{p\leq N}f(T^px)\log p+ \frac1N\sum_{p^k\leq N, k\geq 2}f(T^{p^k}x)\log p.$$
Now, in the second sum if $p^k\leq N$ then $p\in[1,\sqrt N]$; the largest value of $\log p$ is bounded by $\frac12\log N$, therefore, the second sum is of order
${\rm O}(\sqrt N\cdot \log N/N)$, hence of order $N^{-\frac12+\vep}$ for each $\vep>0$. Thus, a PNT in $(X,T)$ means that
\beq\label{pnt2}
\frac1N\sum_{p\leq N}f(T^px)\log p\tend{N}{\infty} 0\eeq
for all zero mean $f\in C(X)$ and $x\in X$.
Note that by the classical PNT to prove~\eqref{pnt2}, we need to show it for a linearly dense set of functions.\footnote{Indeed, we have
\begin{multline*}
\left| \frac1N\sum_{p\leq N}f(T^px)\log p - \frac1N\sum_{p\leq N}g(T^px)\log p\right|\\
\leq \frac1N\sum_{p\leq N}|f(T^px)-g(T^px)|\log p\leq\|f-g\| \frac1N\sum_{p\leq N}\log p={\rm O}(\|f-g\|),
\end{multline*}
as condition $\frac1N\sum_{n\leq N}\boldsymbol{\Lambda}(n)\tend{N}{\infty} 1$   is equivalent to $\frac1N\sum_{p\leq N}\log p\tend{N}{\infty} 1$.}

Let us now write
$$
\frac1N\sum_{p\leq N}f(T^px)\log p=\frac1N\sum_{p\leq N/\log N}f(T^px)\log p +\frac1N\sum_{N/\log N\leq p\leq N}f(T^px)\log p.$$
We have $\frac1N\sum_{p\leq N/\log N}f(T^px)\log p={\rm O}(1/\log N)$ (by $\frac1M\sum_{p\leq M}\log p\to 1$ when $M\to\infty$). Moreover, write $f=f_+-f_-$ and then we have
\begin{multline*}
\frac{\log N-\log\log N}N\sum_{N/\log N\leq p\leq N}f_+(T^px)\\
 \leq \frac1N\sum_{N/\log N\leq p\leq N}f_+(T^px)\log p
\leq\frac{\log N}N\sum_{N/\log N\leq p\leq N}f_+(T^px)
\end{multline*}
as  $\log N-\log\log N\leq \log p\leq \log N$ for the $p$ in the considered interval. Now, $\pi(N)/(N/(\log N-\log\log N))\tend{N}{\infty}1$ and $\pi(N)/(N/\log N)\tend{N}{\infty}1$, whence
$$
\left|\frac1N\sum_{p\leq N}f_+(T^px)\log p-\frac1{\pi(N)}\sum_{p\leq N}f_+(T^px)\right|\tend{N}{\infty} 0.$$
Repeating the same reasoning with $f_+$ replaced by $f_-$
and by~\eqref{pnt2}, we obtain that the statement {\em a PNT holds in $(X,T)$} is equivalent to~\eqref{meaningPNT}
for all zero mean $f\in C(X)$ and $x\in X$.

\begin{Remark}
By replacing $\boldsymbol{\Lambda}$ in~\eqref{pnt1a} by $\mob$, we come back to Sarnak's conjecture.  The identity $\boldsymbol{\Lambda}=\mob\ast\log$ (see~\eqref{dconv} below), i.e.\
$\boldsymbol{\Lambda}(n)=\sum_{d\divides n}\mob(d)\log(n/d)=
-\sum_{d\divides n}\mob(d)\log d$ suggests some
other connections between the simultaneous validity of a PNT and M\"obius disjointness  in $(X,T)$  but no rigorous theorem toward a formal equivalence of the two conditions has been proved. Actually, such an equivalence taken literally does not hold. Indeed, the fact that the support of $\boldsymbol{\Lambda}$ is of zero upper Banach density makes a PNT vulnerable under zero density replacements of the observable $(f(T^nx))$. On the other hand, M\"obius orthogonality is stable under such replacements. We illustrate this using the following simple example.

Consider the classical case $a_n=1$ for all $n\in\N$. This is the same as to consider a PNT in a uniquely ergodic model\footnote{We recall that if $(Z,\cd,\kappa,R)$ is a measure-preserving system then by its {\em uniquely ergodic model} we mean  a uniquely ergodic system $(X,T)$ with the unique (Borel) $T$-invariant measure $\mu$
such that $(Z,\cd,\kappa,R)$ is measure-theoretically isomorphic to $(X,\cb(X),\mu,T)$.}  of the one-point system. One can now ask if we have a PNT in all uniquely ergodic models
of the one-point system (it is an exercise to prove that all such models are M\"obius disjoint).
Take any sequence  $(c_{p^k})_{p^k}\in\{-1,1\}^{\N}$ and define $b_n$ as $a_n$ when $n\neq p^k$ and   $b_{p^k}=c_{p^k}$.  We can see that
     $$\frac1N\sum_{n\leq N}b_n\boldsymbol{\Lambda}(n)=\frac1N\sum_{p^k\leq N}c_{p^k}\log p.$$
Now, the subshift $X_b\subset\{-1,1\}^{\N}$ generated by $b$ (cf.~\eqref{definicjaXx}) has only one invariant measure $\delta_{11\ldots}$, so it is a uniquely ergodic model of the one-point system and if we take
$f(z)=1-z(1)$ ($z\in X_b$) as our continuous function, we can see that $f$ has zero mean but neither \eqref{pnt1a} nor \eqref{meaningPNT} are satisfied if the sequence $c$ is badly behaving.
It follows that we can expect a PNT to hold only in some classes of ``natural'' dynamical systems, samples of which we will see in Section~\ref{czesc5}.

Returning to our discussion on a PNT, in any such situation, given a bounded sequence $(f(n))\subset\C$, we can write
$$
\sum_{n\leq N}f(n)\boldsymbol{\Lambda}(n)=-\sum_{n\leq N}f(n)\sum_{d\divides n}\mob(d)\log(d)
=-\sum_{d\leq N}\mob(d)\log d\sum_{e\leq N/d}f(ed).
$$
Then a further decomposition of the second sum into a structured part and a remainder leads to two sums and allows one for an application of M\"obius Randomness Law to the second sum in order to predict the correct main term value of $\sum_{n\leq N}f(n)\boldsymbol{\Lambda}(n)$, see~\cite{Sa-CIRM}.
\end{Remark}

\section{Multiplicative functions}\label{czesc2}
\subsection{Definition and examples}
 An arithmetic function $\bfu\colon\N\to\C$ is called {\em multiplicative} if $\bfu(1)=1$ and $\bfu(mn)=\bfu(m)\bfu(n)$ whenever $(m,n)=1$. If $\bfu(mn)=\bfu(m)\bfu(n)$ without the coprimeness restriction on $m,n$, then $\bfu$ is called {\em completely multiplicative}. Clearly, each multiplicative function is entirely determined by its values at $p^{\alpha}$, where $p\in\PP$ is a prime number and $\alpha\in\N$ (for completely multiplicative functions $\alpha=1$). A prominent example of a multiplicative function is the M\"obius function $\mob$ determined by $\mob(p)=-1$ and $\mob(p^\alpha)=0$ for $\alpha\geq2$. Note that $\mob^2$ (which is obviously also multiplicative) is the characteristic function of the set of square-free numbers. The Liouville function $\lio\colon\N\to\C$ is completely multiplicative and is given by $\lio(p)=-1$. Clearly, $\mob=\lio\cdot \mob^2$ and we will see soon some more relations between $\mob$ and $\lio$. Many other classical arithmetic functions are multiplicative, for example:  the Euler function $\boldsymbol{\phi}$; the function $n\mapsto (-1)^{n+1}$ is a periodic multiplicative function which is not completely multiplicative; $\bfd(n):=$number of divisors of~$n$, $n\mapsto 2^{\boldsymbol{\omega}(n)}$, where $\boldsymbol{\omega}(n)$ stands for the number of different prime divisors of $n$; $\bfsigma(n)=\sum_{d\divides n}d$. Recall that given $q\geq1$, a function $\chi\colon\N\to\C$ is called a {\em Dirichlet character of modulus}~$q$ if:
\begin{enumerate}[(i)]
\item
$\chi$ is $q$-periodic and completely multiplicative,
\item
$\chi(n)\neq0$ if and only if $(n,q)=1$.
\end{enumerate}
It is not hard to see that Dirichlet characters are determined by the ordinary characters of the multiplicative group (of order $\boldsymbol{\phi}(q)$) $\left(\Z/q\Z\right)^\ast$ of invertible (under multiplication) elements in $\Z/q\Z$. The Dirichlet character $\chi_1(n):=1$  iff $(n,q)=1$ is called the {\em principal character} of modulus $q$. Moreover, each periodic, completely multiplicative function is a Dirichlet character (of a certain modulus). Another class of important (completely) multiplicative functions is given by Archimedean characters $n\mapsto n^{it}=e^{it\log n}$ which are indexed by $t\in\R$.

\subsection{Dirichlet convolution, Euler's product} Recall that given two arithmetic functions $\bfu,\bfv\colon\N\to\C$, by their {\em Dirichlet convolution} $\bfu\ast\bfv$ we mean the arithmetic function
\beq\label{dconv}
\bfu\ast\bfv(n):=\sum_{d\divides n}\bfu(d)\bfv(n/d),\;n\in\N.\eeq
If by $A$ we denote the set of arithmetic functions then $(A,+,\ast)$ is a ring which is an integral domain and the unit $\bfe\in A$ is given by $\raz_{\{1\}}$.\footnote{The M\"obius Inversion Formula is given by $\mob\ast\raz_\N=\bfe$.} There is a natural ring isomorphism between $A$ and the ring $D$ of (formal)\footnote{We will not discuss here the problem of convergence of Dirichlet series, see~\cite{Ri-this}.} Dirichlet series
$$A\ni \bfu\mapsto U(s):=\sum_{n=1}^\infty \frac{\bfu(n)}{n^s}\in D,\ s\in\C,
$$
under which
$$
U(s)V(s)=\sum_{n=1}^\infty \frac{\bfu\ast\bfv(n)}{n^s}.$$

When $\bfu=\raz_{\N}$ then the Dirichlet series defines the Riemann $\zeta$ function:\footnote{An
analytic continuation of $\zeta$ yields a meromorphic function on $\C$ (with one pole  at $s=1$) satisfying the functional equation
\beq\label{funeq1}
\zeta(s)=2^s\pi^{s-1}\sin\left(\frac{\pi s}2\right)\Gamma(1-s)\zeta(1-s).\eeq
Because of the sine,  $\zeta(-2k)=0$ for all integers $k\geq 1$ -- these are so called trivial zeros  of $\zeta$  ($\zeta(2k)\neq0$ since $\Gamma$ has simple poles at $0,-1,-2,\ldots$).
In ${\rm Re}\, s>1$ there are no zeros of $\zeta$ ($\zeta$ is represented by a convergent infinite product), so except of $-2k$, $k\geq1$, there are no zeros for $s\in\C$, ${\rm Re}\,s<0$ (as ${\rm Re}(1-s)>1$). The
Riemann Hypothesis asserts that
all nontrivial zeros of $\zeta$ are on the line $x=\frac12$. See~\cite{Ri-this}.}
$$\zeta(s)=\sum_{n=1}^\infty\frac1{n^s}\text{ for }{\rm Re}\,s>1.$$

It is classical that if $\bfu$ and $\bfv$ are multiplicative then so is their Dirichlet convolution. The importance of multiplicativity can be seen in the representation of the Dirichlet series of a multiplicative function $\bfu$ as an Euler's product. Indeed,
a general term of
 $\prod_{p\in\PP}(1+\bfu(p)p^{-s}+\bfu(p^2)p^{-2s}+\ldots)$ has the form $\frac{\bfu(p_{i_1}^{\alpha_1})\cdot\ldots\cdot \bfu(p_{i_r}^{\alpha_r})}
 {(p_{i_1}^{\alpha_1}\cdot\ldots\cdot p_{i_r}^{\alpha_r})^s}=\frac{\bfu(p_{i_1}^{\alpha_1}\cdot\ldots\cdot p_{i_r}^{\alpha_r})}{(p_{i_1}^{\alpha_1}\cdot\ldots\cdot p_{i_r}^{\alpha_r})^s}$, i.e.\ equals $\frac{\bfu(n)}{n^s}$ for some $n$.
It easily follows that
$$
 \sum_{n\geq1}\frac{\bfu(n)}{n^s}
=\prod_{p\in \PP}(1+\bfu(p)p^{-s}+\bfu(p^2)p^{-2s}+\ldots).$$
If additionally $\bfu$ is completely multiplicative (and $|\bfu|\leq1$), then $\bfu(p^k)=\bfu(p)^k$ and
$$\sum_{n=1}^\infty\frac{\bfu(n)}{n^s}=\prod_{p\in \PP}(1-\bfu(p)p^{-s})^{-1}.$$
Note that
if $\bfu=\mob$, we obtain
$$
\sum_{n\geq1}\frac{\mob(n)}{n^s}=\prod_{p\in \PP}(1-p^{-s})
$$
since $\mob(p)=-1$ and $\mob(p^r)=0$ whenever $r\geq2$.
Since for the Riemann $\zeta$ function, we have
$\zeta(s)=\prod_{p\in\PP}(1-p^{-s})^{-1}$ for ${\rm Re}\,s>1$, we obtain the following.

\begin{Cor}\label{odwrot}
We have
$\frac1{\zeta(s)}=\sum_{n\geq1}\frac{\mob(n)}{n^s}$ whenever ${\rm Re}\,s>1$.\end{Cor}
We could have derived the above assertion in a different way. Indeed, $\mob\ast \raz_{\N}=\bfe$.
If $G(s):=\sum_{n=1}^\infty\frac{\mob(n)}{n^s}$ stands for  the Dirichlet series of the  M\"obius function, then $$G(s)\cdot\zeta(s)=\sum_{n=1}^\infty
\frac{(\mob\ast\raz_{\N})(n)}{n^s}=
    \sum_{n=1}^\infty\frac{\bfe(n)}{n^s}=1.
$$

\subsection{Distance between multiplicative functions}
Denote by
\beq\label{przM}
\cm:=\{\bfu\colon\N\to\C : \bfu\text{ is multiplicative and }|\bfu|\leq1\}.
\eeq
Let $\bfu,\bfv\in\cm$. Define the ``distance'' function $D$ on $\cm$ by setting
\beq\label{distance1}
D(\bfu,\bfv):=\left(\sum_{p\in\PP}\frac1p\left(1-{\rm Re}\left(\bfu(p)\ov{\bfv(p)}\right)\right)\right)^{1/2}.
\eeq
For each $\bfu,\bfv,\bfw\in\cm$, we have:
\begin{itemize}
\item $D(\bfu,\bfu)\geq0$; $D(\bfu,\bfu)=0$ iff $\sum_{p\in\PP}\frac1p(1-|\bfu(p)|^2)=0$ iff $|\bfu(p)|=1$ for all $p\in\PP$, so  $D(n^{it},n^{it})=0$ for each $t\in\R$, $D(\lio,\lio)=D(\mob,\mob)=0$. Of course, if $\bfu(p)=0$ for each $p\in\PP$ then $D(\bfu,\bfu)=+\infty$. Moreover,  $\boldsymbol{\phi}(n)/n\in\cm$ and $D(\boldsymbol{\phi}(n)/n,\boldsymbol{\phi}(n)/n)=\sum_{p\in\PP}
    \frac1p(1-\frac{(1-p)^2}{p^2})$
    is positive and finite.
\item $D(\bfu,\bfv)=D(\bfv,\bfu)$.
\item $D(\bfu,\bfv)\leq D(\bfu,\bfw)+D(\bfw,\bfv)$, see \cite{Gr-Sa}.
\end{itemize}

When $D(\bfu,\bfv)<+\infty$ then one says that $\bfu$ {\em pretends to be} $\bfv$. For example, $\mob^2$ and $\boldsymbol{\phi}(n)/n$ pretend to be~$\raz$ (as $\sum_{p\in\PP}\frac1p(1-\frac{p-1}p)=
\sum_{p\in\PP}\frac1{p^2}<+\infty$).

\begin{Lemma}[\cite{Gr-Sa}]\label{nier}
For each $\bfu,\bfv,\bfw,\bfw'\in\cm$, we have
\begin{enumerate}[(i)]
\item
$D(\bfu\bfw,\bfv\bfw')\leq D(\bfu,\bfv)+D(\bfw,\bfw')$.
\end{enumerate}
Moreover, by (i) and a simple induction,
\begin{enumerate}[(i)]\label{pow}
\setcounter{enumi}{1}
\item
$mD(\bfu,\bfv)\geq D(\bfu^m,\bfv^m)$ for all $m\in\N$.
\end{enumerate}
\end{Lemma}
If we fix $t\neq0$ and $k\geq k_0$ then the number of $p\in\PP$ satisfying
$$
\exp\left(\frac{2\pi}t(k+\frac13)\right)\leq p\leq \exp\left(\frac{2\pi}t(k+\frac23)\right)$$
is (by the PNT) at least $C\frac{\exp(2\pi k/t)}{k/t}$ (for a constant $C>0$), whence
$$
\left|\left\{p\in\PP: k+\frac13\leq \frac{t\log p}{2\pi}\leq k+\frac23\right\}\right|\geq C\frac{\exp(2\pi k/t)}{k/t}.
$$
It follows that
\begin{equation}\label{arch1a}
\sum_{\exp(\frac{2\pi}t(k+\frac13))\leq p\leq \exp(\frac{2\pi}t(k+\frac23))}\frac1p\left(1-\cos(t\log p)\right)\geq C'\frac1k
\end{equation}
for a constant $C'>0$.
Now, using~\eqref{distance1},~\eqref{arch1a} and summing over $k$, we obtain the following:\footnote{This proof of~\eqref{m3} has been shown to us by G.\ Tenenbaum.}
\beq\label{m3}
D(\raz,n^{it})=\infty\text{ for each }t\neq0.\eeq
It is not difficult to see that for $t\neq0$, $D(\chi,n^{it})=+\infty$ for each Dirichlet character $\chi$, while for $t=0$, we have $D(\chi,1)<+\infty$ if and only if $\chi$ is principal.

\subsection{Mean of a multiplicative function. The Prime Number Theorem (PNT)}
The distance $D$ is useful when we want to compute means of  multiplicative functions. Given an arithmetic function $\bfu\colon\N\to\C$ its {\em mean} $M(\bfu)$ is defined as $M(\bfu):=\lim_{N\to\infty}\frac1N\sum_{n\leq N}\bfu(n)$
(if the limit exists).

\begin{Th}[Hal\'asz; e.g.\ Thm.\ 6.3 \cite{Elliot}]
Let $\bfu\in\cm$.
Then $M(\bfu)$ exists and is non-zero if and only if
\begin{enumerate}[(i)]
\item
there is at least one positive integer $k$ so that $\bfu(2^k)\neq-1$, and
\item
the series $\sum_{p\in\PP}\frac1p(1-\bfu(p))$ converges.
\end{enumerate}
When these conditions are satisfied, we have
$$M(\bfu)=\prod_{p\in\PP}\left(1-\frac1p\right)\left(1+\sum_{m=1}^\infty p^{-m}\bfu(p^m)\right).$$
The mean value $M(\bfu)$ exists and is zero if and only if either
\begin{enumerate}[(i)]
\setcounter{enumi}{2}
\item
there is a real number $\tau$, so that for each positive integer $k$, $\bfu(2^k)=-2^{ki\tau}$, moreover $D(\bfu,n^{i\tau})<+\infty$; or
\item
$D(\bfu,n^{it})=\infty$ for each $t\in\R$.
\end{enumerate}
\end{Th}

\begin{Cor}[Wirsing's theorem]
If $\bfu\in\cm$ is real-valued then $M(\bfu)$ exists.
\end{Cor}
\begin{proof}
Since ${\rm Re}(p^{it})={\rm Re}(p^{-it})$, and $\bfu(p)\in\R$, we have
$$
D(\raz,n^{2it})=D(n^{-it},n^{it})\leq 2D(\bfu,n^{it})$$
by the triangle inequality.
By~\eqref{m3}, it follows that $D(\bfu,n^{it})=+\infty$ for each $0\neq t\in\R$. Hence, if $D(\bfu,\raz)=+\infty$, then $D(\bfu,n^{it})=+\infty$ for each $t\in\R$ and then $M(\bfu)=0$ by Hal\'asz's theorem (iv).

If not then $D(\bfu,1)<+\infty$. Then the series $\sum_{p\in\PP}\frac1p(1-\bfu(p))$ converges (so (ii) is satisfied) and we check whether or not $\bfu(2^k)=-1$ for all $k\in\N$, that is, either (i) holds or (iii) holds.
\end{proof}

\begin{Remark} It follows from~\eqref{m3} that in Hal\'asz's theorem (iii) and (iv) are two disjoint conditions.\end{Remark}


\begin{Remark} Not all functions from $\cm$ have mean. Indeed, an Archimedean character $n^{it}$ has mean iff $t=0$. This can be shown by a direct computation: apply Euler's summation formula to $f(x)=x^{it}$ with $t\neq0$, to obtain $\frac1N\sum_{n\leq N}n^{it}=\frac{N^{it}}{it+1}+{\rm O}\left(\frac{\log N}{N}\right)$.\end{Remark}

\begin{Th}[e.g.\ \cite{Hi,Gr-Sa,Te}]\label{pntmob}\label{t:pnt1}
The PNT is equivalent to $M(\mob)=0$.\end{Th}

\begin{Remark} The statement above is an elementary equivalence, see the discussion in Section~4 \cite{Di}. For a PNT for a more general $f$ (i.e.\ not for $f=1$) the relation between such a disjointness and sums over the primes requires more quantitative estimates than simply ${\rm o}(N)$.\end{Remark}

\begin{Remark}  By Hal\'asz's theorem,  condition $M(\mob)=0$  is equivalent to $D(\mob,n^{it})=\infty$ for each $t\in\R$ ($\mob$ does not pretend to be $n^{it}$), and this can be established similarly to the proof of~\eqref{m3}.
\end{Remark}

The PNT tells us about cancelations of $+1$ and $-1$ for $\mob$. When one requires a behavior similar to random sequences, say ``square-root type cancelation'',  the result is much stronger:

\begin{Th}[Littlewood, see \cite{Ch}]
The Riemann Hypothesis holds if and only if
for every $\vep>0$, we have $\sum_{n\leq N}\mob(n)={\rm O}_\vep(N^{\frac12+\vep})$.
\end{Th}
This result  is not hard to establish and we show the sufficiency: By Corollary~\ref{odwrot}, we have
$$
\frac1{\zeta(s)}=\sum_{n=1}^\infty\frac{\mob(n)}{n^s}=
-\sum_{n=1}^\infty\mob(n)\int_n^\infty dx^{-s}=s\sum_{n=1}^\infty
\mob(n)\int_n^\infty\frac{dx}{x^{s+1}}.$$
Setting $\boldsymbol{M}(x)=\sum_{n\leq x}\mob(n)$, we obtain
\beq\label{RH1}
\frac1{\zeta(s)}=s\int_1^\infty\frac{\boldsymbol{M}(x)}{x^{s+1}}\,dx,\;\;
{\rm Re}\,s>1\eeq
and, by the assumption on $\boldsymbol{M}(\cdot)$,
$$ \int_1^\infty\left|\frac{\boldsymbol{M}(x)}{x^{s+1}}\right|\,dx=
\int_1^\infty\frac{\big|\boldsymbol{M}(x)\big|}{x^{{\rm Re}\,s+1}}\,dx\ll
\int_1^\infty x^{\frac12+\vep-({\rm Re}\,s+1)}\,dx=\int_1^\infty x^{-{\rm Re}\,s-\frac12+\vep}\,dx.
$$
It follows that the integral on the RHS of~\eqref{RH1} is absolutely convergent for ${\rm Re}\,s>\frac12+\vep$. Hence,~\eqref{RH1} yields an analytic extension of $\frac1{\zeta(\cdot)}$ to $\{s\in \C : {\rm Re}\, s>\frac12+\vep\}$. In this domain there are no zeros of $\zeta$ and by the functional equation (see~\eqref{funeq1}) on $\zeta$, we obtain the Riemann Hypothesis.

\subsection{Aperiodic multiplicative functions}
Denote by
$$
\cm_{\rm conv}:= \{\bfu\in\cm :  \lim_{N\to\infty}\frac1N\sum_{n\leq N}\bfu(an+r)\text{ exists for all }a,r\in\N\}.$$
The following is classical.
\begin{Lemma}\label{niejasny} Let $\bfu\in\cm$. Then $\bfu\in \cm_{\rm conv}$ if and only if the mean value $M(\chi\cdot \bfu)$ exists for each Dirichlet character $\chi$. \end{Lemma}

An arithmetic function $\bfu\colon\N\to\C$ is called {\em aperiodic} if, for all $a,r\in\N$, we have
$\lim_{N\to\infty}\frac1N\sum_{n\leq N}\bfu(an+r)=0$. Similarly to Lemma~\ref{niejasny}, we obtain that $\bfu\in\cm$ is aperiodic if and only if $M(\chi\cdot\bfu)=0$ for each Dirichlet character $\chi$. Delange theorem (see, e.g., \cite{Gr-Sa}) gives necessary and sufficient conditions for $\bfu$ to be aperiodic. In particular, each $\bfu\in\cm$ satisfying $D(\bfu,\chi\cdot n^{it})=0$ for all Dirichlet characters $\chi$ and all $t\in\R$, is aperiodic. Classical multiplicative functions as $\mob$ or $\lio$ are aperiodic.

Frantzikinakis and Host in \cite{Fr0} prove a deep structure theorem for multiplicative functions from $\cm$. One of the consequences of it is the following characterization of aperiodic functions: $\bfu\in\cm$ is aperiodic if and only if it is uniform, that is, all Gowers uniformity seminorms\footnote{For $N\in\N$ we write $[N]$ for the set $\{1,2,\ldots,N\}$.
Given $h,N\in\N$ and  $f\colon \N\rightarrow \C$, we let $S_hf(n)=f(n+h)$ and $f_N=\raz_{[N]}\cdot f$. For $s\in\N$, the {\em Gowers uniformity seminorm}  \cite{Gowers01} $\|.\|_{U^s_{[N]}}$ is defined in the following way:
\[
\|f\|_{U^1_{[N]}}:= \left|\frac{1}{N}\sum_{n=1}^N f_N(n)\right|
\]
and for $s\geq 1$
\[
\|f\|_{U^{s+1}_{[N]}}^{2^{s+1}}:= \frac{1}{N}\sum_{h=1}^N \left\|f_N S_h \overline{f_N}\right\|_{U^s_{[N]}}^{2^s}.
\]
A bounded function $f\colon\N\rightarrow \C$ is called
{\em uniform} if $\| f \|_{U^s_{[N]}}$ converges to zero as $N\to\infty$ for each $s\geq1$.} vanish \cite{Fr0}.
In \cite{Be-Ku-Le-Ri1a} (see Theorem~1.3 therein), this result is extended to show that $\bfu\in\cm_{conv}$ is either uniform or rational.\footnote{\label{f:rational}An arithmetic function $\bfu$ is {\em rational} if for each $\vep>0$ there is a periodic function $\bfv$ such that $\limsup_{N\to\infty}\frac1N\sum_{n\leq N}|\bfu(n)-\bfv(n)|<\vep$. Note that since $\mob$ is aperiodic, whence orthogonal to all periodic sequences it will also be orthogonal to each rational $\bfu$ \cite{Be-Ku-Le-Ri1a}. An example of rational sequence is given by $\mob^2$. For more examples, see the sets of $\mathscr{B}$-free numbers in the Erd\"os case in Section~\ref{czesc6}.}  Also, a variation of this result has been proved in \cite{Be-Ku-Le-Ri1a} (see Theorem~A therein):
\begin{equation}\label{relunif}
\parbox{0.85\textwidth}{for each positive density level set $E=\{n\in\N: \bfu(n)=c\}$ of $\bfu\in\cm$ there is a (unique if density is smaller than~1) rational (i.e.\ coming from a rational function from $\cm$) level set $R$ of $\bfv\in\cm$ such that $d(R)\raz_{E}-d(E)\raz_R$ is Gowers uniform.}
\end{equation}
For example, for $E=\{n\in\N :  \mob(n)=1\}$ the unique set $R$ is just the set of square-free numbers.

\subsection{Davenport type estimates on short intervals} Given $\bfu\in\cm$, for our purposes we will need additionally the following:\footnote{To be compared with the estimates~\eqref{vin}, where we drop the $\sup$ requirement.}   for each $(b_n)\subset\N$ with $b_{n+1}-b_n\to\infty$ and any $c\in \C$, $|c|=1$, we have
\beq\label{sh1}
\frac1{b_{K+1}}\sum_{k\leq K}\left|\sum_{b_k\leq n<b_{k+1}}c^n\bfu(n)\right|\tend{K}{\infty}0.\eeq
It is not hard to see that if $\bfu\in\cm$ satisfies~\eqref{sh1} for each $(b_n)$ and $c$ as above, then it must be aperiodic.

In fact, it follows from a break-through result in \cite{Ma-Ra} and \cite{Ma-Ra-Ta} that the class of $\bfu\in\cm$ for which~\eqref{sh1} holds contains all $\bfu$ for which
\beq\label{e:mrt}
\inf_{|t|\leq M,\chi~\text{mod}~q,q\leq Q}D(\bfu,n\mapsto \chi(n)n^{it};M)^2\to\infty,
\eeq
when  $10\leq H\leq M$, $H\to\infty$ and $Q=\min(\log^{1/125}M,\log^5 H)$; here $\chi$ runs over all Dirichlet characters of modulus $q\leq Q$ and
$$
D(\bfu,\bfv;M):=\left(\sum_{p\leq M, p\in\PP}\frac{1-{\rm Re}(\bfu(p)\overline{\bfv(p)})}{p}\right)^{1/2}
$$
for each $\bfu,\bfv\in\cm$.
Moreover, classical multiplicative functions like $\mob$ and $\lio$ satisfy~\eqref{e:mrt}, see \cite{Ma-Ra-Ta}.

Finally, note that~\eqref{sh1} true for all $(b_n)$ as above is equivalent to the following statement:
\beq\label{sh2}
\frac1M\sum_{M\leq m<2M}\left|\sum_{m\leq h<m+H} c^h\bfu(h)\right|\tend{M,H}{\infty,H={\rm o}(M)} 0\eeq
(we can also replace the first sum by $\sum_{1\leq m<M}$), see \cite{Ab-Le-Ru2} for details. This statement is much closer to the original formulations of (simplified versions of) theorems from \cite{Ma-Ra,Ma-Ra-Ta}.

One more consequence of the main result in \cite{Ma-Ra} is the following:

\begin{Th}[Thm.\ 1.1 in \cite{Ma-Ra-Ta} and a corollary for $k=2$ therein] \label{srednie}
For $H\to\infty$ arbitrarily slowly with $M\to\infty$ ($H\leq M$), we have
$$
\sum_{h\leq H}\left|\sum_{m\leq M}\mob(m)\mob(m+h)\right|={\rm o}(HM).$$
\end{Th}

\subsection{The KBSZ criterion}
Sarnak's conjecture is aimed at showing that deterministic sequences (i.e.\ those given as observable sequences  in the zero entropy systems) are orthogonal
to $\mob$. In particular, as $\mob$ is a multiplicative function, the result\footnote{The main ideas for this result appeared in
\cite{Da-De} and \cite{Mo-Va}. It was first established in a slightly different form in \cite{Ka} and
then in \cite{Bo-Sa-Zi}, see also \cite{Ha} for a proof. The criterion has its origin in the bilinear method of Vinogradov \cite{Vi} which is a technique to study sums of $a$ over primes in terms of sums over progressions $\sum_{n\leq N} a_{dn}$ and sums $\sum_{n\leq N} a_{d_1n}a_{d_2n}$. If $a_n=f(T^nx)$ then these sums are Birkhoff sums for powers of $T$ and their joinings.

In what follows we will refer to Theorem~\ref{t:kbsz} as to the KBSZ criterion.} below establishes disjointness with $\mob$.

\begin{Th}[\cite{Ka,Bo-Sa-Zi}]\label{t:kbsz}
Assume that $(a_n)$ is a bounded sequence of complex numbers.
Assume that for all prime numbers $p\neq q$
\beq\label{kbsz1}
\frac1N\sum_{n\leq N}a_{pn}\ov{a}_{qn}\tend{N}{\infty}0.\eeq
Then, for each multiplicative function $\bfu\in\cm$, we have
\beq\label{kbsz2}
\frac1N\sum_{n\leq N}a_n\bfu(n)\tend{N}{\infty}0.\eeq\end{Th}

For example, see \cite{Ka}, the criterion applies to the sequences of the form $(e^{iP(n)})$, where $P\in \R[x]$ has at least one irrational coefficient (different from the constant term).

In the context of dynamical systems, we use this criterion for $a_n=f(T^nx)$, $n\geq 1$. Clearly, this leads us to study the behavior of different (prime) powers of a fixed map $T$. We should warn the reader that when applying Theorem~\ref{t:kbsz}, we do not expect to have \eqref{kbsz1} satisfied for all continuous functions, in fact, even in uniquely ergodic systems, in general, it cannot hold for all zero mean functions\footnote{We can easily see that when $Tx=x+\alpha$ is an irrational rotation on $\T=[0,1)$, then, by the Weyl criterion on uniform distribution,~\eqref{kbsz1} is satisfied for all characters (for all $x\in\T$), but there are continuous zero mean functions for which~\eqref{kbsz1} fails \cite{Ku-Le}.} but we need a subset of $C(X)$ which is linearly dense, cf.\ footnote~\ref{f:lindense}.

We will also need the following variation of Theorem~\ref{t:kbsz},  see \cite{Ab-Le-Ru2}:

\begin{Prop} \label{pr:kbsz} Assume that $(a_n)$ is a bounded sequence of complex numbers. Assume, moreover, that
\beq\label{kbsz}
\limsup_{\stackrel{p,q\to \infty}{\text{ different primes}}}\left(\limsup_{N\to\infty}
\left|\frac1N\sum_{n\leq N}a_{pn}\ov{a}_{qn}\right|\right)=0.\eeq
Then, for each multiplicative function $\bfu\colon\N\to\C$, $\bfu\in\cm$, we have
\beq\label{eq:sa2}
\lim_{N\to\infty}\frac1N\sum_{n\leq N}a_n\cdot \bfu(n)=0.\eeq
\end{Prop}

\begin{Remark}\label{r:etKBSZ}
In contrast to the KBSZ criterion given by Theorem~\ref{t:kbsz}, condition~\eqref{kbsz} has its ergodic theoretical counterpart -- the property called AOP (see Section~\ref{s:AOP}) which is a measure-theoretic invariant.
\end{Remark}

\section{Chowla conjecture}\label{czesc3} In this section we get into the subject of the Chowla conjecture which is the main motivation for Sarnak's conjecture.
\subsection{Formulation and ergodic interpretation}\label{s:sec31}
The Chowla conjecture deals with  higher order correlations of the M\"obius function,\footnote{As a matter of fact, in \cite{Ch}, it is formulated for the Liouville function. We follow \cite{Sa}.  For a discussion on an equivalence of the Chowla conjecture with $\mob$ and $\lio$, we invite the reader to~\cite{Ra-this}. As shown in \cite{Ma-Ra-Ta}, there are non-pretentious (completely) multiplicative functions for which Chowla conjecture fails. For more information, see the discussion on Elliot's conjecture in \cite{Ma-Ra-Ta}.} that is, the conjecture asserts that
\beq\label{chsfor}
\frac1N\sum_{n\leq N}\mob^{j_0}(n)\mob^{j_1}(n+k_1)\ldots\mob^{j_r}(n+k_r)
\tend{N}{\infty} 0
\eeq
whenever $1\leq k_1<\ldots<k_r$, $j_s\in\{1,2\}$ not all equal to 2, $r\geq0$.\footnote{The Chowla conjecture is rather ``close'' in spirit to the Twin Number Conjecture in the sense that the latter is expressed by $(\ast)\;\sum_{n\leq x}\boldsymbol{\Lambda}(n)\boldsymbol{\Lambda}(n+2)=(2\Pi_2)\cdot x+{\rm o}(x)$, where $\Pi_2=\prod_{p}(1-\frac1{(p-1)^2})=0,66016\ldots$ which can be compared with $\sum_{n\leq x}\mob(n)\mob(n+2)={\rm o}(x)$ which is ``close'' to the Chowla conjecture, see e.g. \cite{Ta1}. A recent development shows that it is realistic to claim that the Chowla conjecture with an error term  of the form ${\rm o}((\log N)^{-A})$ for some $A$ large enough ($A$ depending on the number of shifts of $\mob$ that are considered) implies~$(\ast)$.  (Of course, everywhere $\boldsymbol{\Lambda}$ is a good approximation of $\raz_{\PP}$).

See also \cite{Mu-Va1} for a (conditional) equivalence of $(\ast)$ with $\sum_{n\leq N}\boldsymbol{\Lambda}(n)\mob(n+2)={\rm o}(N)$.
}

We will now explain an ergodic meaning of the Chowla conjecture. Recall that given a dynamical system $(X,T)$ and $\mu\in M(X,T)$, a point $x\in X$
is called {\em generic for} $\mu$ if
$$
\frac1N\sum_{n\leq N}f(T^nx)\tend{N}{\infty}\int_X f\,d\mu$$
for each $f\in C(X)$. Equivalently,
$\frac1N\sum_{n\leq N} \delta_{T^nx}\tend{N}{\infty} \mu$ (we recall that $M(X,T)$ is endowed with the weak$^\ast$ topology which makes it a compact metrizable space). By compactness, each point is {\em quasi-generic} for a certain measure $\nu\in M(X,T)$, i.e.\
$$
\frac1{N_k}\sum_{n\leq N_k}\delta_{T^nx}\tend{k}{\infty}\nu
$$
for a certain subsequence $N_k\to\infty$. Let
\beq\label{kugen}
\text{Q-gen}(x):=\{\nu\in M(X,T) : x\text{ is quasi-generic for }\nu\}.\footnote{We recall that either $x$ is generic or $\text{Q-gen}(x)$ is a connected uncountable set, see Proposition~3.8 in \cite{De-Gr-Si}.}
\eeq

Assume now that we have a finite alphabet $A$. We consider $(A^{\Z},S)$, so called {\em full shift}, or more precisely,  {\em two-sided full shift}, where $A^\Z$ is endowed with the product topology and $S((x_n))=(y_n)$ with $y_n=x_{n+1}$ for each $n\in\Z$. Each $X\subset A^\Z$ that is closed and $S$-invariant yields a {\em subshift}, i.e.\ the dynamical system $(X,S)$. One way to obtain a subshift is to choose $x\in A^\Z$ and consider the closure $X_x$ of the orbit of $x$ via $S$. If $x$ is given as a one-sided sequence, $x\in A^{\N}$, we still might consider
\beq\label{definicjaXx}
X_x:=\{y\in A^\Z : \mbox{each block appearing in $y$ appears in $x$}\}
\eeq
to obtain a two-sided subshift. In case when each block appearing in $x$ reappears infinitely often, $X_x=\ov{\{S^n\ov{x} : n\in\Z\}}$, for some $\ov{x}$ for which $\ov{x}(j)=x(j)$ for each $j\geq1$ but, in general, there is no such a good $\ov{x}$.
Moreover, we will let ourselves speak about a one-sided sequence $x$ to be generic or quasi-generic for a measure $\nu\in M(X_x,S)$.

Now take $A=\{-1,0,1\}$. For each subshift $X\subset\{-1,0,1\}^\Z$ let $\theta\in C(X)$ be defined as
\beq\label{theta}
\theta(y)=y(0),\;y\in X.\eeq
Note that directly from the Stone-Weierstrass theorem we obtain the following.

\begin{Lemma}\label{linden}  The linear subspace generated by the constants and the family $$\{\theta^{j_0}\circ S^{k_0}\cdot\theta^{j_1}\circ S^{k_1}\cdot\ldots
\cdot\theta^{j_r}\circ S^{k_r} : k_i\in\Z, j_i\in\{1,2\},i=0,1,\ldots,r, r\geq0 \}$$ of continuous functions
is an algebra of functions separating points, hence it is dense in $C(X)$.\end{Lemma}

The subshift $(X_{\mob},S)$
is called the {\em M\"obius system} and  $X_{\mob^2}\subset\{0,1\}^\Z\subset\{-1,0,1\}^\Z$  is the {\em square-free system}.\footnote{The point $\mob^2$ is recurrent, so there is a ``completion'' of $\mob^2$ to a two-sided sequence generating the same subshift.} Note that $s\colon (z(n))\mapsto (z(n)^2)$ will settle a factor map between the M\"obius system and the square-free system.
The point $\mob^2$ is a generic point for so called {\em Mirsky measure} $\nu_{\mob^2}$ \cite{Ce-Si,Mi} (see Section~\ref{invme}). In other words, there are frequencies of blocks on $\mob^2$:
for each block $B\in\{0,1\}^\ell$, the following limit exists:
$$
\lim_{N\to\infty}\frac1N\left|\{1\leq n\leq N-\ell :  \mob^2(n,n+\ell-1)=B\}\right|=:\nu_{\mob^2}(B).$$
We can now consider the relatively independent extension\footnote{Consider  Bernoulli measure $B(1/2,1/2)$ on $\{-1,1\}^{\Z}$ and Mirsky measure $\nu_{\mob^2}$ on $\{0,1\}^{\Z}$. Measure $\widehat{\nu}_{\mob^2}$ is the image of the product measure $B(1/2,1/2)\ot\nu_{\mob^2}$ via the map
$$(x,y)\mapsto ((x(n)\cdot y(n)))_{n\in\Z}\in\{-1,0,1\}^{\Z}.$$} $\widehat{\nu}_{\mob^2}$ of $\nu_{\mob^2}$ which is the measure on $s^{-1}(X_{\mob^2})\subset\{-1,0,1\}^{\Z}$ given by the following condition: for each block $C\in\{-1,0,1\}^\ell$, we have
$$
\widehat{\nu}_{\mob^2}(C):=\frac1{2^k}\nu_{\mob^2}(C^2),$$
where $C^2$ is obtained from $B$ by squaring on each coordinate and $k$ is the number of~1 in $C^2$.
A straightforward computation shows that
\beq\label{calki}
\int_{\{-1,0,1\}^\Z}\theta^{j_0}\circ S^{k_0}\cdot\theta^{j_1}\circ S^{k_1}\cdot\ldots
\cdot\theta^{j_r}\circ S^{k_r}\,d\widehat{\nu}_{\mob^2}=0\eeq
whenever $\{j_0,\ldots,j_r\}\neq\{2\}$.
On the other hand, in view of Lemma~\ref{linden}, the values of integrals
$$
\int_{\{-1,0,1\}^\Z}\theta^{2}\circ S^{k_0}\cdot\theta^{2}\circ S^{k_1}\cdot\ldots
\cdot\theta^{2}\circ S^{k_r}\,d\widehat{\nu}_{\mob^2}$$
for all  $k_i\in\Z$ and $r\geq0$ entirely determine the Mirsky measure $\nu_{\mob^2}$.

\begin{Cor}\label{Chg}
The Chowla conjecture holds if and only if $\mob$ is a generic point for $\widehat{\nu}_{\mob^2}$.\end{Cor}
\begin{proof} We consider any extension of $\mob$ to a two-sided sequence (for example we set $\mob(n)=0$ for each $n\leq0$).
Suppose that
\beq\label{pom1}
	\frac1{N_k}\sum_{n\leq N_k}\delta_{S^n\mob}\tend{k}{\infty} \kappa.
\eeq
In order to get $\kappa=\widehat{\nu}_{\mob^2}$, in view of Lemma~\ref{linden}, we need to show that
$$\int_{\{-1,0,1\}^\Z}\theta^{j_0}\circ S^{k_0}\cdot\theta^{j_1}\circ S^{k_1}\cdot\ldots
\cdot\theta^{j_r}\circ S^{k_r}\,d\kappa=0$$
for any choice of integers $k_0<k_1<\ldots<k_r$,  $\{j_0,j_1,\ldots,j_r\}\neq\{2\}$ and $r\geq0$. Since the measure $\nu$ is $S$-invariant, it is the same as to show that
$$\int_{\{-1,0,1\}^\Z}\theta^{j_0}\circ \cdot\theta^{j_1}\circ S^{k_1-k_0}\cdot\ldots
\cdot\theta^{j_r}\circ S^{k_r-k_0}\,d\kappa=0.$$
Now, we have $1\leq k_1-k_0<\ldots<k_r-k_0$ and the result follows from~\eqref{chsfor} and~\eqref{pom1}.
\end{proof}

The Chowla conjecture for $r=0$ is just the PNT, however, it remains open even for $r=1$. As in \cite{Sa}, we could consider a weaker version of the Chowla conjecture. Namely, we say that $\mob$ satisfies the {\em topological Chowla conjecture} if $X_{\mob}=s^{-1}(X_{\mob^2})$.

\begin{Remark}\label{r:ratCh} Note that \eqref{chsfor} holds if
$$|\{0\leq t\leq r: j_t=1\}|=1.$$
Indeed, it is not hard to see that if $t_0$ is the only index for which $j_{t_0}=1$ then the sequence $a(n):=\prod_{t\neq t_0}\mob^2(n+k_t)$ is rational. Hence, $\mob$ is orthogonal to $a(\cdot)$, cf.\ footnote~\ref{f:rational}.\end{Remark}

\subsection{The Chowla conjecture implies Sarnak's conjecture}
Assume that $(X,T)$ is a topological system.  Following \cite{Kam,We9} a point $x\in X$ is called {\em completely deterministic} if for each measure $\nu\in \text{Q-gen}(x)$ (see~\eqref{kugen}),  the measure theoretic dynamical system $(X,\cb(X),\nu,T)$ has zero Kolmogorov-Sinai entropy: $h_{\nu}(T)=0$. Of course, if the topological entropy of $T$ is zero, then by the Variational Principle, each $x\in X$ is completely deterministic. On the other hand, $(X_{\mob^2},S)$ has positive topological entropy \cite{Ab-Le-Ru1,MR3430278,Sa} and $\mob^2\in X_{\mob^2}$ is completely deterministic, see \cite{Ce-Si,Ab-Ku-Le-Ru}.

Let $f\in C(X)$ and $x\in X$ be completely deterministic. We have
$$
\frac1N\sum_{n\leq N}f(T^nx)\mob(n)
=\int_{X\times X_{\mob}}
(f\otimes \theta)d\left(\frac1N\sum_{n\leq N}\delta_{(T\times S)^n(x,\mob)}\right).
$$
We can assume that
$$
\frac1{N_k}\sum_{n\leq N_k}\delta_{(T\times S)^n(x,\mob)}\tend{k}{\infty} \rho\text{ in the space }M(X\times X_{\mob},T\times S).
$$
Under the Chowla conjecture, the projection of $\rho$ on $X_{\mob}$ is equal to $\widehat{\nu}_{\mob^2}$ (since, by Corollary~\ref{Chg}, $\mob$ is a generic point for  $\widehat{\nu}_{\mob^2}$), while the projection of $\rho$ on $X$ is some $T$-invariant measure  $\kappa$ and $h_\kappa(T)=0$ (since $x$ is completely deterministic). Note that
$\rho$ is a joining\footnote{Recall that if $R_i$ is an automorphism of a probability standard Borel space $(Z_i,\cd_i, \nu_i)$, $i=1,2$, then each $R_1\times R_2$-invariant measure $\lambda$ on $(Z_1\times Z_2,\cd_1\ot\cd_2)$ having the projections $\nu_1$ and $\nu_2$, respectively is called a {\em joining} of $R_1$ and $R_2$: we write $\lambda\in J(R_1,R_2)$. If $R_1,R_2$ are ergodic then the set $J^e(R_1,R_2)$ of ergodic joinings between $R_1$ and $R_2$ is non-empty. A fundamental notion here is the {\em disjointness} (in sense of Furstenberg)  \cite{Fu}: $R_1$ and $R_2$ are disjoint if $J(R_1,R_2)=\{\nu_1\ot\nu_2\}$: we write $R_1\perp R_2$. For example, zero entropy automorphisms are disjoint with automorphisms having completely positive entropy (Kolmogorov automorphisms) and also a relativized version of this assertion holds.} of the (measure-theoretic) dynamical systems  $(X,\kappa,T)$ and $(X_{\mob},\widehat{\nu}_{\mob^2},S)$. Moreover, the latter automorphism has the so called relative Kolmogorov property with respect to the factor $(X_{\mob^2},\nu_{\mob^2},S)$. We then consider the restriction of the joining $\rho|_{X\times X_{\mob^2}}$ and $\rho|_{X_{\mob}}$ to obtain two systems that have a common factor (namely $X_{\mob^2}$) relatively to which the first one has zero entropy and the second being relatively Kolmogorov. Since the function $\theta$ is orthogonal to $L^2(X_{\mob^2},\nu_{\mob^2})$, the relative disjointness theorem on zero entropy and Kolmogorov property yields the following (see also Remark~\ref{r:chs2}):
\begin{Th}[\cite{Ab-Ku-Le-Ru}]\label{ChtoS} The Chowla conjecture implies
$$
\frac1N\sum_{n\leq N}f(T^nx)\mob(n)\to0$$
for each dynamical system $(X,T)$, $f\in C(X)$ and $x\in X$ completely deterministic.
In particular, the Chowla conjecture implies Sarnak's conjecture.\footnote{We will see later that some special cases of validity of convergence in~\eqref{chsfor} also have their ergodic interpretations and they imply M\"obius disjointness for restricted classes of dynamical systems of zero entropy; in particular, see Corollary~\ref{c:mn3} and Corollary~\ref{c:logsarnaksing}.}
\end{Th}

\begin{Remark}\label{r:chs1} It is also proved in \cite{Ab-Ku-Le-Ru} that this seemingly stronger statement of the validity of Sarnak's conjecture at completely deterministic points is in fact equivalent to the M\"obius disjointness of all zero entropy systems.\end{Remark}

\begin{Remark}\label{r:chs2}A word for word repetition of the above proof\footnote{The above proof was already suggested by Sarnak in \cite{Sa}.} yields the same result when we replace $\mob$ by another generic point of $\widehat\nu_{\mob^2}$ in which we control the relative Kolmogorov property over the maximal factor with zero entropy, so called Pinsker factor. In particular, we can replace $\mob$ by $\lio$ (for which the Pinsker factor will be just the one-point dynamical system).

As a matter of fact, it is expected that each aperiodic real-valued multiplicative function satisfies the Chowla type result (and hence satisfies the Sarnak type result), see the conjectures by Frantzikinakis and Host formulated after Theorem~\ref{t:frho2}.\end{Remark}

\begin{Remark}\label{r:chs3} The original proof of Sarnak of the implication ``Chowla conjecture $\Rightarrow$ Sarnak's conjecture'' used some combinatorial arguments and probabilistic methods, see \cite{Ta1}.\end{Remark}

Sarnak's conjecture~\eqref{sc1} is formulated for the M\"obius function. But of course one can consider other multiplicative functions.\footnote{If M\"obius disjointness in a dynamical system is shown through the KBSZ criterion then we obtain orthogonality with respect to all multiplicative functions.} Below, we show that if we use the Liouville function then nothing changes.

\begin{Cor}\label{c:equiv}
Sarnak's conjecture with respect to $\mob$ is equivalent to Sarnak's conjecture with respect to $\lio$.\end{Cor}
\begin{proof}
Let us recall the basic relation between these two functions: $\lio(n)=\sum_{d^2\divides n}\mob(n/d^2)$.

Assume that $(X,T)$ is a dynamical system with $h(T)=0$. As the zero entropy class is closed under taking powers, we assume M\"obius disjointness for all powers of $T$.
Then
\begin{align*}
\frac1N\sum_{n\leq N}f(T^nx)\lio(n)&=\frac1N\sum_{n\leq N} f(T^nx)\left( \sum_{d^2 \divides n}\mob(n/d^2)\right)\\
&=\frac1N\sum_{n\leq N}\sum_{d^2 \divides n}\mob(n/d^2)f((T^{d^2})^{n/d^2}x)\\
&=\sum_{d\leq\sqrt{N}}\frac1{d^2}\cdot\frac1{N/d^2}\sum_{n\leq N/d^2}\mob(n)f((T^{d^2})^nx).
\end{align*}
Take $\vep>0$ and select $M\geq1$ so that $\sum_{d\geq M}\frac1{d^2}<\vep$.
Consider $T,T^2,T^3,\ldots, T^M$. We have
$$
\left|\frac1N\sum_{n\leq N}f(T^{kn}x)\mob(n)\right|<\vep$$
for all $k=1,\ldots, M$ whenever $N\geq N_0$. It follows that
$$
\left|\frac1{N/d^2}\sum_{n\leq N/d^2}\mob(n)f(T^{d^2n}x)\right|<\vep$$
for all $d=1,\ldots,M$ if $N>MN_0$. Otherwise we estimate such a sum by $\|f\|_\infty$.

To obtain the other direction, we first recall that $\mob^2$ is a completely deterministic point. Then use Theorem~\ref{ChtoS} for $\lio$ (see Remark~\ref{r:chs2}),
write $\lio(n)\mob^2(n)=\mob(n)$ for each $n\geq1$ and we obtain
$$
\frac1N\sum_{n\leq N} f(T^nx)\mob^2(n)\lio(n)=\frac1N\sum_{n\leq N}(f\otimes\theta)((T\times S)^n(x,\mob^2))\lio(n)\to0$$
as the point $(x,\mob^2)$ is completely deterministic.\end{proof}

\subsection{The logarithmic versions of Chowla and Sarnak's conjectures} \label{s:TAO} An intriguing problem arises whether the Chowla and Sarnak's conjecture are equivalent. An intuition from ergodic theory
would say that this is rather not the case as the class of systems that are disjoint (in the Furstenberg sense) from all zero entropy measure-theoretic systems is the class of Kolmogorov automorphisms and not only Bernoulli automorphisms (and a relative version of this result persists).\footnote{If we consider general sequences $z\in\{-1,0,1\}^{\N}$ then we can speak about the Sarnak and Chowla properties on a more abstract level: for example the Chowla property of $z$ means~\eqref{chsfor} with $\mob$ replaced by $z$. See Example~5.1 and Remark~5.3 in \cite{Ab-Ku-Le-Ru} for sequences  orthogonal to all deterministic sequences but not satisfying the Chowla property. However,  arithmetic functions in these examples are not multiplicative.

However, an analogy between disjointness results in ergodic theory and disjointness of sequences is sometimes accurate. For example, a measure-theoretic dynamical system has zero entropy if and only if it is disjoint with all Bernoulli automorphisms. As pointed out in \cite{Ab-Ku-Le-Ru} (Prop. 5.21), a sequence $t\in\{-1,1\}^\N$ is completely deterministic if and only if it is disjoint with any sequence $z\in\{-1,0,1\}^{\N}$ satisfying the Chowla property.}

From that point of view a recent remarkable result of Terence Tao \cite{Ta4} about the equivalence of logarithmic versions of the Chowla and Sarnak's conjectures is quite surprising. We will formulate some versions\footnote{See Remark 1.9. Also, in \cite{Ta4} the Liouville function $\lio$ is considered, see page~2 in \cite{Ta4} how to replace $\lio$ by $\mob$.} of three (out of five)
conjectures from \cite{Ta4}.

\vspace{2ex}

\noindent
{\bf Conjecture A:}  We have
$$\frac1{\log N}\sum_{n\leq N}\frac{\mob^{j_0}(n)\mob^{j_1}(n+k_1)\ldots\mob^{j_r}(n+k_r)}n
\tend{N}{\infty} 0
$$
whenever $1\leq k_1<\ldots<k_r$, $j_s\in\{1,2\}$ not all equal to 2, $r\geq0$.

\vspace{2ex}
\begin{Remark}\label{r:fromP}
It should be noted that passing to such logarithmic averages moves one away from questions about primes, twin primes and subtleties such as the parity problem. For example, the statement $\sum_{n\leq N}\frac{\mob(n)}n={\rm o}(\log N)$ is easy to establish (in fact, $\left|\sum_{n\leq N}\frac{\mob(n)}n\right|\leq1$), while the PNT is equivalent to much stronger statement $\sum_{n=1}^\infty\frac{\mob(n)}n=0$ (as conditionally convergent series).

On the other hand, the logarithmically averaged Chowla conjecture implies that all ``admissible'' configurations do appear on $\mob$, see Corollary~\ref{c:topCh} below (the topological Chowla conjecture for $\lio$ implies that all blocks of $\pm1$ appear in $\lio$).
\end{Remark}

\vspace{2ex}

\noindent
{\bf Conjecture B:} We have
$$\frac1{\log N}\sum_{n\leq N}\frac{f(T^nx)\mob(n)}n\tend{N}{\infty} 0$$
whenever $(X,T)$ is a topological system of zero topological entropy, $f\in C(X)$ and $x\in X$.

\vspace{2ex}

To formulate the third conjecture, we need to recall the definition of a nilrotation.  Let $G$ be a connected, simply connected Lie group and $\Gamma\subset G$ a lattice (a discrete, cocompact subgroup).
For any $g_0\in G$ we define $T_{g_0}(gH):=g_0gH$. Then the topological system $(G/\Gamma, T_{g_0})$ is called a {\em nilrotation}.

\vspace{2ex}

\noindent
{\bf Conjecture C:}
Let $f\in C(G/\Gamma)$ be Lipschitz continuous and $x_0\in G$. Then (for $H\leq N$)
$$
\sum_{n\leq N}\frac{\sup_{g\in G}\left|\sum_{h\leq H}f(T_g^{h+n}(x_0\Gamma))\mob(n+h)\right|}n={\rm o}(H\log N).
$$

\begin{Th}[\cite{Ta4}]\label{logChS}
Conjectures A, B and C are equivalent.\end{Th}

\begin{Remark} Tao also shows that  if instead of logarithmic averages we come back to  Ces\`aro averages, then
$$
\mbox{Conjecture A $\Rightarrow$ Conjecture B $\Rightarrow$ Conjecture C}$$
and it is the implication Conjecture C $\Rightarrow$ Conjecture A  that requires logarithmic averages.
\end{Remark}

\begin{Remark} Let us consider the Ces\`aro version of Conjecture~C with $H={\rm o}(N)$ and we drop the assumption on the sup (which is inside), i.e.:   for each $g\in G$, we have
$$
\frac1N\sum_{n\leq N}\left|\sum_{h\leq H}f(T_g^{h+n}(x_0\Gamma))\mob(n+h)\right|\tend{H,N}{\infty,H={\rm o}(N)} 0.
$$
This is a particular case of
what we will see in Section~\ref{s:AOP}, where we introduce the strong MOMO notion (hence, the validity of Sarnak's conjecture on (typical) short interval).\end{Remark}

\begin{Cor}[a letter of W.\ Veech in June 2016] \label{c:topCh}  Sarnak's conjecture implies topological Chowla conjecture. Equivalently, Sarnak's conjecture implies that each block $B\in\{-1,0,1\}^\ell$ for which $B^2$ appears in $\mob^2$ appears in $\mob$ (and the entropy of $(X_{\mob},S)$ equals $\frac6{\pi^2}\log3$).
\end{Cor}
\begin{proof}
Indeed, Sarnak's conjecture implies its logarithmic version which, by Theorem~\ref{logChS}, implies logarithmic Chowla conjecture, that is, $\frac1{\log N}\sum_{n\leq N}\frac{\delta_{S^n\mob}}n\to \widehat{\nu}_{\mob^2}$.  However, the logarithmic averages of the Dirac measures are convex combinations of the consecutive Ces\`aro averages\footnote{\label{topCh}
Assume that $(a_n)$ is a bounded sequence and set $A_n=a_1+\ldots+ a_n$.  Then, we have by summation by parts
\begin{multline}\label{sparts}
\frac1{\log N}\sum_{n\leq N}\frac{a_n}n=\frac1{\log N}\sum_{n\leq N}(A_{n+1}-A_n)\frac1n\\
=\frac1{\log N}\sum_{n\leq N}A_n
\left(\frac1n-\frac1{n+1}\right)+{\rm o}(1)=
\frac1{\log N}\sum_{n\leq N}\frac{A_n}n\frac1{n+1}+{\rm o}(1).
\end{multline}
It follows that:
\begin{itemize}
\item
If the Ces\`aro averages of $(a_n)$ converge, so do the logarithmic averages of $(a_n)$.
\item The converse does not hold (see e.g.\ \cite{MR1512943} in $\mathscr{B}$-free case, Section~\ref{sofm}).
\item If the Ces\`aro averages converge along a subsubsequence $(N_k)$ then not necessarily the logarithmic averages do the same.
Indeed, by~\eqref{sparts},$\frac1{\log N_k}\sum_{n\leq N_k}\frac{a_n}n$  is (up to a small error) a convex combination of the Ces\`aro averages for all $n\leq N_k$.
\end{itemize}}
$\frac1n\sum_{j\leq n}\delta_{S^j\mob}$, so if we take a block $B\in s^{-1}(X_{\mob^2})$, we have $\widehat{\nu}_{\mob^2}(B)>0$ and therefore there exists $n$ such that $\frac1n\sum_{j\leq n}\delta_{S^j\mob}(B)>0$, which means that $B$ appears in $\mob$.
\end{proof}

\begin{Remark}\label{r:gkl1} (added in October 2017) As a matter of fact, as shown in \cite{Go-Kw-Le}, Sarnak's conjecture implies the existence of a subsequence $(N_k)$ along which $\frac1{N_k}\sum_{n\leq N_k}\delta_{S^n\mob}\to \widehat{\nu}_{\mob^2}$. This follows from a general observation that, given a topological system $(X,T)$, whenever an ergodic
  measure $\nu$ is a limit of a subsequence $(M_k)$ of logarithmic averages of Dirac measures: $\nu=\lim_{k\to \infty}\frac1{\log M_k}\sum_{m\leq M_k}\frac{\delta_{T^{m}x}}m$, then there exists a subsequence $(N_k)$ for which $\nu=\lim_{k\to\infty}\frac1{N_k}\sum_{n\leq N_k}\delta_{T^nx}$. We apply this to the measure $\widehat{\nu}_{\mob^2}$ which is ergodic.\end{Remark}

In \cite{Ta3}, Tao proves the logarithmic version of Chowla conjecture  for the correlations of order~2 (which we formulate for the Liouville function):

\begin{Th}[\cite{Ta3}]\label{t:logCh2}
For each $0\neq h\in\Z$, we have
$$
\frac1{\log N}\sum_{n\leq N}\frac{\lio(n)\lio(n+h)}n\tend{N}{\infty}0.$$
\end{Th}

See also \cite{Ma-Ra}, where it is proved that for each integer $h\geq1$ there exists $\delta(h)>0$ such that $\limsup_{N\to\infty}\frac1N\left|\sum_{n\leq N}\lio(n)\lio(n+h)\right|\leq1-\delta(h)$ and \cite{Ma-Ra-Ta1}, where it is proved that for the Liouville function the eight patterns of length~3 of signs occur with positive lower density, and the density result with lower density replaced by upper density persists for $k+5$ patterns (out of total $2^k$) for each $k\in\N$.

For a proof of a function field Chowla's conjecture, see~\cite{Ca-Ru}.

\begin{Remark}
See also \cite{Ta-Te}, where, given $k_0,\ldots k_{\ell}\in\Z$ and $\bfu_0,\ldots,\bfu_{\ell}\in\cm$, one studies sequences of the form
$$
n\mapsto \bfu_0(n+ak_0)\cdot\ldots\cdot\bfu_{\ell}(n+ak_{\ell}),\ a\in \Z.
$$
By considering their logarithmic averages, one obtains a sequence $(f(a))$. The main result of \cite{Ta-Te} is a structure theorem (depending on whether or not the product $\bfu_0\cdot\ldots\bfu_{\ell}$ weakly pretends to be a Dirichlet character) for the sequences $(f(a))$. As a corollary,  the logarithmically averaged Chowla conjecture is proved for any odd number of shifts.
\end{Remark}

\subsection{Frantzikinakis' theorem}
Tao's approach from \cite{Ta4} is continued in \cite{Fr}. Before we formulate Frantzikinakis' results, let us
interpret some arithmetic properties, especially the role of a ``good behavior'' on (typical) short interval of a multiplicative function in the ergodic theory language.

\subsubsection{Ergodicity of measures for which $\mob$ is quasi-generic}\label{s:interpr}
In this subsection we summarize ergodic consequences of some recent, previously mentioned number-theoretic results, cf. \cite{Fr2}. By that we mean that we consider all measures $\kappa\in \text{Q-gen}(\mob)$ and we study ergodic properties of the dynamical systems $(X_{\mob},\kappa, S)$.

Let $\kappa\in \text{Q-gen}(\mob)$, i.e.\  $\frac1{M_k}\sum_{m\leq M_k}\delta_{S^m\mob}\tend{k}{\infty}\kappa\in M(X_{\mob},S)$ for some increasing sequence $(M_k)$. As usual, $\theta(x)=x(0)$ ($\theta\in C(X_{\mob})$). We have
\beq\label{wl1}
\int_{X_{\mob}}\theta\,d\kappa=0,
\eeq
as the integral equals $\lim_{k\to\infty}\frac1{M_k}\sum_{n\leq M_k}\theta(S^n\mob)=0$ (by the PNT). Denoting by $Inv$ the $\sigma$-algebra of $S$-invariant (modulo the measure $\kappa$) subsets of $X_{\mob}$, we recall that
$$
\frac1H\sum_{h\leq H}\theta\circ S^h\tend{H}{\infty} \EE(\theta|Inv) \text{ in }L^2(X_{\mob},\kappa)
$$
(by the von Neumann ergodic theorem). We
 want to show that
$$
 \theta\perp L^2(X_{\mob}, Inv, \kappa)
$$
(i.e.\ $\kappa$ must be ``slightly'' ergodic). In other words, we want to show that
$$
\int_{X_{\mob}}\left|\frac1H\sum_{h\leq H}\theta\circ S^h\right|^2\,d\kappa\tend{H}{\infty}0.
$$
But such integrals can be computed:
$$\frac1{M_k}\sum_{m\leq M_k}\left|\frac1H\sum_{h\leq H}\theta\circ S^h(S^m\mob)\right|^2\tend{k}{\infty}
    \int_{X_{\mob}}\left|\frac1H\sum_{h\leq H}\theta\circ S^h\right|^2\,d\kappa.$$
Putting things together, given $\vep>0$, for $H\geq1$ large enough, we want to see
$$\limsup_{k\to\infty}\frac1{M_k}\sum_{m\leq M_k}\left|\frac1H\sum_{h\leq H}\mob(m+h)\right|^2\leq\vep.$$
The latter is true because of \cite{Ma-Ra}:
for  a ,,typical'' $m$ the sum $\left|\frac1H\sum_{m\leq h<m+H}\mob(h)\right|$ is small.

\begin{Remark}As the calculation above shows, the fact that $$\frac1{M}\sum_{m\leq M}\left|\frac1H\sum_{h\leq H}\mob(m+h)\right|^2\to0$$ when $H\to\infty$ and $H={\rm o}(M)$ is equivalent to $\theta\perp L^2(X_{\mob},Inv,\kappa)$ for each $\kappa\in \text{Q-gen}(\mob)$. In particular, the Chowla conjecture implies the above short interval behavior.
\end{Remark}

However, remembering that $\kappa|_{X_{\mob^2}}=\nu_{\mob^2}$, one can ask now whether $\theta$ is measurable with respect to the factor given  by the Mirsky measure. As this factor has rational discrete spectrum \cite{Ce-Si}, to show that this is not the case, we need to prove that
$\theta\perp L^2(\Sigma_{rat})$, where $\Sigma_{rat}$ stands for the factor given by the whole rational spectrum of $(X_{\mob},\kappa,S)$.  To do it, we need to show that for each $r\geq1$, we have
$$
\frac1N\sum_{n\leq N}\theta\circ S^{rn}\tend{N}{\infty} 0\text{ in }L^2(X_{\mob},\kappa).$$
This convergence can be shown  by using the strong MOMO property (which we will consider in Section~\ref{s:AOP}) for the rotation $j\mapsto j+1$ on $\Z/r\Z$. We skip this argument here and show still a stronger consequence.

Assume that $\kappa\in \text{Q-gen}(\mob)$ and that we want to show that the spectral measure of $\theta\in L^2(X_{\mob},\kappa)$ is continuous. Hence, we need to show that
$$
\frac1H\sum_{h\leq H}|\widehat{\sigma}_{\theta}(h)|\tend{H}{\infty}0$$
when $H\to\infty$. Equivalently, we need to show that
$$
\frac1H\sum_{h\leq H}\left|\int_{X_{\mob}} \theta\circ S^h\cdot \theta\,d\kappa\right|\tend{H}{\infty}0.$$
If we fix $H\geq1$ then
\begin{multline*}
\int_{X_{\mob}}\theta\circ S^h\cdot \theta\,d\kappa=
\lim_{k\to\infty}
\frac1{M_k}\sum_{m\leq M_k}\theta\circ S^h(S^m\mob)\cdot\theta(S^m\mob)\\
=\frac1{M_k}\sum_{m\leq M_k}\mob(m+h)\mob(m).
\end{multline*}
It follows that we need to show that
$$
\frac1H\sum_{h\leq H}\left|\frac1{M_k}\sum_{m\leq M_k}\mob(m+h)\mob(m)\right|\to 0$$
when $H,M_k\to\infty$; to be precise, given $\vep>0$ we want to show that for $H>H_\vep$, we have $\limsup_{k\to\infty}\frac1H\sum_{h\leq H}\left|\frac1{M_k}\sum_{m\leq M_k}\mob(m+h)\mob(m)\right|<\vep$. Hence, directly from Theorem~\ref{srednie}, we obtain the following.

\begin{Cor}\label{c:mn2} The spectral measure of $\theta$ is continuous for each $\kappa\in \text{Q-gen}(\mob)$.\end{Cor}

While it is obvious that the subshift $X_{\mob}$ is uncountable (indeed, it is the subshift $X_{\mob^2}$ which is already uncountable, see Section~\ref{czesc6}), it is not clear whether $X_{\lio}$ is uncountable. However, if a subshift $(Y,S)$ is countable, all its ergodic measures are given by periodic orbits, hence there are only countably many of them and it easily follows that each $\kappa\in M(Y,S)$ yield a system with discrete spectrum.
Hence, immediately from Corollary~\ref{c:mn2}, we obtain that:

\begin{Cor}\label{c:xlio}
The subshift $X_{\lio}$ is uncountable.\footnote{The result has been observed in \cite{Fr-Ho1}, cf.\ also \cite{Hu-Wa-Ye}.}
\end{Cor}

From Corollary~\ref{c:mn2} we derive immediately the M\"obius disjointness of all dynamical systems with ``trivial'' invariant measures (see also \cite{Hu-Wa-Zh}). This kind of problems will be the main subject of our discussion in Section~\ref{s:AOP}.

\begin{Cor}\label{c:mn3} Let $(X,T)$ be any topological dynamical system such that, for each measure $\nu\in M(X,T)$,  $(X,\nu,T)$ has discrete spectrum (not necessarily ergodic, of course). Then $(X,T)$ is M\"obius disjoint. In particular, the result holds if $M^e(X,T)$ is countable with each  member of $M^e(X,T)$ yielding a discrete spectrum dynamical system.\end{Cor}
\begin{proof} Fix $x\in X$ and consider
$$\frac1{M_k}\sum_{m\leq M_k}\delta_{(T^mx,S^m\mob)}\tend{k}{\infty}\rho.$$
We have $\rho|_{X_{\mob}}=:\kappa\in\text{Q-gen}(\mob)$ and $\rho|_X=:\nu$. Now, we fix $f\in C(X)$ and we need to show that
$\int f\ot \theta\,d\rho=0$. But
\begin{equation}\label{rnie}
\int_{X\times X_{\mob}}f\ot \theta\,d\rho=\int_{X\times X_{\mob}}(f\ot 1)\cdot (1\ot \theta)\,d\rho=0.
\end{equation}
Indeed, the spectral measure of $f\ot 1$ with respect to $\rho$ is the same as the spectral measure of $f$ with respect to $\nu$ and the spectral measure of $1\ot\theta$ with respect to $\rho$ is the same as the spectral measure of $\theta$ with respect to $\kappa$. Therefore, these spectral measures are mutually singular by assumption and Corollary~\ref{c:mn2}. Hence, the functions $f\ot 1$ and $1\ot\theta$ are orthogonal, i.e.~\eqref{rnie} holds.\footnote{We use here the standard result in the theory of unitary operators that mutual singularity of spectral measures implies orthogonality. Recall also the classical result in ergodic theory that spectral disjointness implies disjointness.}
\end{proof}

If we have all ergodic measures giving discrete spectrum but we have too many ergodic measures then the argument above does not go through. Consider
\begin{equation}\tag{$\ast$}\label{Z}
(x,y)\mapsto (x,x+y)\text{ on }\T^2.~\footnote{Consider $X_1=X_2=\T^2$ with $\mu_1=\mu_2=Leb_{\T^2}$, the diagonal joining $\Delta$ on $X_1\times X_2$ and $f(x,y)=\overline{\theta(x,y)}$ with $\theta(x,y)=e^{2\pi iy}$. The spectral measure of $\theta$ is Lebesgue, and all ergodic components of the measure $\mu_1$ have discrete spectra.}
\end{equation}
\begin{Question}[Frantzikinakis (2016)]  \label{q:Fr} Can we obtain $\kappa\in\text{Q-gen}(\lio)$, so that $(X_{\lio},\kappa,S)$ is isomorphic to~\eqref{Z}?\end{Question}
Of course, the answer to Question~\ref{q:Fr} is expected to be negative.

\subsubsection{Frantzikinakis' results}
We now follow \cite{Fr} and formulate results for the Liouville function, although, up to some obvious modifications, they also hold for $\mob$.

\begin{Th}[\cite{Fr}]\label{t:fr1} Assume that $N_k\to\infty$ and let $\frac1{\log N_k}\sum_{n\leq N_k}\frac{\delta_{S^n\lio}}n\tend{k}{\infty} \kappa$.
If $\kappa$ is ergodic then the Chowla conjecture (and Sarnak's conjecture) holds along $(N_k)$ for the logarithmic averages.\end{Th}

Taking into account footnote~\ref{topCh}, we cannot deduce a similar statement for ordinary averages along $(N_k)$ but in view of \cite{Go-Kw-Le}, see Remark~\ref{r:gkl1}, the Chowla conjecture holds along another subsequence. The situation becomes clear when $(N_k)$ is the sequence of all natural numbers and we assume genericity.

\begin{Cor}[\cite{Fr}]\label{c:fr2} If $\lio$ is generic for an ergodic measure then the Chowla conjecture holds.
\end{Cor}

Let us say a few words on the proof. Recall that given a bounded sequence $(a(n))\subset \C$ admitting correlations,\footnote{I.e., we assume the existence of the limits of sequences $\left(\frac1N\sum_{n\leq N} a'(n)a'(n+k_1)\ldots a'(n+k_r)\right)_{N\geq1}$ for every $r\in\N$ and $k_1,\ldots, k_r\in\N$ (not necessarily distinct) with $a'=a$ or $\ov{a}$. It is not hard to see that $a$ admits correlations if and only if it is generic, cf.\ Section~\ref{s:sec31}.} one defines its  local uniformity seminorms (see Host and Kra \cite{Ho-Kr})
in the following manner:
\begin{align}
\|a\|^2_{U^1(\N)}&=\EE_{h\in\N}
\EE_{n\in\N}a(n+h)\ov{a(n)},\label{correl1}\\
\|a\|^{2^{s+1}}_{U^{s+1}(\N)}&=\EE_{h\in N}\|S_ha\cdot\ov{a}\|^{2^s}_{U^s(\N)},\;s\geq2,\label{correl2}
\end{align}
where, for each bounded sequence $(b(n))$,  $(S_hb)(n):=b(h+n)$ and $\EE_{n\in \N}b(n)=\lim_{N\to\infty}\frac1N\sum_{n\leq N}b(n)$.
(Similar definitions are considered along a subsequence $(N_k)$.)

The following result has been proved by Tao:

\begin{Th}[\cite{Ta4}] \label{t:tao4}Assume that $\lio$ is generic. The Chowla conjecture holds if and only if $\|\lio\|_{U^s(\N)}=0$ for each $s\geq 1$.\footnote{We have $\|\lio\|_{U^1(\N)}=0$ by \cite{Ma-Ra}, moreover $\|\lio\|_{U^2(\N)}=0$ is equivalent to $\lim_{N\to\infty}\EE_{m\in\N}\sup_{\alpha\in[0,1)}
\left|\EE_{n\in[m,m+N]}\lio(n)e^{2\pi in\alpha}\right|=0$ (cf.\ Conjecture~C) and remains open. For a subsequence version of Theorem~\ref{t:tao4} for logarithmic averages, see \cite{Ta4}.}
\end{Th}

\begin{Remark} We have assumed in the statement of Theorem~\ref{t:tao4} that $\lio$ is generic but we would like also to note that, without this latter (strong) assumption, Tao obtained the equivalence in Theorem~\ref{t:tao4} for the logarithmic averages, see Conjecture~1.6 and Theorem~1.9 in \cite{Ta4} (however, one has to modify the definition of seminorms \cite{Ta4}).\end{Remark}

Hence, under the assumption of Corollary~\ref{c:fr2}, we need to prove that all local uniform seminorms of $\lio$ vanish. The inverse theorem for seminorms reduces this problem to the statement: for every basic nilsequence $(a(n))$\footnote{By that we mean $a(n)=f(g^n\Gamma)$ for some continuous $f\in C(G/\Gamma)$ and $g\in G$.} on an $s-1$-step nilmanifold $G/\Gamma$ and  every $s-2$-step manifold $H/\Lambda$, we have
$$
\lim_{N\to\infty}\EE_{m\in\N}\sup_{b\in \Psi_{H/\Lambda}}\left|\EE_{n\in[m,m+N]}\lio(n)a(n)b(n)\right|=0,$$
where $\Psi_{H/\Lambda}$ is a special class of basic nil-sequences (coming from Lipschtz functions). The latter is then proved using a deep induction argument.

\subsection{Dynamical properties of Furstenberg systems associated to the Liouville and M\"obius functions} \label{s:frho}
We now continue considerations about logarithmic version of Sarnak's conjecture, cf. Conjecture~B, Theorem~\ref{t:fr1}. Consider all measures $\kappa$ for which $\lio$ is logarithmically quasi-generic, i.e. $\frac1{\log N_k}\sum_{n\leq N_k}\frac{\delta_{S^n\lio}}{n}\to \kappa$ for some $N_k\to\infty$. We denote the set of all such measures by $\text{$\log$-Q-gen}(\lio)$. Following \cite{Fr-Ho1}, for each $\kappa\in\text{$\log$-Q-gen}(\lio)$ the corresponding measure-theoretic dynamical system $(X_{\lio},\kappa,S)$ will be called a {\em Furstenberg system of} $\lio$.  Before we get closer to the results of \cite{Fr-Ho1}, let us see first some consequence of Theorem~\ref{t:logCh2} for the logarithmic Sarnak's conjecture:
\beq\label{logtheta}
\parbox{0.8\textwidth}{For each Furstenberg system $(X_{\lio},\kappa,S)$,  the spectral measure $\sigma_{\theta}$ of $\theta$ is Lebesgue.}
\end{equation}
Indeed, assuming
$\frac1{\log N_k}\sum_{n\leq N_k}
\frac{\delta_{S^n\lio}}n\tend{k}{\infty}\kappa$,
Theorem~\ref{t:logCh2} tells us  that for each $h\in\Z\setminus\{0\}$, we have
$$
\widehat{\sigma}_\theta(h)=\int_{X_{\lio}}\theta\circ S^h\cdot \theta\,d\kappa=\lim_{k\to\infty}\frac1{\log N_k}\sum_{n\leq N_k}\frac{\lio(n+h)\lio(n)}n=0.$$

Using~\eqref{logtheta} and repeating the proof of Corollary~\ref{c:mn3}, we obtain the following.

\begin{Cor}\label{c:logsarnaksing} Let $(X,T)$ be a topological system such that each of its Furstenberg's systems has singular spectrum. Then $(X,T)$ is logarithmically Liouville disjoint.
\end{Cor}

The starting point of the paper \cite{Fr-Ho1} is a surprising Tao's identity (implicit in \cite{Ta3}) for general sequences which in its ergodic theory language (cf.\ Subsection~\ref{s:interpr}) takes
the following form:

\begin{Th}[Tao's identity, \cite{Fr-Ho1}] \label{t:Taoid}
Let $\kappa\in \text{$\log$-Q-gen}(\lio)$. Then
$$
\int_{X_{\lio}}\left(\prod_{j=1}^{\ell} \theta\circ S^{k_j}\right)\,d\kappa=
(-1)^{\ell}\lim_{N\to\infty}\frac{\log N}N\sum_{\PP\ni p\leq N}\int_{X_{\lio}}\left(\prod_{j=1}^{\ell} \theta\circ S^{pk_j}\right)\,d\kappa$$
for all $\ell\in\N$ and $k_1,\ldots,k_{\ell}\in\Z$.
\end{Th}
Now, the condition in Theorem~\ref{t:Taoid} is purely abstract (indeed, the function~$\theta$ generates the Borel $\sigma$-algebra), and the strategy to cope with logarithmic Sarnak's conjecture is to describe the class of measure-theoretic dynamical systems satisfying the assertion of Theorem~\ref{t:Taoid} and then to obtain Liouville disjointness for all systems which are disjoint (in the Furstenberg sense) from all members of the class. In fact, Frantzikinakis and Host deal with extensions of Furstenberg systems of $\lio$, so called {\em systems of arithmetic progressions with prime steps}.\footnote{Given a measure-theoretic dynamical system $(Z,\mathcal{D},\rho,R)$, its system of arithmetic progressions with prime steps is of the form $(Z^\Z,\cb(Z^{\Z}),\widetilde{\rho},S)$, where $S$ is the shift and the (shift invariant) measure $\widetilde{\rho}$ is determined by
$$
\int_{Z^{\Z}}\prod_{j=-m}^mf_j(z_j)\,d\widetilde{\rho}(z)=
\lim_{N\to\infty}\frac{\log N}N\sum_{p\leq N}\int_Z\prod_{j=-m}^m f_j\circ R^{pj}\,d\rho$$
for all $m\geq0$, $f_{-m},\ldots,f_m\in L^\infty(Z,\rho)$ (here $z=(z_j)$). It is proved that such shift systems have no irrational spectrum. One of key observations is that each Furstenberg system of the Liouville function is a factor of the associated system of arithmetic progressions with prime steps.}  They prove the following result.

\begin{Th}[\cite{Fr-Ho1}] \label{t:FH1} For each system of arithmetic progressions with prime steps, its ``typical'' ergodic component is isomorphic to the direct product of an infinite-step nilsystem and a Bernoulli automorphism.\footnote{The product decomposition depends on the component.} In particular, each Furstenberg system $(X_{\lio},\kappa,S)$ of $\lio$ is a factor of a system which:\\
(i) has no irrational spectrum and\\
(ii) has ergodic components isomorphic to the direct product of an infinite-step nilsystem and a Bernoulli automorphism. \end{Th}

\begin{Remark} All the above results are also true when we replace $\lio$ by $\mob$.\end{Remark}

Then, some new disjointness results in ergodic theory are proved (for example, all totally ergodic automorphisms are disjoint from an automorphism satisfying (i) and (ii) in Theorem~\ref{t:FH1}) and the following remarkable result is obtained:

\begin{Th}[\cite{Fr-Ho1}]\label{t:frho1} Let $(X,T)$ be a topological dynamical system of zero entropy with countably many ergodic invariant measures.
Then Conjecture~B holds for $(X,T)$.\end{Th}

In particular, logarithmic Sarnak's conjecture holds for all zero entropy uniquely ergodic systems.
As a matter of fact,  some new\footnote{They are new even for irrational rotations. Cf.\ the notions of (S)-strong and (S$_0$)-strong and their equivalence to the Chowla type condition in \cite{Ab-Ku-Le-Ru}.} consequences are derived:

\begin{Th}[\cite{Fr-Ho1}] \label{t:frho2} Let $(X,T)$ be a topological dynamical system with zero entropy. Assume that $x\in X$ is generic for a measure $\nu$
with only countably many ergodic components all of which yield totally ergodic systems. Then, for every $f\in C(X)$, $\int_Xf\,d\nu=0$, we have
$$
\lim_{N\to\infty}\frac1{\log N}\sum_{n\leq N}\frac{f(T^nx)\prod_{j=1}^{\ell}\mob(n+k_j)}{n}=0$$
for all $\ell\in\N$ and $k_1,\ldots,k_{\ell}\in\Z$.\end{Th}

New conjectures are proposed in \cite{Fr-Ho1}:
\begin{enumerate}
\item
Every real-valued $\bfu\in\cm$ has a unique Furstenberg system (i.e.\ $\bfu$ is generic) which is ergodic and isomorphic to the direct product of a Bernoulli automorphism and an odometer.
\item
If, additionally, $\bfu\in\cm$ takes values $\pm1$ then its Furstenberg system is either Bernoulli or it is an odometer.
\end{enumerate}

Finally, it is noticed in \cite{Fr-Ho1} that the complexity of the Liouville function has to be superlinear, that is
\beq\label{mobcompl}
\lim_{N\to\infty}\frac1N\left|\{B\in\{-1,1\}^N:\:B\text{ appears in }\lio\}\right|=\infty.
\eeq
The reason is that, as shown in \cite{Fr-Ho1}, for transitive systems having linear block growth we have only finitely many ergodic measures (and clearly systems with linear block growth have zero topological entropy). Hence, by Theorem~\ref{t:frho1}, such systems are Liouville disjoint. As $X_{\lio}$ is not Liouville disjoint, $\lio$ cannot have linear block growth, i.e.~\eqref{mobcompl} holds.

\section{The MOMO and AOP properties}\label{s:AOP}
\subsection{The MOMO property and its consequences}
We will now consider Sarnak's conjecture from the ergodic theory point of view. We ask whether (already) measure-theoretic properties of a measurable system $(Z,\cd,\kappa,R)$ imply the validity of~\eqref{mdis1} for any $(X,T)$, $f\in C(X)$ provided that $x\in X$ is a generic point for a measure $\mu$ such that the measure-theoretic system $(X,\cb(X),\mu,T)$ is measure-theoretically isomorphic to $(Z,\cd,\kappa,R)$. More specifically,
we can ask whether some measure-theoretic properties of  $(Z,\cd,\kappa,R)$ can imply M\"obius disjointness of all its uniquely ergodic models.\footnote{Note that the answer is positive in all uniquely ergodic models of the one-point system: each such a model has a unique fixed point that attracts each orbit on a subset of density~1, cf.\ the map $e^{2\pi ix}\mapsto e^{2\pi ix^2}$, $x\in[0,1)$. This argument is however insufficient already for uniquely ergodic models of the exchange of two points: in this case we have a density~1 attracting 2-periodic orbit $\{a,b\}$, but we do not control to which point $a$ or $b$ the orbit returns first. Quite surprisingly, it seems that already in this case we need \cite{Ma-Ra} to obtain M\"obius disjointness of all uniquely ergodic models.}  We recall that the Jewett-Krieger theorem implies the existence of a uniquely ergodic model of each ergodic system.\footnote{If all uniquely ergodic systems were M\"obius disjoint, then as noticed by T.\ Downarowicz, we would get that the Chowla conjecture fails in view of the result of B.\ Weiss \cite{We1} Thm.\ 4.4' on approximation of generic points of ergodic measures  by uniquely ergodic sequences.}
As a matter of fact, there are plenty of such models and they can have various additional topological properties including topological mixing\footnote{Topological mixing for example excludes the possibility of having eigenfunctions continuous.}  \cite{Leh}.
Here is another variation of the approach to view M\"obius disjointness as a measure-theoretic invariant:
\begin{Question}\label{pyt2}
Does the M\"obius disjointness in a certain uniquely ergodic model of an ergodic system yield the M\"obius disjointness in all its uniquely ergodic models?
\end{Question}
To cope with these questions we need a definition. Let $\bfu\colon\N\to\C$ be an arithmetic function.\footnote{Our objective is of course the M\"obius function $\mob$, however the whole approach can be developed for an arithmetic function satisfying some additional properties.}

\begin{Def}[strong MOMO\footnote{The acronym comes from M\"obius Orthogonality of Moving Orbits.} property \cite{Ab-Ku-Le-Ru1}]\label{def:SM}
\label{def:strongMOMO}
We say that $(X,T)$ satisfies the \emph{strong MOMO property} (relatively to $\bfu$) if, for any increasing sequence of integers
$0=b_0<b_1<b_2<\cdots$ with $b_{k+1}-b_k\to\infty$, for any sequence $(x_k)$ of points in $X$, and any $f\in C(X)$, we have
  \begin{equation}
    \label{eq:defMOMOSI}
    \frac{1}{b_{K}} \sum_{k< K} \left|\sum_{b_k\le n<b_{k+1}} f(T^{n-b_k}x_k) \bfu(n)\right| \tend{K}{\infty} 0.
  \end{equation}
\end{Def}

\begin{Remark}\label{r:mondra} The property~\eqref{eq:defMOMOSI} looks stronger than the condition on M\"obius disjointness. The idea behind it is to look at the pieces of orbits (of different points) in one system as a single orbit of a point in a different, larger but ``controllable'' (from measure-theoretic point of view) system.\end{Remark}

\begin{Remark}\label{momointerpr}
One can easily show (as in Section~\ref{s:interpr}) that the strong MOMO property (relative to $\mob$) implies $f\ot\theta\perp L^2(Inv,\rho)$ for each $\rho\in  \text{Q-gen}((x,\mob),T\times S)$.\footnote{$Inv$ stands here for the $\sigma$-algebra of $T\times S$-invariant sets modulo~$\rho$.}\end{Remark}

By taking $f=1$ in Definition~\ref{def:SM}, we obtain that whenever strong MOMO holds, $\bfu$ has to satisfy:
\beq\label{eq:Mobius-like}
 \frac{1}{b_{K}} \sum_{k< K} \left|\sum_{b_k\le n<b_{k+1}} \bfu(n)\right| \tend{K}{\infty} 0\eeq
 for every sequence
$0=b_0<b_1<b_2<\cdots$ with $b_{k+1}-b_k\to\infty$.
In particular, $\frac1N\sum_{n<N}\bfu(n)\tend{N}{\infty}0$. This is to be compared with~\eqref{sh1}, \eqref{sh2} and~\eqref{e:mrt} to realize that we require a special behavior of $\bfu$ on a typical short interval.

\begin{Th}[\cite{Ab-Ku-Le-Ru1}]\label{thmA} Let $(Z,\cd,\kappa,R)$ be an ergodic dynamical system. Let $\bfu\colon\N\to\C$ be an arithmetic function.
The following conditions are equivalent:
\begin{enumerate}[(a)]
\item
There exist a topological system $(Y,S)$ enjoying the strong MOMO property (relative to $\bfu$) and $\nu\in M^e(Y,S)$ such that the measurable systems $(Y,\cb(Y),\nu,S)$ and  $(Z,\cd,\kappa,R)$ are isomorphic.\label{p11}
\item
  For any topological dynamical system $(X,T)$ and any $x\in X$, if there exists a finite number of $T$-invariant measures $\mu_j$, $1\le j\le t$, such that
  \begin{itemize}
    \item $(X,\cb(X),\mu_j,T)$ is measure-theoretically isomorphic to $(Z,\cd,\kappa,R)$ for each $j$,
    \item  any measure for which $x$ is quasi-generic is a convex combination of the measures $\mu_j$, i.e.\ $\text{Q-gen}(x)\subset{\rm conv}(\mu_1,\ldots,\mu_t)$,
  \end{itemize}
then $\frac1N\sum_{n\leq N}f(T^nx)\bfu(n)\tend{N}{\infty}0$ for each $f\in C(X)$.\label{p2}
\item
All uniquely ergodic models of $(Z,\cd,\kappa,R)$ enjoy the strong MOMO property (relative to $\bfu$).\label{p3}
\end{enumerate}
\end{Th}

The proof of implication \eqref{p11}$\Rightarrow$\eqref{p2}
borrows some ideas from \cite{Hu-Wa-Zh}
and the proof of implication \eqref{p2}$\Rightarrow$\eqref{p3} uses some ideas from \cite{Ab-Le-Ru2}.

\begin{Remark}\label{r:momorot} It can be easily shown that any minimal (hence uniquely ergodic) rotation on a compact Abelian group satisfies the strong MOMO property (say, relatively to $\mob$). It follows from Theorem~\ref{thmA} (and the Halmos-von Neumann theorem) that in each uniquely ergodic model of an ergodic automorphism with discrete spectrum, we also have the strong MOMO property (in particular, the M\"obius disjointness).\end{Remark}

We now list three consequences of Theorem~\ref{thmA}:

\begin{Cor}[\cite{Ab-Ku-Le-Ru}] \label{c:uniformity}
\begin{enumerate}[(a)]
\item
If Sarnak's conjecture holds then the strong MOMO property (relative to $\mob$) holds for every zero entropy dynamical system.\footnote{That is, Sarnak's conjecture and the strong MOMO property (relatively to $\mob$) for all deterministic systems are equivalent statements.}
\item
If Sarnak's conjecture holds then it holds uniformly, that is, the convergence in~\eqref{mdis1} is uniform in $x$.\footnote{\label{f:eqSandMOMO}It is not hard to see that the MOMO property implies the relevant uniform convergence. As a matter of fact, the strong MOMO property is equivalent to the uniform convergence (in $x$, for a fixed $f\in C(X)$) on short intervals: $\frac1M\sum_{1\leq m<M}\left|\frac1H\sum_{m\leq h<m+H}f(T^hx)\mob(n)\right|\to 0$ (when $H,M\to\infty$ and $H={\rm o}(M)$). It follows that we have equivalence of: Sarnak's conjecture~\eqref{sc1}, Sarnak's conjecture in its uniform form, Sarnak's conjecture in its short interval uniform form and the strong MOMO property. Moreover, each of these conditions is implied by the Chowla conjecture.}
\item
Fix $\delta_{(\ldots0.00\ldots)}\neq\kappa\in M^e((\D_L)^{\Z},S)$, where $\D_L=\{z\in \C :  |z|\leq L\}$. Let $(X,T)$ be any uniquely ergodic model of $((\D_L)^{\Z},\kappa,S)$. Then for any $\bfu\in(\D_L)^{\Z}$ for which $\text{Q-gen}(\bfu)\subset{\rm conv}(\kappa_1,\ldots,\kappa_m)$, where $((\D_L)^{\Z},\kappa_j,S)$ for $j=1,\ldots,m$  is measure-theoretically isomorphic to $((\D_L)^{\Z},\kappa,S)$, the system $(X,T)$ does not satisfy the strong MOMO property (relative to $\bfu$).\footnote{This result means that there must be an observable sequence in $(X,T)$ which significantly correlates with $\bfu$.}
\end{enumerate}
\end{Cor}

\begin{Remark}\label{r:compdet} Let us come back to Theorem~\ref{ChtoS} and Remark~\ref{r:chs2}, i.e.\ to the reformulation of Sarnak's conjecture using completely deterministic sequences. We intend to show that a natural generalization of Corollary~\ref{c:uniformity}~(b) to the completely deterministic case fails. Indeed, consider the square-free system $(X_{\mob^2},S)$. In Remark~\ref{r:ratCh}, we have already noticed that whenever $k_j$, $j=1,\ldots,r$ are different non-negative integers, then
\begin{equation}\tag{$\ast$}\label{A}
\sum_{n\leq N}\mob^2(n+k_1)\ldots\mob^2(n+k_{r-1})\mob(n+k_r)={\rm o}(N).
\end{equation}
It follows that for each $f\in C(X_{\mob^2})$,  for each $k\in\Z$, we have
\begin{equation}\tag{$\ast\ast$}\label{B}
\frac1N\sum_{n\leq N} f(S^{n+k}\mob^2)\mob(n)\to 0.
\end{equation}
On the other hand, the convergence in~\eqref{B} cannot be uniform in $k\in\Z$. Indeed, if it were then the whole square-free system would be M\"obius disjoint. This is however impossible since $(X_{\mob^2},S)$ is hereditary, see Remark~\ref{r:hererd}. Indeed,
we can find $y\in X_{\mob^2}$ such $y(n)=1$ if and only if $\mob(n)=1$ and $y(n)=0$ otherwise (then $y\leq \mob^2$) and if we set $\theta(z):=z(0)$ then $\lim_{N\to\infty}\frac1N\sum_{n\leq N}\theta(S^ny)\mob(n)=\frac3{\pi^2}$.

See also \cite{Mu-Va2}, where a quantitative version  of~$(\ast)$ has been proved.
\end{Remark}

Note that Theorem~\ref{thmA} does not fully answer Question~\ref{pyt2}. In certain situations the following  general (lifting) lemma of Downarowicz and Lema\'nczyk can be helpful:

\begin{Lemma}[\cite{Ab-Ka-Le,Do-Ka}]\label{model2} Assume that an ergodic automorphism $R$ is coalescent.\footnote{This means that each measure-preserving transformation commuting with $R$ must be invertible. Finite multiplicity of the Koopman operator associated to $R$ guarantees coalescence. In particular, all ergodic rotations are coalescent.} Let $(\widetilde{X},\widetilde{T})$ and $(X,T)$ be uniquely ergodic models of $R$. Assume that $T$ is a topological factor of $\widetilde{T}$, i.e.\ there exists  $\pi\colon\widetilde{X}\to X$ which is continuous and onto and which satisfies $\pi\circ\widetilde{T}=T\circ \pi$. If $T$ is M\"obius disjoint then also $\widetilde{T}$ is M\"obius disjoint. \end{Lemma}

\subsection{M\"obius disjointness and entropy}
Sarnak's conjecture deals with deterministic systems but M\"obius disjointness, a priori, does not exclude the possibility of positive (topological) entropy systems which are M\"obius disjoint.\footnote{Sarnak in \cite{Sa} mentions that Bourgain has constructed a positive entropy system which is M\"obius disjoint but this construction has never been published.}
The first ``natural'' trial would be to take the square-free system $(X_{\mob^2},S)$ which has positive entropy (see Section~\ref{invme}) and clearly $\mob^2$ is orthogonal to $\mob$.  However, in spite of the orthogonality of the two sequences, as we have noticed in Remark~\ref{r:compdet}, the square-free system is not M\"obius  disjoint.

Recently, Downarowicz and Serafin \cite{Do-Se} constructed M\"obius disjoint positive entropy homeomorphisms of arbitrarily large entropy. On the other hand, see~\cite{Kar2}, in the subshift of finite type case we do not have M\"obius disjointness. Using Katok's horseshoe theorem, it follows that $C^{1+\delta}$-diffeomorphisms of surfaces are not M\"obius disjoint but the following question seems to be open:

\begin{Question}\label{q:smooth} Is there a positive entropy diffeomorphism of a compact manifold which is M\"obius disjoint?\end{Question}

Viewed all this above, another natural question arises:

\begin{Question}\label{q5}
Does there exist an ergodic positive entropy measure-theoretic system all uniquely ergodic models of which are M\"obius disjoint?
\end{Question}
\noindent

Using Theorem~\ref{thmA}, Sinai's theorem on Bernoulli factors (see e.g.\ \cite{Gl}) and B.\ Weiss' theorem \cite{We2} on strictly ergodic models of some diagrams a partial answer to Question~\ref{q5} is given by the following result:

\begin{Cor}[\cite{Ab-Ku-Le-Ru1}]\label{c11}
Assume that $\bfu\in (\D_L)^{\Z}$ is generic for a Bernoulli measure~$\kappa$. Let $\bfv\in(\D_L)^\Z$, $\bfu$ and $\bfv$ correlate. Then for each dynamical system $(X,T)$ with $h(X,T)>h((\D_L)^{\Z},\kappa,S)$, we do not have the strong MOMO property relatively to $\bfv$.\end{Cor}

By substituting $\bfu=\lio$, $\bfv=\mob$ and assuming the Chowla conjecture for $\lio$, we obtain that no system $(X,T)$ with entropy $>\log2$ satisfies the strong MOMO relatively to $\mob$. When $\mob$ is replaced by $\lio$, we still have a stronger result.
\begin{Prop}[\cite{Ab-Ku-Le-Ru1}]\label{dodENT}
Assume that the Chowla conjecture holds for $\lio$. Then
no topological system $(X,T)$ with positive entropy satisfies the strong MOMO property relatively  to $\lio$.
\end{Prop}

\begin{Remark} The proof of Theorem~\ref{thmA} tells us that when $(Z,\cd,\kappa,R)$ is ergodic and has positive entropy  then there exists a system $(X,T)$, which is not Liouville disjoint, with at most three ergodic measures and all of these measures yield a measurable system isomorphic to $R$. Therefore, it seems reasonable to conjecture that the answer to Question~\ref{q5} is negative.\end{Remark}

We now have a completely clear picture for the Liouville function: it follows  from Theorem~\ref{ChtoS} (for $\lio$) and Proposition~\ref{dodENT} that if the Chowla conjecture holds for $\lio$ then the strong MOMO property (relatively to $\lio$) holds for $(X,T)$ if and only if $h(X,T)=0$. Using  footnote~\ref{f:eqSandMOMO}, we immediately obtain  Proposition~\ref{dodENT} in its equivalent form:

\begin{Cor}\label{c:shintUnif} Assume that the Chowla conjecture holds for $\lio$. Then, the short interval uniform  convergence in~\eqref{mdis1} (with $\mob$ replaced by $\lio$) takes place if and only if $h(X,T)=0$.\end{Cor}



\subsection{The AOP property and its consequences}
We need an ergodic criterion to establish the strong MOMO property in models of an automorphism. This turns out to be a natural ergodic counterpart of the KBSZ criterion (Theorem~\ref{t:kbsz}). Following \cite{Ab-Le-Ru2}
an ergodic automorphism $R$ is said to have {\em asymptotically orthogonal powers} (AOP)  if for each $f,g\in L^2_0(Z,\cd,\kappa)$, we have
\beq\label{momoe4}
\lim_{\PP\ni p,q\to\infty, p\neq q} \sup_{\kappa\in J^e(R^p,R^q)}\left|\int_{X\times X} f\ot g\,d\kappa\right|=0.
\eeq
Rotation $Rx=x+1$ acting on $\Z/k\Z$ with $k\geq2$ has no AOP property because of Dirichlet's theorem on primes in arithmetic progressions. Hence, AOP implies total ergodicity (clearly, AOP is closed under taking factors).
The AOP property implies zero entropy \cite{Ab-Le-Ru2}.

Clearly, if the powers of $R$ are pairwise disjoint\footnote{This is a ``typical'' property of an automorphism of a probability standard Borel space~\cite{Ju}. M\"obius disjointness for uniquely ergodic models for this case is already noticed in \cite{Bo-Sa-Zi}.} then $R$ enjoys the AOP property.
In order to see a less trivial example of an AOP automorphism, consider any totally ergodic discrete spectrum automorphism $R$ on $(Z,\cd,\kappa)$. For $f,g$ take eigenfunctions  corresponding to eigenvalues $c,d$, respectively.
Now, take $\rho\in J^e(R^p,R^q)$ and consider
$$
\int_{Z\times Z}f\ot g\,d\rho=\int_{Z\times Z}(f\ot\raz_Z)\cdot(\raz_Z\ot g)\,d\rho.$$
Notice that $f\ot \raz_Z$ and $\raz_Z\ot g$ are eigenfunctions of $(Z\times Z,\rho,R^p\times R^q)$ corresponding to $c^p$ and $d^q$, respectively. If $c^p\neq d^q$ (and this is the case for all but one pair $(p,q)$ because of total ergodicity)  then these eigenfunctions are orthogonal and we are done. We will see more examples in Section~\ref{czesc5}.

\begin{Remark}\label{r:alliso} For an AOP automorphism the powers need not be disjoint. As a matter of fact, we can have an AOP automorphism with all of its non-zero powers isomorphic.\footnote{Take an ergodic rotation with the group of eigenvalues $\{e^{2\pi i\alpha m/n} :  m,n\in\Z, n\neq0, \alpha\notin\Q\}$.}\end{Remark}

\begin{Th}[\cite{Ab-Le-Ru2,Ab-Ku-Le-Ru}]\label{thmB}
Let $\bfu\in\cm$. Suppose that $(Z,\cd,\kappa,R)$ satisfies AOP. Then the following are equivalent:
\begin{itemize}
\item
$\bfu$ satisfies \eqref{eq:Mobius-like};
\item
The strong MOMO property relatively to $\bfu$ is satisfied in each uniquely ergodic model $(X,T)$ of $R$.
\end{itemize}
In particular, if the above holds, for each $f\in C(X)$, we have
$$
\frac1N\sum_{n\leq N}f(T^nx)\bfu(n)\tend{N}{\infty} 0 \text{ uniformly in } X.
$$
\end{Th}

\begin{Cor}\label{corC} Assume that $(Z,\cd,\kappa,R)$ enjoys the AOP property. Then, in each uniquely ergodic model $(X,T)$ of $R$, we have
\beq\label{momoe5}
\frac1M\sum_{M\leq m<2M}\left|\frac1H\sum_{m \leq h<m+H}f(T^nx)\mob(n)\right|\tend{H,M}{\infty,H={\rm o}(M)}0\eeq
for all $f\in C(X)$, $x\in X$.\end{Cor}

The AOP property can be defined for actions of locally compact (second countable) groups. Then, for induced actions this property lifts \cite{Fl-Fr-Ku-Le}, and in particular (by taking the induced $\R$-action), if we have an automorphism then the corresponding suspension flow\footnote{By the {\em suspension flow} of $R$ we mean the special flow over $R$ under the constant function (equal to~1).}  has this lifted property. In particular, using induced $\Z$-actions (for $a\Z\subset\Z$), one can derive easily that for uniquely ergodic systems $(X,T)$ with the measure-theoretic AOP property we not only have M\"obius disjointness but also
\beq\label{e:ind}
\frac1N\sum_{n\leq N}f(T^nx)\mob(an+b)\tend{N}{\infty}0\eeq
for each $a,b\in\N$, $f\in C(X)$ and the convergence is uniform in $x$ \cite{Fl-Fr-Ku-Le}.\footnote{The same argument shows that if Sarnak's conjecture holds then~\eqref{e:ind} holds for each zero entropy $(X,T)$, $a,b\in\N$, $f\in C(X)$ uniformly in $x\in X$.}

\section{Glimpses of results on Sarnak's conjecture}\label{czesc5}
The cases for which the M\"obius disjointness has been proved, depend on the complexity of the deterministic system. They fit into two basic types. The first comes with sufficiently quantitative estimates for the disjointness sums which makes possible an analysis of the sums on primes yielding a PNT. This group includes Kronecker systems (Vinogradov \cite{Vi1}), nilsystems (Green and Tao \cite{Gr-Ta}) and, perhaps the most striking, the Thue-Morse system (Mauduit and Rivat \cite{Ma-Ri1}) which resolved a conjecture of Gelfond \cite{Ge}. When the systems are more complex, such as horocycles flows,\footnote{Horocycle flows are mixing of all orders, see \cite{Ma}.} then at least to date they do not come with a PNT,\footnote{In case of horocycle flows (Bourgain, Sarnak and Ziegler \cite{Bo-Sa-Zi}) Ratner's theorems on joinings are used and these provide no rate.} and for them the KBSZ criterion is used, in other words, the disjointness (perhaps in its weaker form, see Section~\ref{s:AOP}) is achieved.

We now review most of important cases in which M\"obius disjointness has been proved.

\subsection{Systems of algebraic origin}
\subsubsection{Horocycle flows}
Let $\Gamma\subset PSL_2(\R)$ be a discrete subgroup  with finite
covolume.\footnote{We will tacitly assume that $\Gamma$ is cocompact, so that the homogenous space $\Gamma\backslash PSL_2(\R)$ is compact and the system is uniquely ergodic by \cite{Fu0}; otherwise, as in the modular case when $\Gamma=PSL_2(\Z)$ we need to compactify our space. The proof of Theorem~\ref{t:bsz} in the modular case is slightly different than what we describe below.} Then the homogeneous space $X=\Gamma\backslash
PSL_2(\R)$ is the unit tangent bundle of a surface $M$ of constant
negative curvature. Let us consider the corresponding  \emph{horocycle flow}\footnote{We have $h_t(\Gamma x)=\Gamma\cdot\left(x\cdot\left[\begin{array}{ll}1&t\\
0&1\end{array}\right]\right)$ and $g_s(\Gamma x)=\Gamma\cdot\left(x\cdot\left[\begin{array}{ll}
e^{-s}&0\\0&e^s\end{array}\right]\right)$; we identify $g_s$ and $h_t$ with the relevant matrices.}
$(h_t)_{t\in\R}$ and the \emph{geodesic flow} $(g_s)_{s\in\R}$
on $X$. Since
\begin{equation}\label{horkom}
g_sh_tg_s^{-1}=h_{e^{-2s}t}\text{ for all }s,t\in\R,
\end{equation}
the  flows $(h_t)_{t\in\R}$ and $(h_{e^{-2s}t})_{t\in\R}$ are measure-theoretically
isomorphic for each $s\in\R$. In order to show that $T:=h_1$ is M\"obius disjoint,  the KBSZ criterion is used, and, given $x\in PSL_2(\R)$, one studies
limit points of $\frac1N\sum_{n\leq N}\delta_{(T^{pn}\Gamma x,T^{qn}\Gamma x)}$, $N\geq1$. Now, the celebrated Ratner's rigidity theorem \cite{Ra2}  tells us two important things: the point $(\Gamma x,\Gamma x)$ is generic for a measure $\rho$ (which must be a joining by unique ergodicity: $\rho\in J(T^p,T^q)$) and moreover this joining is ergodic.\footnote{The measure $\rho$ depends on $p,q$ and $x$ and it is so called algebraic measure, i.e.\ a Haar measure.} Again using Ratner's theory (cf.\ \cite{Ra1}) such joinings are determined by the commensurator $Com(\Gamma)$ of the lattice $\Gamma$:
$$Com(\Gamma):=\{z\in PSL_2(\R) :  z^{-1}\Gamma z\cap \Gamma\text{ has finite index in both }\Gamma\text{ and }z^{-1}\Gamma z\}.$$
Set $x_{p,q}:=x g_{\frac12\log(\frac pq)}x^{-1}(\infty)$.
The intersection of the stabilizer of $x_{p,q}$ with $Com(\Gamma)$ yields the correlator of $x_{p,q}$: it is a subgroup $C(\Gamma,x_{p,q})\subset\R_+^\ast$ and if $\rho$  is not the product measure then $\frac pq\in C(\Gamma,x_{p,q})$.
The careful analysis of the arithmetic and non-arithmetic cases done in \cite{Bo-Sa-Zi} shows that given $x\in PSL_2(\R)$,
$\frac pq\in C(\Gamma,x_{p,q})$ only for finitely many different primes $p,q$. Hence,
the joining $\rho$ has to be product measure for all but finitely many pairs $(p,q)\in\PP^2$ with $p\neq q$ which, by Theorem~\ref{t:kbsz}, yields the following:

\begin{Th}[\cite{Bo-Sa-Zi}]\label{t:bsz} All time-automorphisms of horocycle flows are M\"obius disjoint.\end{Th}

\begin{Remark} As noticed in \cite{Ab-Le-Ru1}, this is  \eqref{horkom} which yields the absence\footnote{To be compared with Remark~\ref{r:alliso};~the difference however is that when the ratio of $p$ and $q$ is close to~1, we can choose graph joinings in a compact set.} of AOP and makes the following questions of interest.\end{Remark}

\begin{Question}\label{q:horo} Do we have the MOMO property for horocycle flows? Are all uniquely ergodic models of horocycle flows M\"obius disjoint? Do we have uniform convergence in~\eqref{mdis1}?
\end{Question}

Since the method to prove M\"obius disjointness is through the KBSZ criterion (hence offers no rate of convergence), the following question is still open:

\begin{Question}[Sarnak]\label{q:PNThoro}
Do we have a PNT for horocycle flows?\end{Question}

For a partial answer, see \cite{Sa-Ub}, where it is proved that if $\Gamma x$ is a generic point for Haar measure $\mu_X$ of $X$ then any limit point of $\left(\frac1{\pi(N)}\sum_{p\leq N}\delta_{T^p\Gamma x}\right)$ is a measure which is absolutely continuous with respect to $\mu_X$.
\begin{Question} [Ratner]\label{q:Ratn} Are smooth time changes for horocycle flows M\"obius disjoint?\end{Question}

As smooth time changes of horocycle flows enjoy so called Ratner's property, the above question can be asked in  the larger context of flows possessing Ratner's property.
\paragraph{Added in September 2017:} In the recent paper \cite{Ka-Le-Ul}, a new criterion (of Ratner's type) for disjointness of different time-automorphisms of flows has been proved. The criterion applies for some classes of flows with Ratner's property, namely, in case of so called Arnold flows and for non-trivial smooth time changes of horocycle flows (in particular, the answer to Question~\ref{q:Ratn} is positive).

\subsubsection{Nilrotations, affine automorphisms}
Green and Tao in \cite{Gr-Ta} proved M\"obius disjointnes in the following strong form:

\begin{Th}[\cite{Gr-Ta}] Let $G$ be a simply-connected
nilpotent Lie group with a discrete and cocompact subgroup $\Gamma$. Let $p \colon \Z \to G$ be any is polynomial sequence\footnote{I.e.\ $p(n)=a_1^{p_1(n)}\ldots a_k^{p_k(n)}$, where $p_j\colon\N\to\N$ is a polynomial, $j=1,\ldots,k$. See, Section~6 in \cite{Gr-Ta07} for the equivalence with the classical definition of polynomials sequences in nilpotent Lie groups.} and $f\colon G/\Gamma\to \R$  a Lipschitz function.
Then $$\left|\frac1N\sum_{n\leq N} f(p(n)\Gamma)\mob(n)\right|={\rm O}_{f,G,\Gamma,A}\left(\frac N{\log^AN}\right)$$
for all $A > 0$.\end{Th}

In particular, by considering $T_g(x\Gamma)=gx\Gamma$, we see that all nilrotations are M\"obius disjoint with uniform Davenport's estimate~\eqref{vin}.

Also, a PNT holds for nilrotations:
Let $2=p_1<p_2<\ldots$ denote the sequence of primes.

\begin{Th}[\cite{Gr-Ta}, Theorem 7.1]  Assume that a nil-rotation $T_g$ is ergodic.\footnote{We assume that $G$ is connected.} Then, for every $x\in G$, we have
$$
\lim_{N\to\infty}
\frac1N\sum_{n\leq N} f(T_g^{p_n}x\Gamma) = \int_{G/\Gamma} f\,d\lambda_{G/\Gamma}$$
for all continuous functions $f\colon G/\Gamma \to [-1, 1]$.
\end{Th}

In \cite{Fl-Fr-Ku-Le}, it is proved that all nil-rotations enjoy the AOP property (hence all uniquely ergodic models of nil-rotations are M\"obius disjoint). In fact, the result is proved for all nil-affine automorphisms whose M\"obius disjointness has been established   earlier in \cite{Li-Sa}. Earlier, AOP has been proved for all quasi-discrete spectrum automorphism in \cite{Ab-Le-Ru2}, that is (following \cite{Ha-Pa}) for all unipotent affine automorphisms
$Tx=Ax+b$ of compact Abelian groups ($A$ is a continuous group automorphism and $b$ is an element of the group). The M\"obius disjointness of the latter automorphisms has been established still earlier in \cite{Li-Sa}.

The proof of the following corollary in \cite{Ab-Le-Ru2} shows that Furstenberg's proof~\cite{Fu00} (see e.g.\ \cite{Ei-Wa}) of Weyl's uniform distribution theorem can be adapted to the short interval version.

\begin{Cor}[\cite{Ab-Le-Ru2}]\label{c:wielsi}
  Assume that $\bfu\colon\N\to\C$, $\bfu\in\cm$. Then, for each non constant polynomial $P\in\R[x]$ with irrational leading coefficient, we have
  \[
    \frac{1}{M} \sum_{M\le m<2M} \frac{1}{H} \left| \sum_{m\le n<m+H} e^{2\pi iP(n)}\bfu(n)  \right|\tend{H,M}{\infty,H={\rm o}(M)}  0.\footnote{For degree~1 polynomials, the result is already in \cite{Ma-Ra-Ta}.}
  \]
\end{Cor}

Recall that a sequence $(a_n)\subset\C$ is called a {\em nilsequence} if it is a uniform limit of basic nilsequences, i.e.\ of sequences of the form $(f(T_g^nx\Gamma))$, where $f\in C(G/\Gamma)$ (here, we do not assume that $G/\Gamma$ is connected, neither that $T_g$ is ergodic).

\begin{Cor}[\cite{Fl-Fr-Ku-Le}]\label{c:nilsi} We have \[
    \frac{1}{M} \sum_{M\le m<2M} \frac{1}{H} \left| \sum_{m\le n<m+H} a_n\bfu(n)  \right|\tend{H,M}{\infty,H={\rm o}(M)}  0.
  \]
\end{Cor}

It has been  proved by Leibman \cite{Lei} that all polynomial multicorrelation sequences\footnote{More precisely, given an automorphism~$ T $ of a probability standard Borel space $\xbm$, we consider
$$ a_n=\int_X
    g_1\circ T^{p_1(n)}\cdot\ldots\cdot g_k\circ T^{p_k(n)}\,d\mu,$$
where~$ g_i\in L^\infty(X,\mu) $, $ p_i\in\Z[x] $, $ i=1,\ldots,k $ ($k\geq1$).}
are limits in the Weyl pseudo-metric of nil-sequences, all such polynomial sequences are orthogonal to $\mob$ on typical short interval, cf.\ Section~\ref{czesc6}.

The main problem connected with nilsequences is to prove the uniform version of convergence on short intervals as it is made precise in Conjecture~C of Tao (see Section~\ref{s:TAO} and also Frantzikinakis' proofs \cite{Fr}).

\subsubsection{Other algebraic systems}
For a more general zero entropy algebraic systems and their M\"obius disjointness we refer the reader to \cite{Pe}, where in particular the Ad-unipotent translation case is treated.

\subsection{Systems of measure-theoretic origin. Substitutions and interval exchange transformations}
\subsubsection{Systems whose powers are disjoint} We are interested in ergodic automorphisms $(Z,\cd,\kappa,R)$ for which (sufficiently large) prime powers $R^p$ are pairwise disjoint. Clearly, such automorphisms enjoy the AOP property. A typical automorphism has this property \cite{Ju} but there are also large classes of rank one (we detail on this class below) automorphisms with this property \cite{Ab-Le-Ru,Bo1,Ry}. Also minimal self-joining automorphisms \cite{Ju-Ru} enjoy this property. Chaika and Eskin in \cite{Ch-Es} show that for a.e.\ 3-interval exchange transformation (we detail on interval exchange transformations below) there are sufficiently many prime powers that are disjoint. It follows that all uniquely ergodic models of these automorphisms are M\"obius disjoint.


\subsubsection{Adic systems and Bourgain's criterion}
Let $(Z,\cd,\kappa,R)$ be a measure-theoretic system.
\begin{Def}\label{nm} In $(Z,\cd,\kappa,R)$, a Rokhlin {\em tower}  is a collection
    of disjoint measurable sets called {\em levels} $F$, $RF$, \ldots, $R^{h-1}F$.  If $Z$ is
    equipped with a partition $P$ such that each level $R^{r}F$ is contained in
    one atom
$P_{w(r)}$,
 the {\em name} of the
tower is the word
 $w(0)\ldots w(h-1)$.\end{Def}

  \begin{Def}\label{dr1} A system $(Z,\cd,\kappa,R)$  is of {\em rank one} if there exists a sequence of Rokhlin towers $(F_n,\ldots,R^{h_n-1}F_n)$, $n\geq1$, such that the whole $\sigma$-algebra is generated by the partitions $\{F_n, RF_n, \ldots, R^{h_n-1}F_n, X\setminus\bigcup_{j=0}^{h_n-1}R^jF_n\}$.
      \end{Def}

For topological systems, there is no canonical notion of rank, but the useful notion is that of adic presentation \cite{Ve-Li}, which we translate here from the original vocabulary into the one of Rokhlin towers.

\begin{Def}\label{dtt} An {\em adic presentation} of a topological system $(X,T)$ is given, for each $n\geq 0$, by a finite collection  $\mathcal Z_n$ of Rokhlin towers such that:
\begin{itemize} \item the levels of the towers in $\mathcal Z_n$ partition $X$,
\item each level of a tower in $\mathcal Z_n$ is a union of levels of towers in $\mathcal Z_{n+1}$,
\item the levels of the towers in $\bigcup_{n\geq 0}\mathcal Z_n$ form a basis of the topology of $X$.
\end {itemize}
\end{Def}
In that case, the towers of $\mathcal Z_{n+1}$ are built from the towers of $\mathcal Z_n$ by  cutting and stacking, following recursion rules: a given tower in $\mathcal Z_{n+1}$ can be built by taking columns of successive towers in $\mathcal Z_n$ and stacking them successively one above  another. These rules are best seen by looking at
 the partition $P$ into levels of the towers in $\mathcal Z_0$; possibly replacing $Z_0$ by some $Z_k$, we can always assume $P$ has at least two atoms. The names of the towers in $\mathcal Z_n$ form sets of words $\mathcal W_n$, and the cutting and stacking of towers gives a canonical decomposition of every $W\in \mathcal W_n$:
$$W=W_1^{k_1}\cdots W_r^{k_r}$$
for $r$ words $W_i\in\mathcal{W}_{n-1}$, $1\leq i\leq r$, integers  $k_1$, \ldots , $k_r$;  all these parameters depend on the word $W$. These decompositions are called the {\em rules of cutting and stacking} of the system. \\

The following result is an improvement on Theorem 3.1 of \cite{Fe-Ma}, which itself can be found in \cite{Bo1}, though it is not completely explicit in that paper (it is stated in full only in a particular case, as Theorem 3, and its proof is understated). The following effective bound stems from a closer reading of \cite{Bo1}:

 \begin{Th}\label{cbou}
 Let $(X,T)$  be a  topological dynamical system
 admitting an adic presentation, as in Definition \ref{dtt} and the comment just after.

\noindent Suppose that  for any $n$ and $W$ in $\mathcal W_n$, we have:
\begin{itemize}
\item in the rules of cutting and stacking $r\leq C$, with $C\geq 2$,
\item  if we decompose $W$ into words $W_\ell\in \mathcal W_{n-s}$ by iteration of the rules of cutting and stacking then  for all $\ell$ and $s$ large enough, we have $$\ab W\ab >C^{200s} \ab W_\ell\ab.$$
\end{itemize}
Then  $(X,T)$ is M\"obius disjoint.

 If such a system is uniquely ergodic and weakly mixing for its invariant probability, it satisfies also the following PNT: for any word $W=w_1\ldots w_N$ which is a factor of a word in any $\mathcal W_n$, we have $$\sum_{i=1}^N\boldsymbol{\Lambda}(i)w_i=\sum_{i=1}^Nw_i+{\rm o}(N).$$
\end{Th}
\begin{proof}
We look at Theorem 2 of \cite{Bo1}. It requires
 a stronger assumption, denoted by relations (2.2) and (2.3) in p.~119 of \cite{Bo1}, which is indeed the assumption of the present theorem with the estimate $C^{200s}$ replaced by $\beta(s)$ for some function satisfying $\frac{\log \beta(s)}{s}\to \infty$ when $s\to \infty$ (note that the assumption in \cite{Bo1} that the words $W_n$ are on the alphabet $\{0,1\}$ is not used in the proof, which works for any finite alphabet). Then this theorem gives, for any word $w_1\cdots w_N$ in some $\mathcal W_m$ and $N$ large enough, an  estimate for $$\int \left| \sum_{n\leq N} w_ne^{2\pi in\theta}\right|
 \left|\sum_{n\leq N} \mob(n)e^{2\pi in\theta}\right| d\theta,$$ and this, through the
 relation (1.62) on p.~118, implies that $\sum_{n\leq N} w_n\mob (n)={\rm o}(N)$.

 Lacking space to rewrite the extensive computations in \cite{Bo1}, we explain how to weaken the hypothesis. First, as suggested in the remark at the beginning of Section 2,  p.~119, of that paper, we replace $\beta(s)$ by $C_0^s$ for some constant $C_0$, as yet unknown (the $C_0s$ written in the same p.~119 is a misprint).  The relations (2.2) and (2.3) are used twice in the course of the proof: first, to get the relation (2.15), namely $$\left(C\frac{\log \ab W\ab}{n}\right)^s<\ab W\ab ^{\epsilon}$$ for a word $W in \mathcal W_n$, and then to get
  the estimate (2.42), which states that $$\left(C\frac{\log K}{s}\right)^s<K^{\epsilon},$$ where $s$ is the number of stages such that a word of length $N$ in
   $\mathcal W_n$ is divided into words of  $\mathcal W_{n-s}$, of lengths in the order of $\frac{N}{K}$. Under our hypothesis, in the first case, $\ab W \ab$ is in $C_0^n$, and in the second case $K$ is in $C_0^s$. Thus both (2.15) and (2.42) are implied by the relation
    $$\frac{\log\log C_0+\log C}{\log C_0}<\epsilon.$$
  The value of $\epsilon$ is dictated by relation (2.49), which requires $Q^{\epsilon}K^{\epsilon}(Q+K)^{-\frac{1}{4}}\leq (Q+K)^{-\frac{1}{5}}$ for some large numbers $Q$ and $K$, thus we can take $\epsilon=\frac{1}{20}$. Then $\frac{\log\log C_0}{\log C_0}$ will be bounded if $C_0$ is large enough independently of $C$, while to bound $\frac{\log C}{\log C_0}$ we need to take $C_0=C^a$; as $C\geq 2$, we see that $a=200$ is convenient for the sum of the two terms.

Now, if we replace $w_n$ by  $ u (n) = f (T^n(x_0)) $, because of Definition \ref{dtt} above, we can first assume that $f$ is constant on all levels of the towers of some stage $m$, and then conclude by approximation. Such an $f$ is also constant on all levels of all towers at stages $q>m$; fixing $x_0$ and $N$, except for some initial values $u(1)$ to $u(N_0)$ where $N_0$ is much smaller than $N$, we can replace $u(n)$ by $w'_{n}$, where $w'_n$ is the value of $f$ on the $n$-th level of some tower with name $W$ in some $\mathcal W_q$ for $q\geq m$. Then the $w'_1\cdots w'_N$ are built by the same induction rules as the $w_1\cdots w_N$, and the estimates using the $w'_n$ are computed as those using the $w_n$ in the proof of Theorem 2 of \cite{Bo1}, thus we get the same result.

 The PNT is  in (3.4),  (3.7), (3.14)  of \cite{Bo1}  ((3.14)  is proved for the particular case of $3$-interval exchanges but holds in the same way for the more general case).
 \end{proof}

Of course, the value of $C_0$ could be improved, but we need it to be at least some power of $C$.

\subsubsection{Substitutions}

We start with some basic notions.
\begin{Def} A {\em substitution} $\sigma$ is an application from
an alphabet $\mathcal A$ into the set ${\mathcal A}^{\star}$ of finite words on
$\mathcal A$; it
extends to a morphism of ${\mathcal A}^{\star}$ for the concatenation. A {\em fixed point} of $\sigma$ is an infinite sequence $u$ with $\sigma u=u$. The  associated symbolic dynamical system $(X_{\sigma},S)$ is $(X_u,S)$ for a fixed point $u$.

Substitution $\sigma$ has {\em  constant length} $q$ if $|\sigma a|=q$ for all $a$ in $\mathcal A$.

The {\em Perron-Frobenius eigenvalue} is the largest eigenvalue of the matrix giving the number of occurrences of $j$ in $\sigma a$. A substitution $\sigma$ is {\em primitive} if a power of this matrix has strictly positive entries. \end{Def}

For the class of constant length substitutions, there have been a lot of partial results on  M\" obius orthogonality:
\begin{itemize}
\item
First for the most famous example, the Thue-Morse substitution $0\to 01$, $1\to 10$, with
Indlekofer and  K\'atai \cite{In-Ka}, Dartyge and  Tenenbaum \cite{Da-Te}, Mauduit and Rivat \cite{Ma-Ri1}, El Abdalaoui, Kasjan and Lema\'nczyk \cite{Ab-Ka-Le}.\footnote{In \cite{In-Ka,Da-Te,Ma-Ri1} it is proved that the sequence $(-1)^{u(n)}$, $n\geq1$ is orthogonal to $\mob$.}
\item
 The  case of the Rudin - Shapiro substitution  $0\to 01$, $1\to 02$, $2\to 31$, $3\to 32$ was solved by   Mauduit and Rivat \cite{Ma-Ri2}. Then  Drmota \cite{Dr}, Deshouillers,  Drmota and  M\" ullner \cite{De-Dr-Mu},
\item
 Ferenczi, Ku\l aga-Przymus, Lema\'nczyk and Mauduit \cite{Fe-Ku-Le-Ma} extended these results to various subclasses of systems of arithmetic origin.\footnote{While in \cite{Dr,Fe-Ku-Le-Ma} M\"obius disjointness is proved for the dynamical systems given by bijective substitutions, \cite{De-Dr-Mu} treats the opposite case, so called synchronized. As noted in \cite{Be-Ku-Le-Ri}, this leads to dynamical systems given by rational sequences and  such are M\"obius disjoint. Note also that for the synchronized case, once the system is uniquely ergodic, it is automatically a uniquely ergodic model of an automorphism with discrete spectrum, cf.\ Corollary~\ref{c:mn3} and Remark~\ref{r:momorot}.}
\end{itemize}
See also \cite{Ma-Ma-Ri1, Ma-Ma-Ri2}  for a PNT for some digital functions.

But all this was superseded by the general result of M\"ullner \cite{Mu},  whose proof uses the arithmetic techniques of \cite{Ma-Ri2} together with a new structure theorem on the underlying {\em automata}:

\begin{Th}[\cite{Mu}] For any substitution of constant length,
the associated symbolic system is M\"obius disjoint. Moreover,  a PNT holds if the substitution is primitive. \end{Th}

The substitutions which are not of constant length are much less known:
\begin{itemize}
\item
The most famous example is the Fibonacci substitution, $0\to 01$, $1\to 0$: in that case, the associated symbolic system is a coding of an irrational rotation, hence it is M\"obius disjoint as a uniquely ergodic model of a discrete spectrum automorphism, see Section~\ref{s:dissp}.
\item
 Drmota, M\"ullner and Spiegelhofer~\cite{Dr-Mu-Sp} have just shown M\"obius disjointness for a new example, a substitution which generates $(-1)^{s_{\phi(n)}}$, where $s_{\phi(n)}$ is the Zeckendorf sum-of-digits function.\footnote{This example has partly continuous spectrum.}
\item
 Also,  we can exhibit a small subclass of examples which are M\"obius disjoint, by a straightforward translation of Bourgain's criterion above:
\end{itemize}

\begin{Th} Suppose that $\sigma$ is a primitive substitution  satisfying
\begin{itemize}
\item  for all $i\in \mathcal A$, $\sigma i=(j_1(i))^{a_1(i)}\ldots (j_{q_i}(i))^{a_{q_i}(i)}$, , $a_1(i)\in \mathcal A$, \ldots, $a_{q_i}(i)\in \mathcal A$, $q_i\leq C$
(this can be expressed as: the multiplicative length of $\sigma$ is smaller than $C$),
\item the Perron-Frobenius eigenvalue of $\sigma$ is larger than $C^{200}$;
\end{itemize}
then the associated symbolic dynamical system is M\"obius disjoint. If $(X_{\sigma}, S)$ is  weakly mixing, the fixed points satisfy a PNT.\end{Th}
\begin{proof}
If all fixed points are periodic, the result is trivial. If $\sigma$ has
 a non-periodic fixed point, it is well known (and proved by the methods of \cite{Qu} together with the recognizability result of \cite{Mo}) that the system has an adic presentation, where the names of the towers in
 $Z_n$ are the words $\sigma^na$, $a\in \mathcal A$. Thus the results come from Theorem \ref{cbou} above and the properties of the matrix of $\sigma$.
\end{proof}

\begin{Example}
Here are some substitutions for which the above theorem applies, with a PNT: $0\to 0^{k+1}12$, $1\to 12$, $2\to 0^k12$, $k+2>3^{200}$.
\end{Example}

\begin{Question} Are dynamical systems associated to substitutions M\"obius disjoint?\footnote{One can also ask about M\"obius disjointness of related systems as tiling systems.}
\end{Question}

\subsubsection{Interval exchanges}
\begin{Def}A {\em $k$-interval exchange} with probability vector $(\alpha _1,\alpha _2,\ldots ,\alpha _k)$,
 and
permutation $\pi$ is
defined by
$$
Tx=x+\sum_{\pi^{-1}(j)<\pi^{-1}(i)}\alpha_{j}-\sum_{j<i}\alpha_{j}.
$$
when $x\in \Delta_i=\left[ \sum_{j<i}\alpha_{j}
,\sum_{j\leq i}\alpha_{j}\right).$
\end{Def}

Exchanges of $2$ intervals are just rotations, thus M\"obius disjointness holds for them by the Prime Number Theorem (on arithmetic progressions)  when the rotation is rational and from a result of Davenport \cite{Da}  -- using a result of Vinogradov \cite{Vi} -- when the rotation is irrational, cf.~\eqref{vin} in Introduction.

Then \cite{Bo1} exhibits exchanges of $3$ intervals which are M\"obius disjoint, with a PNT if weak mixing holds: these use the criterion developed in Theorem~\ref{cbou} above, together with the adic presentation built in \cite{Fe-Ho-Za}. Generalizing these methods, it is shown in \cite{Fe-Ma} that M\"obius disjointness holds for examples of exchanges of $k$ intervals for every $k\geq 2$ and every {\em Rauzy class}, with a PNT in the weak mixing case. A breakthrough came with \cite{Ch-Es}, for a large subclass of exchange of~3 intervals:


\begin{Th}
[\cite{Ch-Es}]
For (Lebesgue)-almost all $(\alpha_1,\alpha_2)$, Sarnak's conjecture holds for exchanges of $3$ intervals with permutation $\pi i=3-i$ and probability vector $(\alpha_1,\alpha_2,1-\alpha_1-\alpha_2)$.\end{Th}

To prove this, Chaika and Eskin use first  the well-known fact that such an exchange of $3$ intervals, denoted by $T$, is the induced map of the rotation of angle $\alpha=\frac{1-\alpha_1}{1+\alpha_2}$ on the interval $[0,x)$ where $x=\frac{1}{1+\alpha_2}$. This approach, of course, does not generalize to $4$ intervals or more.

In fact, in \cite{Ch-Es} two different results are proved. In the easier one, they deduce M\"obius disjointness from   the disjointness of powers of $T$; they  give a sufficient condition for $T^m$ to be disjoint from  $T^n$ for all $m\neq n$, which is satisfied by almost all these $T$. Namely, if we take $(a_1,\ldots)$ to be the continued fraction of $\alpha$ and  $(b_1,\ldots)$ the {\em $\alpha$-Ostrowski} expansion of $x$, then it is enough that, for any ordered $k$-tuple of pairs $((c_1,d_1),\ldots (c_k,d_k))$ of natural numbers such that $d_i\leq c_i-1$, there are infinitely many $i$ with $a_i=c_1$,\ldots, $a_{i+k-1}=c_k$,  $b_i=d_1$,\ldots, $b_{i+k-1}=d_k$.

Then most of the paper is used to give an explicit Diophantine condition on $\alpha$ and $x$, which implies a slightly weaker property than the disjointness of powers. Under that condition, there exists a constant $C$ such that for all $n$, and $0\leq m\leq n$, $T^m$ is disjoint from $T^n$ except maybe when $m$ belongs to a sequence $m_i(n)$ in which any two consecutive terms satisfy $m_{i+1}(n)>Cm_i(n)$, and this is proved to imply M\"obius disjointness. The Diophantine condition holds for almost all $T$, and, as it is long, we refer the reader to Theorem 1.4 of \cite{Ch-Es}; it expresses the fact that the geodesic ray from a certain flat torus with two marked points, defined naturally from $T$ and its inducing rotation,  spends significant time in compact subsets of the space of such tori.

\subsubsection{Systems of rank one}
These systems form a measure-theoretic class defined in Definition \ref{dr1} above. It is well known, but has been shown explicitly for all cases only  in the recent \cite{Ad-Fe-Pe}, that each system of rank-one is measure-theoretically isomorphic to one of the topological systems we define now.

\begin{Def}\label{dsc}
A {\em standard model of rank one} is
the shift on the orbit closure of the sequence $u$ which, for each $n\geq 0$, begins with the word $B_n$ defined recursively by concatenation as follows.
We take sequences of
positive integers $q_n, n\geq 0$, with $q_n>1$ for infinitely many $n$,  and $a_{n,i}, n \geq 0,
0\leq i \leq q_n-1$, such that, if $h_n$ are defined by
 $h_0=1$,
$h_{n+1}=q_nh_n+\sum_{j=0}^{q_n-1}a_{n,i}$,
then
$$\sum_{n=0}^{\infty}{{h_{n+1}-q_nh_n}\over{h_{n+1}}} <\infty.$$
We define $B_0=0$,
$$B_{n+1}=B_n 1^{a_{n,0}}B_n \dots B_n 1^{a_{n,q_n-1}}$$ for $n \geq 0$.
\end{Def}

In \cite{Bo1}, Bourgain proved M\"obius disjointness for a standard model of rank one if both the  $q_n, n\in {\mathbb N}\cup \{0\}$   and $a_{n,i}, n \in {\mathbb N}\cup \{0\}$, are bounded by some constant $C$ (we will refer to this as to a {\em bounded rank one construction}).

Note however that, in the same paper, the  half-hidden  criterion deduced from Theorem~2 or~3, see Theorem~\ref{cbou}  above, is much more than an auxiliary to prove the supposedly main Theorem~1 of \cite{Bo1}; it applies  to a much wider class of systems,  and even for some famous rank one systems  this criterion works while Theorem~1  does not apply.

Bourgain's  result was improved in \cite{Ab-Le-Ru}, where so called {\em recurrent rank one constructions} are considered with a stabilizing bounded subsequence of spacers (that is, of a subsequence of $(a_{n,i})$).\footnote{Moreover, M\"obius disjointness is established for some  other famous classes of rank one transformations such as: Katok's $\alpha$-weak mixing class (these are a special case of three interval exchange maps) or rigid generalized Chacon's maps.} One of main tools in \cite{Ab-Le-Ru} is a representation of each rank one transformation as an integral automorphism over an odometer with so called Morse-type roof function which goes back to \cite{Ho-Me-Pa}. See also \cite{Ry} for a simpler proof of a generalization of Bourgain's result to a class of partially bounded rank one constructions. 

\paragraph{Spectral approach}

In order to prove M\"obius disjointness for standard models of rank one transformations, both papers \cite{Bo1} and \cite{Ab-Le-Ru} use a spectral approach. In \cite{Ab-Le-Ru}, unitary operators $U$ (of separable Hilbert spaces) are considered and  weak limits  of powers $(U^{pm_k})$ (for different primes $p$) are studied. Once such  limits yield
sufficiently different (for different $p$) analytic functions (of $U$), the powers $U^p$ and $U^q$ are spectrally disjoint.\footnote{Hence, $T^p$ and $T^q$ are disjoint in Furstenberg's sense, and, in fact, we even have AOP.} If for a positive real number $a$ we set $s_a(x)=ax$ mod~1 on the additive circle $\T=[0,1)$, then the above spectral disjointness means that
\beq\label{e:spe1}
\sigma^{(p)}:=(s_p)_\ast(\sigma)\text{ are mutually singular for different }p\in\PP,
\eeq
where $\sigma=\sigma_U$ stands for the maximal spectral type of $U$.

In \cite{Bo1}, a different spectral approach (sufficient for a use of the KBSZ criterion, hence, sufficient for M\"obius disjointness) is used. Namely,  if $r\geq1$ is an integer, then 
by $\sigma_r$, we will denote the measure which is obtained first by taking the image of $\sigma$ under the map $x\mapsto\frac1rx$, i.e.\ the measure $\sigma^{(1/r)}$, and then repeating this new measure periodically in intervals $[\frac jr,\frac{j+1}r)$, that is:
$$
\sigma_r:=
\frac1r\sum_{j=0}^{r-1}\sigma^{1/r}\ast\delta_{j/r}.$$
Bourgain~\cite{Bo1} uses a representation of the maximal spectral type of a rank one transformation as a Riesz product and then shows the mutual disjointness of measures $\sigma_p$ and $\sigma_q$ for different $p,q\in\PP$ (for more information about the measures $\sigma_r$, see e.g.\ \cite{Qu}, p.\ 196). Although, there seems not to be too much relation between the measures $\sigma^{(r)}$ and $\sigma_r$, the following observation\footnote{This has been proved, e.g.\ in an unpublished preprint of El Abadalaoui, Ku\l aga-Przymus, Lema\'nczyk and de la Rue.} explains some equivalence of these both spectral approaches:

\begin{Lemma}\label{l:bourgaintype}
Assume that $\sigma$ and $\eta$ are two probability measures on the circle. Then:
\begin{enumerate}[(a)]
\item
if $\sigma^{(r)}\perp\eta^{(s)}$ then $\sigma_s\perp \eta_r$;
\item
 if $(r,s)=1$ then $\sigma^{(r)}\perp\eta^{(s)}$ if and only if $\sigma_s\perp \eta_r$.
\end{enumerate}
\end{Lemma}

\subsubsection{Rokhlin extensions} Let $T$ be a uniquely ergodic homeomorphism of a compact metric space and let $f\colon X\to \R$ be continuous. Set $T_f(x,t):=(Tx,f(x)+t)$ to obtain a skew product homeomorphism on $X\times\R$. Note that the latter space is not compact. But, if we take any continuous flow $\cs=(S_t)_{t\in\R}$ acting on a compact metric space $Y$ then the skew product
$T_{f,\cs}$ acting on $X\times Y$ by the formula:
$$
T_{f,\cs}(x,y)=(Tx,S_{f(x)}(y)),\;(x,y)\in X\times Y$$
is a homeomorphism of the compact space $X\times Y$ and it
is called a {\em Rokhlin extension} of $T$. To get a good theory, usually one has to put some further assumptions on $f$ (considered as a cocycle taking values in a locally compact but not compact group, see e.g.\ \cite{Le-Pa,Sch}).
It is proved in \cite{Ku-Le1} that there are irrational rotations $Tx=x+\alpha$ and continuous $f\colon\T\to\R$ (even smooth) such that $\tfs$ has the AOP property for each uniquely ergodic $\cs$.\footnote{If $\cs$ preserves a measure $\nu$ then $\tfs$ preserves measure $\mu\ot\nu$, the AOP property is considered with respect to this measure.}

We would like to emphasize that the Rokhlin skew product construction are usually relatively weakly mixing \cite{Le-Pa}, so the class we consider here is drastically different from the distal class which is our next object to give account.

This approach leads in \cite{Ku-Le1} to so called random sequences\footnote{Such a sequence $(a_n)$ is of the form $(\va^{(n)}(x))$  with $\va^{(n)}(x)=\va(x)+\va(Tx)+\ldots+\va(T^{n-1}x)$, $n\geq0$.} $(a_n)\subset\R$ such that
$$
\frac1N\sum_{n\leq N}g(S_{a_n}y)\mob(n)\to0$$
for each uniquely ergodic flow $\cs$ acting on a compact metric space $Y$, each $g\in C(Y)$ and (due to \cite{Ab-Ku-Le-Ru1}) uniformly in $y\in Y$.

\subsection{Distal systems}
Assume that $R$ is an ergodic automorphism of a probability standard Borel space $(Z,\cd,\kappa)$. $R$ is called (measurably) {\em distal} if it can be represented as transfinite sequence of consecutive isometric extensions, where in case of a limit ordinal, we take the corresponding inverse limit (i.e.\ we start with the one-point dynamical system, the first isometric extension is a rotation and then we take a further isometric extension of it etc.). Recall  that by a {\em separating sieve} we mean a sequence
$$
Z\supset A_1\supset A_2\supset\ldots\supset A_n\supset\ldots$$
of sets of positive measure such that $\mu(A_n)\to0$ and there exists $Z_0\subset Z$, $\mu(Z_0)=1$, such that for each $z,z'\in Z_0$ if for each $n\geq1$ there is $k_n\in\Z$ such that $R^{k_n}z,R^{k_n}z'\in A_n$, then $z=z'$.
A theorem by Zimmer \cite{Zi1} says that $T$ is distal if and only if it has a separating sieve.

Distal automorphisms play a special role in ergodic theory: each automorphism has a maximal distal factor and is relatively weakly mixing over it \cite{Fu1,Zi1,Zi2}. Hence, many problems in ergodic theory can be reduced to study the two opposite cases: the distal and the weak mixing one.\footnote{See the most prominent example of such a reduction, namely,  Furstenberg's ergodic proof of Szemer\'edi theorem on the existence of arbitrarily long arithmetic progressions in subsets of integers of positive upper Banach density \cite{Fu1}.} Recall that distal automorphisms have entropy zero.

There is also a notion of distality in topological dynamics. A homeomorphism $T$ of a compact metric space $X$ is called {\em distal} if the orbit $(T^nx,T^nx')$, $n\in\Z$, is bounded away from the
diagonal in $X \times X$  for each $x \neq x'$. Some of topologically distal classes already appeared in previous sections. Indeed,
zero entropy (minimal) affine transformations are examples of distal homeomorphisms. Another natural class of distal (uniquely ergodic) homeomorphisms is given by nil-translations and, more generally, affine unipotent diffeomorphisms of nilmanifolds.
A theorem by Lindenstrauus \cite{Li} says that a measurably distal automorphism $R$ has a minimal\footnote{In general, there is no uniquely ergodic model $(X,T)$ of $R$ with $T$ topologically distal.} model $(X,T)$ together with $\mu\in M^e(X,T)$  of full support (and $(X,\mu,T)$ is isomorphic to $(Z,\kappa, R)$) in which $T$ is topologically distal.

The following (still open) question seems to be a natural and important step in proving Sarnak's conjecture:

\begin{Question}[Liu and Sarnak \cite{Li-Sa}]
Are all topologically distal systems M\"obius disjoint?
\end{Question}

As transformations with discrete spectrum are measurably distal and Theorem~\ref{dissp} holds, we can of course ask whether given a measurably distal automorphism, all of its uniquely ergodic models
are M\"obius disjoint.\footnote{As a matter of fact, such a question remains open even for 2-point extensions of irrational rotations.}

We now focus on the famous class of Anzai skew products. This is the class of transformations defined on $\T^2$ by the formula:
$$
T_{\va}\colon\T^2\to\T^2,\;T_{\va}(x,y)=(x+\alpha,\va(x)+y).
$$
In other words, Anzai skew products are given by $Tx=x+\alpha$ an irrational rotation on the (additive) circle, and a measurable  $\va\colon\T\to\T$; the skew product $T_{\va}$ preserves the Lebesgue measure. If $\va$ is continuous, $T_{\va}$ is a homeomorphism of $\T^2$. If we cannot solve the functional equations \beq\label{equa1}k\va(x)=\xi(x)-\xi(Tx)\eeq ($k\in \N$) in continuous functions $\xi\colon\T\to\T$, then $T_{\va}$ is minimal, but if for one $k\in\N$ we have a measurable solution then $T_{\va}$ is not uniquely ergodic. In \cite{Li-Sa}, we find examples of Anzai skew products which are minimal not uniquely ergodic but are M\"obius disjoint,\footnote{As a matter of fact, in \cite{Ab-Ku-Le-Ru1} it is proved that if a uniquely ergodic homeomorphism $T$ satisfies the strong MOMO property (see Definition \ref{def:SM} on page \pageref{def:SM}) and (continuous) $\va\colon X\to G$ ($G$ is a compact Abelian group) satisfies $\va:=\xi-\xi\circ T$ has a measurable solution $\xi:X\to G$, then the homeomorphism $T_{\va}$ of $X\times G$ is M\"obius disjoint. This applies  if~\eqref{equa1} has a measurable solution for $k=1$. It is however an open question whether we have M\"obius disjointness when there is no measurable solution for $k=1$ but there is such a solution for some $k\geq2$.} moreover it is proved that if $\va$ is analytic with an additional condition on the decay (from below) of Fourier coefficients then $T_{\va}$ is M\"obius disjoint for each irrational $\alpha$. In \cite{Ku-Le}, it is proved that if $\va$ is of class $C^{1+\delta}$ then for a typical (in topological sense) $\alpha$, we have M\"obius disjointness of $T_{\va}$.\footnote{It follows from a subsequent paper \cite{Ku-Le1} that the Anzai skew products considered in \cite{Ku-Le} enjoy the AOP property.} A remarkable result is proved by Wang \cite{Wa}: all analytic Anzai skew products are M\"obius disjoint.  The proofs in all these papers are using Fourier analysis techniques but in \cite{Wa}, it is also a short interval argument from \cite{Ma-Ra-Ta} used in one crucial case.

Nothing seems to be proved about a PNT in the class of distal systems (except for rotations).

\subsubsection{Discrete spectrum automorphisms} \label{s:dissp}
The simplest examples of (measurably) distal automorphisms are those with discrete spectrum. Recall that a measure-theoretic system $(Z,\cd,\kappa,R)$ is said {\em to have discrete spectrum} if the $L^2$-space is generated by the eigenfunctions of the Koopman operator $Tf:=f\circ T$. The classical Halmos-von Neumann theorem tells us that each ergodic automorphism with discrete spectrum has a uniquely ergodic model being a rotation on a compact Abelian (monothetic) group.

\begin{Th}\label{dissp}
All uniquely ergodic models of automorphisms with discrete spectrum are M\"obius disjoint.\end{Th}

This result was first proved in \cite{Ab-Le-Ru2} for totally ergodic discrete spectrum automorphisms (as they have the AOP property) and in full generality by Huang, Wang and Zhang in \cite{Hu-Wa-Zh}. In fact, the latter result is stronger:

\begin{Th}[\cite{Hu-Wa-Zh}]\label{dissp1} Let $(X,T)$ be a dynamical system, $x\in X$ and $N_i\to\infty$. Assume that
$\frac1{N_i}\sum_{n\leq N_i}\delta_{T^nx}\tend{i}{\infty} \mu$. Assume that $\mu$ is a convex combination of countably many ergodic measures, each of which yields a system with discrete spectrum. Then $\lim_{i\to\infty}\frac1{N_i}\sum_{n\leq N_i}f(T^nx)\mob(n)=0$ for each $f\in C(X)$.\end{Th}

Note that Theorem~\ref{dissp} also follows from Theorem~\ref{thmA} because ergodic rotations enjoy the strong MOMO property \cite{Ab-Ku-Le-Ru1} (see Remark~\ref{r:momorot}). As a matter of fact, as we have already noticed in Corollary~\ref{c:mn3}, Theorem~\ref{dissp} follows from~\cite{Ma-Ra-Ta}.

\subsection{Sub-polynomial complexity}
Let $T$ be a homeomorphism of a compact metric space $(X,d)$ and let $\mu\in M(X,T)$. Assume also that $a\colon\N\to\R$ is increasing with $\lim_{n\to\infty}a(n)=\infty$. In the spirit of \cite{Fe10}, we say that the measure complexity of $\mu$ is weaker than $a$ if
$$
\liminf_{n\to\infty}\frac{\min\{m\geq1:
\mu(\bigcup_{j=1}^mB_{d_n}(x_j,\vep))>1-\vep\text{ for some } x_1,\ldots,x_m\in X\}}{a(n)}=0$$
for each $\vep>0$ (here $d_n(y,z)=\frac1n\sum_{j=1}^nd(T^jy,T^jz)$).

The main result of the recent article \cite{Hu-Wa-Ye} states the following:

\begin{Th}[\cite{Hu-Wa-Ye}] \label{t:hwy} If $(X,T)$ is a topological system for which all its invariant measures have sub-polynomial complexity, i.e.\ their complexity is weaker than $n^\delta$ for each $\delta>0$, then $(X,T)$ is M\"obius disjoint.\end{Th}

As shown in \cite{Hu-Wa-Ye}, Theorem~\ref{t:hwy} applies to: topological systems whose all invariant measures yield systems with discrete spectrum (cf.\ Corollary~\ref{c:mn3}), Anzai skew products of $C^\infty$-class (over each irrational rotation), $K(\Z)$-sequences introduced by Veech \cite{Ve} and tame systems.\footnote{For the latter two classes all invariant measures yield discrete spectrum.}

\subsection{Systems of number-theoretic origin}
Recall that a sequence $x\in\{0,1\}^{\N}$ is called a {\em generalized Morse sequence}~\cite{Ke} if
\beq\label{defm}
x=b^0\times b^1\times \ldots\eeq
with $b^i\in\{0,1\}^{\ell_i}$, $\ell_i\geq2$, $b^i(0)=0$ for each $i\geq0$.\footnote{If $B\in\{0,1\}^k$ and $C=C(0)C(1)\ldots C(\ell-1)\in\{0,1\}^{\ell}$ then we define
$
B\times C:=(B+C(0))(B+C(1))\ldots (B+C(\ell-1))$.}
The following question still remains open.

\begin{Question}[Mauduit (2014)]\label{q:CM}
Are dynamical systems arising from generalized Morse sequences M\"obius disjoint?\end{Question}

Consider  the simplest subclass of the class of generalized Morse sequences, for which in~\eqref{defm} we have $|b^i|=2$ for all $i\geq0$ (in other words, either $b^i=01$ or $b^i=00$). Such sequences are called {\em Kakutani sequences} \cite{Kw}. A particular case of Sarnak's conjecture, namely:
\beq\label{kk}
\frac1N\sum_{n=1}^N(-1)^{x(n)}\mob(n)\to0,\eeq
for the classical Thue-Morse sequence $x=01\times01\times\ldots$ follows from \cite{In-Ka,Ka} (see also \cite{Da-Te} where, additionally, the speed of convergence to zero is given and \cite{Ma-Ri1}, where, additionally, a PNT has been proved).
Then~(\ref{kk}) has been proved for some subclass of Kakutani sequences in \cite{Gr}. As a matter of fact, in~\cite{Gr},   the problem whether $\frac1N\sum_{n=1}^N(-1)^{s_E(n)}\mob(n)\to 0$ is considered. Here
$E\subset\N$ is fixed and $s_E(n):=\sum_{i\in E}n_i$, where $n=\sum_{i=0}^\infty n_i2^i$ ($n_i\in\{0,1\}$). To see a relationship with Kakutani sequences define a Kakutani sequence $x=b^0\times
b^1\times\ldots$ with $b^n=01$ iff $n+1\in E$; it is now not hard to see that $s_E(n)=x(n)$ mod~2.  Finally, using some methods from~\cite{Ma-Ri1}, Bourgain \cite{Bo0} completed the result from \cite{Gr} so that~(\ref{kk}) holds in the whole class of Kakutani sequences (moreover, in \cite{Bo0,Gr}  a relevant PNT has been proved).
One can show that the methods used in the aforementioned papers allow us to have~(\ref{kk}) with $x$ replaced by every $y\in\co(x)$ (as shown in \cite{Fe-Ku-Le-Ma} in Lemma~6.5 therein, this can be sufficient to show M\"obius disjointness  for the simple spectrum case; for example, this approach works for the Thue-Morse system).

The problem of M\"obius disjointness is also studied (and solved) in the class of (generalized) Kakutani sequences  taking values in  compact  (even non-Abelian) groups, see \cite{Ve}.

\subsection{Other research around Sarnak's conjecture} As all periodic observable sequences are orthogonal to $\mob$, one could think that a limit of periodic constructions of type of Toeplitz sequences\footnote{A sequence $x\in A^{\N}$ is called Toeplitz if for each $n\in\N$ there is $q_n\in\N$ such that $x(n+jq_n)=x(n)$ for each $j=0,1,\ldots$} also yields systems that are M\"obius disjoint.\footnote{So called regular Toeplitz sequences are treated in \cite{Ab-Ka-Le} and \cite{Do-Ka}, these are however uniquely ergodic models of odometers.} However, in \cite{Ab-Ka-Le} (and then \cite{Do-Ka}) there are examples of Toeplitz systems which are not M\"obius orthogonal. These examples have positive entropy \cite{Ab-Ku-Le-Ru, Do-Ka}. Karagulyan in \cite{Kar1} shows M\"obius disjointness of zero entropy continuous maps of the interval and (orientation preserving) homeomorphisms of the circle. In~\cite{EI1}, Eisner proposes to study a polynomial version of Sarnak's conjecture (in the minimal case). See also~\cite{Ab-Di,Ab-Ye,Br-Te,Do-Gl,Fa,HLSY}.

\section{Related research: $\mathscr{B}$-free numbers}
\label{czesc6}

\subsection{Introduction}\label{sofm}
\paragraph{Sets of multiples}
We have already seen that some properties of the M\"obius function $\mob$ can be investigated by looking at its square $\mob^2$, i.e.\ the characteristic function of the set of square-free numbers $Q:=\{n\in \Z  : p^2\not\divides n \text{ for all primes }p\}$. A natural generalization comes when we study sets of integers that are not divisible by elements of a given set. Let $\mathscr{B}\subset \N$ and let $\mathcal{M}_\mathscr{B}$ be the corresponding set of multiples, i.e.\ $\mathcal{M}_\mathscr{B}=\bigcup_{b\in\mathscr{B}}b\Z$ and the associated set of \emph{$\mathscr{B}$-free numbers} $\mathcal{F}_\mathscr{B}:=\Z\setminus \mathcal{M}_\mathscr{B}$ (for convenience, we will deal now with subsets of $\Z$ instead of subsets of $\N$ -- the M\"obius function $\mob$ is not defined for negative arguments, but its square has a natural extension to negative integers). By $\eta=\eta_\mathscr{B}$ we will denote the characteristic function of $\mathcal{F}_\mathscr{B}$. It is not hard to show that a symmetric subset $F\subset\Z$ is a $\mathscr{B}$-free set (for some $\mathscr{B}$) if and only if $F$ is closed under taking divisors.

\paragraph{Historical remarks}
Sets of multiples were an object of intensive studies already in the 1930s \cite{bessel1929zahlentheorie, Davenport:1933aa, zbMATH03014412,MR1574879}. The basic motivating example there was the set of \emph{abundant numbers} ($n\in\Z$ is {\em abundant} if $|n|$ is smaller than the sum of its (positive) proper divisors, i.e.\ $|n|<\boldsymbol{\sigma}(|n|)$), see also more recent~\cite{MR3254753,MR3189967,MR2134854} on that subject. Also many natural questions on general $\mathscr{B}$-free sets emerged. Besicovitch~\cite{MR1512943} showed that the asymptotic density of $\mathcal{M}_\mathscr{B}$ may fail to exist.  It turned out that it was more natural to use the notion of logarithmic density (denoted by $\bdelta$) which always exists in this case and equals the lower density. More precisely, we have the following result of Davenport and Erd\"os:
\begin{Th}[\cite{Davenport1936,MR0043835}]\label{da-er}
For any $\mathscr{B}$, the logarithmic density $\bdelta(\cm_\sB)$ of $\cm_\mathscr{B}$ exists. Moreover,
$
\bdelta(\cm_\mathscr{B})=\un{d}(\cm_{\mathscr{B}})=
\lim_{n\to\infty}d(\mathcal{M}_{\{b\in \mathscr{B} : b\leq n\}})$.
\end{Th}
In the so-called {\em Erd\"os case} when $\mathscr{B}$ consists of pairwise coprime elements whose sum of reciprocals converges, the density does exist, cf.~\cite{Ha-Ro} (in particular, $\raz_{\cf_{\mathscr{B}}}$ is rational). We refer the reader to~\cite{Ha-Ro,Ha} for a coherent, self-contained introduction to the theory of sets of multiples from the analytic and probabilistic number theory viewpoint.

\paragraph{Dynamics comes into play}
Sarnak in~\cite{Sa}, suggested to study $\mob^2$ from the dynamical viewpoint and he announced the following results:
\begin{enumerate}[(i)]
\item $\mob^2$ is \emph{generic} for an ergodic $S$-invariant measure $\nu_{\mob^2}$ on $\{0,1\}^\Z$ such that the  measure-theoretical dynamical system $(X_{\mob^2},\nu_{\mob^2},S)$ has zero measure-theoretic entropy;\footnote{This is clearly a refinement of the fact that the asymptotic density of square-free integers exists (it is given by $6/\pi^2=1/\zeta(2)$).  It follows that $\mob^2$ is a completely deterministic point.}\label{sa1}
\item the topological entropy of $(X_{\mob^2},S)$ is equal to $\nicefrac{6}{\pi^2}$;\label{sa2}
\item  $X_{\mob^2}=X_{\{p^2 : p\in\PP \}}$ (see the definition of admissibility below);\label{sa3}
\item $(X_{\mob^2},S)$ is \emph{proximal}.\label{sa4}
\end{enumerate}
This triggered intensive research in analogous direction for dynamical systems given by other $\mathscr{B}$-free sets. In~\cite{Ab-Le-Ru1}, Abdalaoui, Lema\'nczyk and de la Rue developed the necessary tools in the Erd\"os case and covered (i)-(iii) from the above list. Given $\mathscr{B}=\{b_k : k\geq 1\}$, In particular, they defined a function $\varphi\colon G=\prod_{k\geq 1}\Z/b_k\Z\to \{0,1\}^\Z$ given by
$$
\varphi(g)(n)=1 \iff g_k+n \not\equiv 0\bmod b_k \text{ for all }k\geq 1.
$$
Note that $\eta_\sB=\varphi(0)$ and $\varphi$ is the coding of points under the translation by $(1,1,\dots)$ on $G$ with respect to a two-set partition $\{W,W^c\}$, where
\begin{equation}\label{okienko}
W=\{h\in G : h_b\neq 0 \text{ for all }b\in\sB\}.
\end{equation}
This study was continued in a general setting in~\cite{Ba-Ka-Ku-Le} and the first obstacle was that it was no longer clear which subshift to study -- it turned out that the most important role is played by the following three subshifts, which coincide in the Erd\"os case (for the square-free, case see~\cite{MR3430278} by Peckner and for the Erd\"os case, see \cite{Ab-Le-Ru1}):
\begin{itemize}
\item
$X_\eta$ is the closure of the orbit of $\eta_\mathscr{B}$ under $S$ (\emph{$\mathscr{B}$-free subshift}),
\item
$\widetilde{X}_\eta$ is the smallest hereditary subshift containing $X_\eta$ (a subshift $(X,S)$ is {\em hereditary},  whenever $x\in X$ and $y\leq x$ coordinatewise, then $y\in X$),
\item
$X_\mathscr{B}$ is the set of $\mathscr{B}$-{\em admissible} sequences, i.e.\ of $x\in\{0,1\}^{\Z}$ such that, for each $b\in\mathscr{B}$, the support ${\rm supp}\,x:=\{n\in\Z : x(n)=1\}$ of $x$ taken modulo $b$ is a proper subset of $\Z/b\Z$ (\emph{$\mathscr{B}$-admissible subshift}).
\end{itemize}

\begin{Remark}\label{r:hererd} As $X_{\mathscr{B}}$ is hereditary, we have $X_\eta\subset\widetilde{X_\eta}\subset X_{\mathscr{B}}$. In the Erd\"os case, we have $X_\eta=X_{\mathscr{B}}$ \cite{Ab-Le-Ru1} (for the square-free system \cite{Sa}).\end{Remark}

Also the group $G$ turned out to be too large for the studies -- it is natural to consider its closed subgroup
\beq\label{podgrupa}
H:=\overline{\{(n,n,\dots)\in G : n\in\Z\}}.
\eeq
In the Erd\"os case we have $H=G$. Certain special cases more general than the Erd\"os one were considered in~\cite{Ba-Ka-Ku-Le}:
\begin{itemize}
\item
we say that $\mathscr{B}$ is {\em taut} whenever $\bdelta(\mathcal{F}_\mathscr{B})<\bdelta(\mathcal{F}_{\mathscr{B}\setminus \{b\}})$ for each $b\in\mathscr{B}$;
\item
we say that $\mathscr{B}$ has {\em light tails}, i.e.\ $\ov{d}(\sum_{b>K}b\Z)\to 0$ as $K\to \infty$.
\end{itemize}
Following~\cite{Ha}, we also say that $\mathscr{B}$ is \emph{Besicovitch} if $d(\cm_{\mathscr{B}})$ exists (equivalently, $d(\cf_\mathscr{B})$ exists). A set $\mathscr{B}\subset \N\setminus \{1\}$ is called \emph{Behrend} if $\boldsymbol{\delta}(\cm_{\mathscr{B}})=1$. Throughout, we will tacitly assume that $\sB$ is {\em primitive}, i.e.\ does not contain $b\neq b'$ with $b\divides b'$. Recall that $\mathscr{B}$ is taut if and only if $\sB$ does not contain $d\mathscr{A}$, where $\mathscr{A}\subset \N\setminus\{1\}$ is Behrend and $d\in\N$.

\paragraph{Further generalizations}\label{fu-ge}
Several further generalizations of $\mathscr{B}$-free integers were discussed in the literature from the dynamical viewpoint. Let us briefly recall them here:
\begin{itemize}
\item
Pleasants and Huck~\cite{Pe-Hu} considered {\em $k$-free lattice points} $\mathcal{F}_k=\mathcal{F}_k(\Lambda):=\Lambda \setminus \bigcup_{p\in\mathbb{P}} p^k \Lambda$, where
$\Lambda$ is a lattice in $\R^d$ (the corresponding dynamical system given by the orbit closure of $\raz_{\mathcal{F}_k}\in\{0,1\}^\Lambda$ under the multidimensional shift).
\item
Cellarosi and Vinogradov~\cite{MR3296562} considered {\em $k$-free integers in number fields} $\mathcal{F}_k=\mathcal{F}_k(\mathcal{O}_K):=\mathcal{O}_K \setminus \bigcup_{\mathfrak{p}\in\mathfrak{P}}\mathfrak{p}^k$. Here $K$ is a finite extension of $\Q$, $\mathcal{O}_K\subset K$ is the ring of integers, $\mathfrak{P}$ stands for the family of all prime ideals in $\mathcal{O}_K$ and $\mathfrak{p}^k$ stands for ${\mathfrak{p}\dots\mathfrak{p}}$ ($\mathfrak{p}$ is taken $k$ times).
\item
Baake and Huck in their survey~\cite{Baake:2015aa} considered {\em $\mathscr{B}$-free lattice points} $\mathcal{F}_{\mathscr{B}}=\mathcal{F}_{\mathscr{B}}(\Lambda):=\Lambda \setminus \bigcup_{b\in\mathscr{B}}b\Lambda$. Here $\Lambda$ is a lattice in $\R^d$ and $\mathscr{B}\subseteq \N\setminus\{1\}$ is an infinite pairwise coprime set with $\sum_{b\in\mathscr{B}}1/{b^d}<\infty$.
\item
Finally, one can consider $\mathfrak{B}$-free integers $\mathcal{F}_\mathfrak{B}$ in number fields as suggested in~\cite{Baake:2015aa}. Here $K$ is a finite extension of $\Q$, $\mathcal{O}_K\subset K$ is the ring of integers and $\mathfrak{B}$ is a family of pairwise coprime ideals in $\mathcal{O}_K$ such that the sum of reciprocals of their norms converges.
\end{itemize}

We will recall some of the main results from the above papers in the relevant sections below.

\subsection{Invariant measures and entropy}\label{invme}


\paragraph{Mirsky measure}
Cellarosi and Sinai proved \eqref{sa1} in~\cite{Ce-Si}: they showed that $\nu_{\mob^2}$ is generic for a shift-invariant measure $\nu_{\mob^2}$ on $\{0,1\}^\Z$, and that $(X_{\mob^2},\nu_{\mob^2},S)$ is isomorphic to a rotation on the compact Abelian group $\prod_{p\in\mathbb{P}}\Z/p^2\Z$. In particular, $(X_{\mob^2},\nu_{\mob^2},S)$ is of zero Kolmogorov entropy.\footnote{The frequencies of blocks on $\mob^2$ were first studied by Mirsky~\cite{MR0021566,Mi} and that is why we  refer to $\nu_{\mob^2}$ (and the analogous measure in case of general $\mathscr{B}$-free systems) as the {\em Mirsky measure}.
} In case of $k$-free lattice points and $k$-free integers in number fields an analogous result can be found in~\cite{Pe-Hu} and~\cite{MR3296562}, respectively and for $\mathscr{B}$-free lattice points it was announced in~\cite{Baake:2015aa}. Recently, Huck~\cite{Huck-daer} showed that in case of $\mathfrak{B}$-free integers in number fields, the logarithmic density of $\mathcal{F}_\mathfrak{B}$ always exists and equals the lower density, thus extending Theorem~\ref{da-er} in the (1-dimensional) Erd\"os case.

Since $\mathcal{F}_\mathscr{B}$ may fail to have asymptotic density, the more $\eta$ may fail to be a generic point. However (Proposition E in~\cite{Ba-Ka-Ku-Le}), for any $\mathscr{B}\subset \N$, $\eta$ is always a quasi-generic point for a natural ergodic $S$-invariant measure $\nu_\eta$ on $\{0,1\}^\Z$ (the relevant Mirsky measure). Moreover, $\mathscr{B}$ is Besicovitch if and only if $\eta$ is generic for $\nu_\eta$. Now, if we additionally assume that $\mathscr{B}$ is taut, then $(X_\eta,\nu_\eta,S)$ is isomorphic to an ergodic rotation on a compact metric group (Theorem F in~\cite{Ba-Ka-Ku-Le}).\footnote{More precisely, it is isomorphic to $(H,\PP,T)$, where $H$ is the closure of ${\{(n\text{ mod }b_k)_{k\geq1} : n\in \Z\}}$ in $\prod_{k\geq 1}\Z/b_k\Z$ and $Tg=g+(1,1,\dots)$, cf.~\eqref{podgrupa}.} In particular, $(X_\eta,\nu_\eta,S)$ has zero entropy.

Finally, for a generalization to so-called weak model sets, see \cite{Ba-Hu-St}, and for some results related to the distribution of $\mathscr{B}$-free integers, see~\cite{Av,Av-Ce-Si}.

\paragraph{Entropy}
The topological entropy of $X_{\mob^2}$ is positive and equals $6/\pi^2=\prod_{p\in\PP}(1-1/p^2)=d(\mathcal{F}_\sB)$ for $\sB=\{p^2 : p\in\PP\}$, see~\cite{MR3430278}. This extends to the Erd\"os case, where the topological entropy of $X_\eta=\widetilde{X}_\eta=X_\mathscr{B}$ equals $\prod_{b\in\mathscr{B}}(1-1/b)=d(\mathcal{F}_\mathscr{B})$, see~\cite{Ab-Le-Ru1}. In the general case of $\mathscr{B}$-free systems, we have $h_{top}(\widetilde{X}_\eta,S)=h_{top}(X_\mathscr{B},S)=\bdelta(\cf_{\mathscr{B}})$ (see Theorem K in~\cite{Ba-Ka-Ku-Le}). The formula for the topological entropy of $k$-free lattice points is provided in~\cite{Pe-Hu}.

In view of the variational principle, the positivity of the topological entropy evokes two problems: whether the system under consideration is intrinsically ergodic (i.e. whether there is a unique measure of maximal entropy) and to describe the set of all invariant measures. We address them next.

\paragraph{Maximal entropy measure}
In the square-free case, the intrinsic ergodicity is proved by Peckner in~\cite{MR3430278}. This extends to the Erd\"os case, see~\cite{Ku-Le-We} by Ku\l{}aga-Przymus, Lema\'nczyk and Weiss. Finally, for any $\mathscr{B}\subset \N$, the subshift $(\widetilde{X}_\eta,S)$ is intrinsically ergodic, see Theorem J in~\cite{Ba-Ka-Ku-Le}. In particular, if $\mathscr{B}$ has light tails and contains an infinite pairwise coprime subset then $(X_\mathscr{B},S)$ is intrinsically ergodic.

\paragraph{All invariant measures}
Notice that for each $\mathscr{B}$, the map $M\colon X_\eta\times \{0,1\}^\Z\to \widetilde{X}_\eta$ given by the coordinatewise multiplication of sequences is well-defined and each $S\times S$-invariant measure $\rho$ on $X_\eta\times \{0,1\}^\Z$ yields an $S$-invariant measure on $\widetilde{X}_\eta$. In particular, this applies to those $\rho$ whose projection on the first coordinate is $\nu_\eta$. It turns out that the converse is also true: for any $S$-invariant measure $\nu$ on $\widetilde{X}_\eta$ there exists an $S\times S$-invariant measure $\rho$ on $X_\eta\times \{0,1\}^\Z$ whose projection on the first coordinate is $\nu_\eta$ and such that $M_\ast(\rho)=\nu$. For the Erd\"os case see~\cite{Ku-Le-We} and for general $\mathscr{B}$-free systems, see Theorem I in~\cite{Ba-Ka-Ku-Le} (for further generalizations of $\mathscr{B}$-free systems listed before (see page \pageref{fu-ge}) no analogous description of the set of all invariant measures is known).


It turns out that a special role is played by $\mathscr{B}$ that are taut. We have the following: for any $\mathscr{B}$, there exists a unique taut set $\mathscr{B}'\subset \N$ such that $\cf_{\sB'}\subset \cf_{\sB}$, $\widetilde{X}_{\eta'}\subset \widetilde{X}_\eta$ and all $S$-invariant measures on $\widetilde{X}_\eta$ are in fact supported on $\widetilde{X}_{\eta'}$ (Theorem C in~\cite{Ba-Ka-Ku-Le}).

More subtle properties of the simplex of invariant measures of the $\mathscr{B}$-shift have been studied in~\cite{Ku-Le-We1} by Ku\l{}aga-Przymus, Lema\'nczyk and Weiss -- it was shown that in the positive entropy case the simplex of $S$-invariant measures on $\widetilde{X}_\eta$  is {\em Poulsen}, i.e.\ the ergodic measures are dense. In particular, if we additionally know that $X_\eta$ is hereditary (and has positive entropy), then its simplex of invariant measures is Poulsen. However, this is no longer true for a general (not necessarily $\mathscr{B}$-free) hereditary system. On the other hand, Konieczny, Kupsa and Kwietniak~\cite{Ko-Ku-Kw} showed that the set of ergodic invariant measures of a hereditary shift is always arcwise connected (when endowed with the $d$-bar metric).

\subsection{Topological results}
A lot can be said about the topological properties of $(X_\eta,S)$. E.g.\ for any $\mathscr{B}\subset \N$ the subshift $X_\eta$ has a unique minimal subset that is the orbit closure of a Toeplitz system (Theorem A in~\cite{Ba-Ka-Ku-Le}).  In particular, $X_\eta$ is minimal if and only if $X_\eta$ is a Toeplitz system.\footnote{This has been recently improved in~\cite{Kasjan:2017aa}  and by A. Bartnicka: $X_\eta$ is minimal if and only if $\eta$ is Toeplitz.}  In fact, $\eta$ itself can be a Toeplitz sequence (see Example~3.1 in~\cite{Ba-Ka-Ku-Le}) and it was shown in~\cite{Kasjan:2017aa} that $\eta$ is a Toeplitz sequence different from $\dots0.00\dots$ if and only if $\mathscr{B}$ does not contain a subset of the form $d\mathscr{A}$, where $d\in\N$ and $\mathscr{A}\subset \N\setminus\{1\}$ is infinite and pairwise coprime. Moreover, if $\eta$ is Toeplitz then $\sB$ is necessarily taut~\cite{Kasjan:2017aa}.

On the other hand, the proximality of $X_\eta$ is equivalent to $\{\dots0.00\dots\}$ being the unique minimal subset of $X_\eta$. Moreover, $X_\eta$ is proximal if and only if $\mathscr{B}$ contains an infinite pairwise coprime subset (Theorem B in~\cite{Ba-Ka-Ku-Le}).



Some of the properties of the $\mathscr{B}$-free subshift $X_\eta$
 can be characterized via properties of a set $W$ called the {\em window}: $W=\{h\in H : h_b\neq 0 \text{ for all }b\in\sB\}$, cf.\ \eqref{okienko}. This name has its origins in the theory of weak model sets (for more details see~\cite{Ba-Gr}); $\mathcal{F}_\mathscr{B}$ is an example of such a set. Again a special role is played by sets $\mathscr{B}$ that are taut. In~\cite{Kasjan:2017aa}, Kasjan, Keller and Lema\'nczyk show the following:
\begin{itemize}
\item
$\mathscr{B}$ is taut if and only if $W$ is Haar regular, i.e.\ the topological support of Haar measure restricted to $W$ is the whole $W$;
\item
if $\mathscr{B}$ is primitive then $X_\eta$ is a Toeplitz system if and only if $W$ is topologically regular;
\item
$X_\eta$ is proximal if and only if $W$ has empty interior.
\end{itemize}
In~\cite{Kasjan:2017aa} there is also a detailed description of the maximal equicontinuous factor of $X_\eta$ (with no extra assumptions on $\mathscr{B}$). See also~\cite{Ke-Ri}.

Clearly, if $X_\eta$ is hereditary, i.e.\ $X_\eta=\widetilde{X}_\eta$ then $(\dots0.00\dots)\in X_\eta$ and hence $X_\eta$ is proximal. If we assume that $\mathscr{B}$ is taut then the converse is true: proximality yields heredity (Theorem D in~\cite{Ba-Ka-Ku-Le}). However, $\widetilde{X}_\eta=X_\sB$ may fail to hold, even under quite strong assumptions on $\mathscr{B}$. Indeed, the set of abundant numbers $\mathbb{A}$ is the corresponding set of multiples $\mathcal{M}_\mathscr{B}$ for a certain set $\mathscr{B}$ with the property that $\sum_{b\in\mathscr{B}}1/b<\infty$. Here, $\widetilde{X}_\eta\neq X_\mathscr{B}$, see Section 11 in~\cite{Ba-Ka-Ku-Le}.

More subtle results on heredity were recently obtained by Keller in~\cite{Keller:2017aa}. He shows that whenever $X_\eta$ is proximal then it is contained in a slightly larger subshift that is hereditary (there is no need to make extra assumptions on $\mathscr{B}$). He also generalizes the concept of heredity to the non-proximal case.

It is also interesting to ask about the (invertible) centralizer of $(S,X_\eta)$. In the Erd\"os case it was proved by Mentzen\footnote{Mentzen's result is extended in \cite{Ba-Hu-Le} to every hereditary $\mathscr{B}$-free subshift.} in~\cite{MR3612882} that the group of homeomorphisms commuting with the shift $(S,X_\eta)$ consists only of the powers of $S$. In case of some Toeplitz $\mathscr{B}$-free systems an analogous result was proved by Bartnicka in~\cite{Ba-top}.

\begin{Question}\label{centr}
Is the invertible centralizer trivial for each $\mathscr{B}$-free subshift?
\end{Question}

\subsection{Ergodic Ramsey theory}

We will now see some connections of the theory of $\mathscr{B}$-free sets with the theory uniform distribution and ergodic Ramsey theory.
\paragraph{Polynomial recurrence and divisibility}
Recall that Szemer{\'e}di showed~\cite{MR0369312} that any set $S\subset \N$ with positive upper density contains arbitrarily long arithmetic progressions and Furstenberg~\cite{MR0498471,Fu1} introduced an ergodic approach to this result that proved very fruitful from the point of view of various generalizations. In particular, it allowed one to prove the following:
for any probability space $\xbm$, invertible measure preserving transformation $T\colon X\to X$,
$A\in\cb$ with $\mu(A)>0$
and any polynomials $p_i\in\Q[t]$ satisfying
$p_i(\Z)\subset\Z$ and $p_i(0)=0$, $1\leq i\leq \ell$,
there exists arbitrarily large $n\in\N$
such that
\beq\label{gwiazdka}\mu\big(A\cap T^{-p_1(n)}A\cap\ldots\cap T^{-p_\ell(n)}A\big)>0.\eeq In fact, we have
$$
\lim_{N\to\infty}\frac{1}{N}\sum_{n=1}^N
\mu\Big(A\cap T^{-p_1(n)}A\cap\ldots\cap T^{-p_\ell(n)}A\Big)>0
$$
\cite{MR1325795,MR2191208,MR2151605}. One can now restrict attention to a specific subset $R$ of $n\in\N$ for which we ask whether~\eqref{gwiazdka} holds or even demand
\begin{equation}\label{eq:dem}
\lim_{N\to\infty}\frac{1}{|R\cap[1,N]|}\sum_{n=1}^N\raz_R(n)\mu\Big(A\cap T^{-p_1(n)}A\cap\ldots\cap T^{-p_\ell(n)}A\Big)>0.
\end{equation}
If~\eqref{eq:dem} holds for any invertible measure preserving system $\xbmt$,
$A\in\mathcal{B}$ with $\mu(A)>0$,
$\ell\in\N$ and any polynomials
$p_i\in\Q[t]$, $i=1,\ldots,\ell$, with $p_i(\Z)\subset\Z$ and $p_i(0)=0$ for all $i\in\{1,\ldots,\ell\}$, we say (cf. {\cite[Definition 1.5]{MR1412598}}) that $R\subset \N$ is \emph{averaging set of polynomial multiple
recurrence}.
If $\ell=1$, we speak of an \emph{averaging set of polynomial single recurrence}.

We will be interested in polynomial recurrence for $\mathscr{B}$-free sets. Before we get there, let us direct our attention to so-called rational sets. Recall that $R\subset \N$ is rational if it can be approximated in density by finite unions of arithmetic progressions, cf.\ footnote~\ref{f:rational}. Note that the rationality of $\mathcal{F}_\sB$ is equivalent to $\mathscr{B}$ being Besicovitch. An easy necessary condition for $R\subset \N$ to be an averaging set of polynomial recurrence is that the density of $R\cap u\N$ exists and is positive for any $u\in\N$ (indeed, otherwise consider the cyclic rotation on $\Z/u\Z$ to see that even usual recurrence fails).  If the latter holds, we will say that $R$ is {\em divisible}. It turns out that in case of rational sets, divisibility is not only necessary but also sufficient for polynomial recurrence. More precisely, we have the following:

\begin{Th}[\cite{Be-Ku-Le-Ri}] Let $R\subset\N$ be rational and of positive density. The following conditions are equivalent:
\begin{enumerate}[(a)]
\item\label{itm:pmrar-a}
$R$ is divisible.
\item\label{itm:pmrar-b}
$R$ is an averaging set of polynomial single recurrence.
\item\label{itm:pmrar-c}
$R$ is an averaging set of polynomial multiple recurrence.
\end{enumerate}\end{Th}
Recall that it was proved in~\cite{MR1954690} that every self-shift $Q-r$, $r\in Q$, of the set of square-free numbers $Q$ is divisible and these are the only divisible shifts of $Q$. For general $\mathscr{B}$-free sets the situation is more complicated and we have the following result:

\begin{Th}[\cite{Be-Ku-Le-Ri}] Given $\mathscr{B}\subset \N$ that is Besicovitch, there exists a set $D\subset \cf_\sB$ with
$d(\cf_\sB\setminus D)=0$
such that the set $\cf_\sB-r$
is an averaging set of polynomial multiple recurrence if and only if $r\in D$. Moreover, $D=\cf_\sB$ if and only if the set $\sB$ is {taut}.
\end{Th}
This can be generalized to $\mathscr{B}$ that are not Besicovitch by considering divisibility and recurrence along a certain subsequence $(N_k)_{k\geq 1}$. As a combinatorial application, one obtains in~\cite{Be-Ku-Le-Ri} the following result: Suppose that $(N_k)_{k\geq 1}$ is such that the density of $\mathcal{F}_\sB$ along $(N_k)_{k\geq 1}$ exists and is positive. Then there exists $D\subset \mathcal{F}_\mathscr{B}$ which equals $\mathcal{F}_\mathscr{B}$ up to a set of zero density along $(N_k)_{k\geq 1}$ such that for all $r\in D$ and for all $E\subset \N$ with positive upper density, for any polynomials $p_i\in\Q[t]$, $i=1,\ldots,\ell$,
which satisfy
$p_i(\Z)\subset\Z$ and $p_i(0)=0$, for all $1\leq i\leq \ell$, there exists $\beta>0$ such that the set
$$
\left\{n\in \cf_{\sB}-r:\overline{d}\Big(
E\cap (E-p_1(n))\cap \ldots\cap(E-p_\ell(n))
\Big)>\beta \right\}
$$
has positive lower density along $(N_k)_{k\geq 1}$. If, additionally, $\sB$ is taut then one can take $D=\cf_{\sB}$.

Results of similar flavor as above have been also obtained in~\cite{Be-Ku-Le-Ri1a} in the context of level sets of multiplicative functions. In particular, if $E$ is a level set of a multiplicative function and has positive density then every self-shift of $E$ is an averaging set of polynomial multiple recurrence (Corollary~C in~\cite{Be-Ku-Le-Ri1a}). The key tool here is~\eqref{relunif} that provides an important link between level sets of multiplicative functions and rational sets. See also  \cite{Be-Ku-Le-Ri1b}.


\section*{Acknowledgments}
The research resulting in this survey was carried out during the Research in Pairs Program of CIRM, Luminy, France, 15-19.05.2017. J.\ Ku\l{}aga-Przymus and M.\ Lema\'nczyk also acknowledge the support of Narodowe Centrum Nauki grant UMO-2014/15/B/ST1/03736. J.\ Ku\l{}aga-Przymus was also supported by the European Research Council (ERC) under the European Union’s Horizon 2020 research and innovation programme (grant agreement No 647133 (ICHAOS)).

The authors special thanks go to N.\ Frantzikinakis and P.\ Sarnak for a careful reading of the manuscript,  numerous remarks and suggestions to improve presentation. We  also thank M.\ Baake, V. Bergelson, B.\ Green, D.\ Kwietniak, C.\ Mauduit and M.\ Radziwi\l\l \ for some useful comments on the subject.

\small

\bibliography{survey-sarnak}

\bigskip
\footnotesize

\noindent
Sebastien Ferenczi\\
\textsc{Aix Marseille Universit\'e, CNRS, Centrale Marseille, Institut de Math\' ematiques de Marseille, I2M -- UMR 7373\\13453 Marseille\\France\\}
\noindent
\texttt{ssferenczi@gmail.com}

\medskip

\noindent
Joanna Ku\l aga-Przymus\\
\textsc{Aix-Marseille Universit\'e, CNRS, Centrale Marseille, Institut de Math\'ematiques de Marseille, I2M -- UMR 7373\\ 13453, Marseille, France}\\
\textsc{Faculty of Mathematics and Computer Science, Nicolaus Copernicus University, Chopina 12/18, 87-100 Toru\'{n}, Poland}\par\nopagebreak
\noindent
\texttt{joanna.kulaga@gmail.com}

\medskip

\noindent
Mariusz Lema\'nczyk\\
\textsc{Faculty of Mathematics and Computer Science, Nicolaus Copernicus University, Chopina 12/18, 87-100 Toru\'{n}, Poland}\par\nopagebreak
\noindent
\texttt{mlem@mat.umk.pl}

\end{document}